\documentclass[final,leqno]{siamltex}

\usepackage{amsmath}
\usepackage{amssymb}
\usepackage{graphicx}
\usepackage[notcite,notref]{showkeys}
\usepackage{tikz}
\usepackage{float}
\usepackage{bm}
\newcounter{dummy} \numberwithin{dummy}{section}
\newtheorem{defi}[dummy]{\bf Definition}

\newtheorem{algorithm}{\bf Weak Galerkin Algorithm}
\newtheorem{remark}{\bf Remark}

\usepackage[outline]{contour}

\newcommand\myLambda{\mathrm{I} \hspace{-0.4ex} \Lambda}

\newcommand{\bQ}{{\bf Q}}
\newcommand{\bu}{{\bf u}}
\newcommand{\bq}{{\bf q}}
\newcommand{\bw}{{\bf w}}
\newcommand{\bx}{{\bf x}}

\newcommand{\be}{{\bf e}}
\newcommand{\bv}{{\bf v}}
\newcommand{\bH}{{\textbf{\textit{H}}}}
\newcommand{\bpsi}{{\boldsymbol\psi}}

\def\Q{{\mathbb Q}}
\def\T{{\mathcal T}}

\def\pT{{\partial T}}
\def\l{{\langle}}
\def\r{{\rangle}}

\def\T{{\mathcal T}}

\def\bbf{{\bf f}}

\def\bn{{\bf n}}
\def\bq{{\bf q}}

\def\ljump{{[\![}}
\def\rjump{{]\!]}}

\def\3bar{{|\hspace{-.02in}|\hspace{-.02in}|}}

  \def\b#1{\mathbf{#1}} 
\def\a#1{\begin{align*}#1\end{align*}}

\title{A stabilizer free, pressure robust, and superconvergence weak Galerkin finite element method for the Stokes Equations on polytopal mesh }
\author{Lin Mu\thanks{Department of Mathematics, University of Georgia, Athens, GA 30602 (linmu@uga.edu).}
\and Xiu Ye\thanks{Department of
Mathematics, University of Arkansas at Little Rock, Little Rock, AR
72204 (xxye@ualr.edu). This research was supported in part by
National Science Foundation Grant DMS-1620016.}
\and
Shangyou Zhang\thanks{Department of
Mathematical Sciences, University of Delaware, Newark, DE 19716 (szhang@udel.edu).}
}
\begin{document}

\maketitle

\begin{abstract}
In this paper, we propose a new stabilizer free and pressure robust WG method for the Stokes equations with super-convergence on polytopal mesh in the primary velocity-pressure formulation. 
Convergence rates with one order higher than the optimal-order for velocity in both  energy-norm and the $L^2$-norm and for pressure in $L^2$-norm are proved in our proposed scheme. The $H$(div)-preserving operator has been constructed based on the polygonal mesh for arbitrary polynomial degrees and employed in the body source assembling to break the locking phenomenon induced by poor mass conservation in the classical discretization. Moreover, the velocity error in our proposed scheme is proved to be independent of pressure and thus confirm the pressure-robustness. For Stokes simulation, our proposed scheme only modifies the body source assembling but remains the same stiffness matrix.
Four numerical experiments are conducted to validate the convergence results and robustness.
\end{abstract}


\pagestyle{myheadings}

\section{Introduction}

In this paper, we consider the following viscosity dependent Stokes equations: 
find velocity $\bu:\Omega\to\mathbb{R}^d$ (d = 2 or 3) and pressure: $p:\Omega\to\mathbb{R}$ such that:
\begin{eqnarray}
-\nu\Delta\bu + \nabla p &=& {\bf f},\text{ in }\Omega,\label{eq:pde-1}\\
\nabla\cdot\bu &=& 0,\text{ in }\Omega,\label{eq:pde-2}\\
\bu&=& 0,\text{ on }\partial\Omega,\label{eq:bc}
\end{eqnarray}
where $\nu>0$ is a constant viscosity parameter and ${\bf f}\in{\bf L}^d(\Omega)$ is a given vector field.

The Stokes equation has been widely used in realistic applications. However the classical finite element methods for incompressible Stokes problem are usually not pressure-robust: their velocity error is pressure dependent, shown as follows,
\begin{eqnarray*}
\|\nabla(\bu-\bu_h)\|\le C\inf_{\bw\in V_h}\|\nabla(\bu-\bw)\|+\frac{1}{\nu}\inf_{q\in Q_h}\|p-q\|.
\end{eqnarray*}
Here $V_h$ and $Q_h$ denote the inf-sup stable finite element spaces for velocity and the pressure. We observe that there is a pressure dependent term that can be relatively large for small viscosity $\nu$. This may cause the locking phenomenon: as small parameters in $\nu\ll 1$ will produce large error in velocity; bad approximation in pressure may affect the velocity simulation. To reduce these effects, one would have to increase scheme order and also assume that the pressure is smooth enough.

To remove the pressure dependency in velocity simulation, one has to employ the divergence-free finite element scheme.
Nowadays, many divergence-free elements have been developed for two dimensional problems \cite{FalkNeilan,ScottVogelius} and three dimensional problems \cite{Zhang2011}. 
The enrichment of the H(div; $\Omega$)-conforming elements locally with divergence-free rational
shape-functions has been proposed in \cite{GuzmanNeilan2011,GuzmanNeilan2013} for two and three dimensional problems. Then by modifying the variational formulation and introduced tangential penalty and thus divergence free schemes are obtained by Cockburn \cite{CockburnKanschatSchotzau} and Wang \cite{WangYe2016}. Chen proposed to introduce a discrete dual curl operator \cite{ChenWangZhong} to achieve divergence free MAC scheme on triangular grids.
There are other approaches for dealing with such difficulties  by adding grad-div stabilization \cite{
OlshanskiiOlshanskii,Olshanskii2009,JenkinsVolkerLinkeRebholtz}. 
In addition, an alternative method to get pressure independent error estimate is introduced recently in \cite{Linke2012,Linke2014} by employing a divergence preserving velocity reconstruction operator. 
In \cite{Linke2012}, a velocity reconstruction is presented to map discrete divergence free test functions onto exactly divergence free test functions which is applied only on the right hand side for the Stokes equations. This approach was performed for several elements including discontinuous pressure elements \cite{JohnLinkeMerdonNeilan,Linke2014,Linke2012,LinkeMatthiesTobiska,BrenneckeLinkeMerdon} and continuous pressure elements \cite{Lederer,LedererSchoberl}. 
With the help of velocity reconstruction, Linke proposed the pressure-robust schemes leading to a pressure independent velocity error estimate
\begin{eqnarray*}
\|\nabla(\bu-\bu_h)\|\le C\inf_{\bw\in V_h}\|\nabla(\bu-\bw)\|.
\end{eqnarray*}
Such schemes are developed on modifying for the existing inf-sup stable schemes to achieve the robustness without compromising the computational accuracy.

{Recently, the approximations with flexibility on the polygonal/polyhedral meshes are proposed by many researchers. The features, from handling polygonal meshes and meshes with hanging nodes, include not only weakens the complexity in meshing for the irregular domain but also enhances the efficiency in mesh refining/coarsening in adaptive finite element strategies. Due to features of polygonal finite element framework, there are increasing demands for extending existing numerical schemes to polygonal meshes. 
} Virtual element methods (VEM)\cite{BeiVei2013,BeiVei2013-2}, mimetic finite difference methods \cite{LipnikovManziniShashkov}, hybrid high order methods (HHO)\cite{PietroErn2015}, hybridized discontinuous Galerkin methods (HDG)\cite{CockburnDipietroErn}, generalized barycentric coordinates method \cite{ChenWang}, and weak Galerkin finite element methods (WG)\cite{wy} have been developed on these types of meshes for various types of partial differential equations. 

In this paper, we focus on developing a new pressure-robust numerical discretization for viscosity dependent Stokes equations on the polygonal meshes. To the best of our knowledge, there are few previous work on the pressure-robust incompressible solver on the polygonal meshes. On triangles/rectangles/tetrahedrons/cubics, such a reconstruction can be defined as to projection into a standard Raviart Thomas or Brezzi Douglas Marini spaces\cite{BBF2013}. However, on the polygons, one has to employ the rational bases in the $\mathbb{CW}_0$\cite{ChenWang} space for constructing the reconstruction operator. Such projection techinique is limited to the lowest order scheme with $k=0$ and convex partitions in the meshes\cite{mu}. Another recent work regarding pressure-robust scheme developed by Frerichs\cite{FM2020} was to explore the sub-triangulation and solve the small local Dirichlet boundary value problems for each polygon in the virtual element framework. 

In this paper, we propose another $H$(div)-preserving space and propose new projection technique for the velocity reconstruction operator on the polygonal cell and then design the novel pressure-robust based on the weak Galerkin finite element method. The operator is constructed by piecewise RT functions and enforcing extra conditions as (\ref{eq:cond-1})-(\ref{eq:cond-5}).
The contribution in this work is to show that uniform pressure-robustness, meaning on any mesh, can be attained by the proposed $H$(div,$\Omega$)-conforming projection $\Pi_h$ that preserves the divergence of the test functions $\bv\in V_h$. Comparing to the existing WG scheme for Stokes equation, our scheme only modifies the body force assembling but remains the same stiffness matrix, and thus achieves the divergence preserving by minimal effort.   {The techniques designed in this paper can also be integrated into other polygonal finite element methods including HHO, HDG and VEM.}


This paper is organized as follows. 
Brief review regarding basis functions on the polygonal meshes will be presented in Section~\ref{Sect:Preliminary}.
The discretization is developed in Section~\ref{Sect:Scheme}. Our main results regarding error estimates are stated in Section~\ref{Sect:MainResults}.  Section~\ref{Section:Implementation} contributes to demonstrate the numerical implementation for velocity reconstruction operator. Section~\ref{Section:numerical-experiments} presents several numerical examples for Stokes equations. Finally, this paper is summarized with concluding remarks in Section~\ref{Sect:Conclusion}.

\section{Preliminary}\label{Sect:Preliminary}
This section recalls the needed notations, defines the finite element spaces, and introduces the setup of the weak Galerkin finite element methods for the Stokes problem and the velocity reconstruction operator.
\subsection{Finite Element Space}
Let ${\cal T}_h$ be a partition of the domain $\Omega$ consisting of
polygons in two dimension or polyhedra in three dimension satisfying
a set of conditions specified in \cite{wymix}. Denote by ${\cal E}_h$
the set of all edges or flat faces in ${\cal T}_h$, and let ${\cal
E}_h^0={\cal E}_h\backslash\partial\Omega$ be the set of all
interior edges or flat faces. For every element $T\in \T_h$, we
denote by $h_T$ its diameter and mesh size $h=\max_{T\in\T_h} h_T$
for ${\cal T}_h$. Let $P_k(T)$ consist all the polynomials on $T$ with degree no greater than $k$.

For $k\ge 0$ and given $\T_h$, for $T\in\T_h$ define two finite element spaces  for velocity
\begin{eqnarray}
V_h &=&\left\{ \bv=\{\bv_0,\bv_b\}:\ \bv_0|_{T}\in [P_{k}(T)]^d,\;\bv_b|_e\in [P_{k}(e)]^d,e\subset\pT \right\}.\label{vh}
\end{eqnarray}
and for pressure
\begin{equation}
W_h =\left\{w\in L_0^2(\Omega): \ w|_T\in P_{k}(T)\right\}.\label{wh}
\end{equation}
It is noted that the $\bv_b$ has a single value on the faces of triangulation.
Let $V_h^0$ be a subspace of $V_h$ consisting of functions with vanishing boundary value.

The space $H(\operatorname{div};\Omega)$ is defined as the set of vector-valued functions on $\Omega$ which,
together with their divergence, are square integrable; i.e.,
\[
H(\operatorname{div}; \Omega)=\left\{ \bv\in [L^2(\Omega)]^d:\; \nabla\cdot\bv \in L^2(\Omega)\right\}.
\]
In order to introduce the local finite element space to mimic the $H(\operatorname{div}; \Omega)$. For any $T\in\T_h$, it can be divided in to a set of disjoint triangles $T_i$ with $T=\cup T_i$. We define a local space $\Lambda_{k}(T)$(\cite{sf-wg-part-II})
\begin{eqnarray}
\Lambda_{k}(T)=\{\bv\in H(\operatorname{div},T):\ \bv|_{T_i}\in RT_{k}(T_i),\;\;\nabla\cdot\bv\in P_{k}(T)\},\label{vk}
\end{eqnarray}
and a global space $\tilde{V}_h$ as
\begin{eqnarray}
\tilde{V}_h=\{\bv\in H(\operatorname{div},\Omega):\ \bv|_T\in \Lambda_{k}(T),\;T\in\T_h\}.\label{tv}
\end{eqnarray}

  Then we define a space $\myLambda_h(T)$ for the approximation of weak gradient on each element $T$  as
\begin{eqnarray}
\myLambda_{k}(T)=\{\bpsi\in [H(\operatorname{div};T)]^d:&&\ \bpsi|_{T_i}\in [RT_{k}(T_i)]^{d},\;\;\nabla\cdot\bpsi\in [P_k(T)]^d,\label{lambda}\\
&&\bpsi\cdot\bn|_e\in [P_{k}(e)]^d,\;e\subset\pT\},\nonumber
\end{eqnarray}
where $RT_k(T_i)$ is the usual Raviart-Thomas element \cite{bf} of order $k$.

\begin{definition}
For a function $\bv\in V_h$, its weak gradient $\nabla_w\bv$ is a piecewise polynomial such that $\nabla_w\bv|_T \in \myLambda_k(T)$  and  satisfies the following equation,
\begin{equation}\label{wg}
(\nabla_w\bv,\  \bm{\tau})_T = -(\bv_0,\  \nabla\cdot \bm{\tau})_T+
\l\bv_b, \ \tau\cdot\bn \r_\pT\quad\forall\bm{\tau}\in \myLambda_k(T).
\end{equation}
\end{definition}
\begin{defi}
For a function $\bv\in V_h$, its weak divergence $\nabla_w\cdot\bv$ is a piecewise polynomial such that $\nabla_w\cdot\bv|_T \in P_{k}(T)$  and  satisfies the following equation,
\begin{equation}\label{wd}
(\nabla_w\cdot\bv,\  \tau)_T = -(\bv_0,\  \nabla\tau)_T+
\l\bv_b\cdot\bn, \ \tau \r_\pT\quad\forall\tau\in P_{k}(T).
\end{equation}
\end{defi}

\begin{lemma}
For $\bm{\tau}\in [H(\operatorname{div};\Omega)]^d$, there exists a projection $\pi_h$ with $\pi_h\bm{\tau}\in [H(\operatorname{div};\Omega)]^d$ satisfying $\pi_h\bm{\tau}|_T\in \myLambda_k(T)$ and 
\begin{eqnarray}
(\pi_h\bm{\tau},\;\sigma)_T&=&(\bm{\tau},\;\sigma)_T  \quad \forall \sigma\in [P_{k-1}(T)]^{d\times d},\label{P1}\\
\l \pi_h\bm{\tau}\cdot\bn,\; \bq\r_e &=&\l \bv\cdot\bn,\; \bq\r_e\quad\forall \bq\in [P_{k}(e)]^d, e\subset\pT,\label{P2} \\
\|\pi_h\bm{\tau}-\bm{\tau}\|&\le& Ch^{k+1}|\bm{\tau}|_{k+1}.\label{zhang22}
\end{eqnarray}
\end{lemma}
The proof of the above lemma can be found in \cite{sf-sc-wg-stokes}.

\subsection{Velocity Reconstruction Operator} 
This section shall describe the velocity reconstruction operator that is H(div,$\Omega$)-conforming and preserves the divergence of the weak functions for all polygons $T\in\mathcal{T}_h$. The main idea is to employ a sub-triangulation of each polygon and use the piecewise RT$_k(T_j)$ functions with some extra constraints. Here we shall cite the construction of the finite element space $\Lambda_k(T)$, which has been studies in \cite{sf-wg-part-II}.



Assume no additional inner vertex/edges is introduced in subdividing a polygon $T$ into $n$ triangles $\{T_i\}_{i = 1}^n$. Thus, we have $n-1$ internal edges which separate $T$ into $n$ parts. 
In the 2D setting, we are ready to define the velocity reconstruction operator as below:
For $\bv=\{\bv_0,\bv_b\}\in V_h$, the projection $\Pi_h:V_h\to \tilde{V}_h$ is defined such that
\begin{subequations}
\begin{eqnarray}
&&\int_{e_{i}\subset\partial T}(\Pi_h\bv\cdot\bn_{i})p_k ds = \int_{e_i\subset\partial T}(\bv_b\cdot\bn_i) p_k ds,\ \forall p_k\in P_k(e_{i})\label{eq:cond-1}\\
&&\int_T \Pi_h\bv\cdot\bn_1 p_{k-1}d\bx = \int_T\bv_0\cdot\bn_1 p_{k-1}d\bx,\ \forall p_{k-1}\in P_{k-1}(T),\label{eq:cond-2} \\
&&\int_{T_i}\Pi_h\bv\cdot\bn_2 p_{k-1}d\bx =\int_{T_i} \bv_0\cdot\bn_2 p_{k-1},\ \forall p_{k-1}\in P_{k-1}(T_i),\ i = 1,\cdots,n, \label{eq:cond-3}\\
&&\int_{e_{j}\subset T^0}\ljump\Pi_h\bv\rjump\cdot\bn_{j}p_k ds = 0,\ \forall p_k\in P_k(e_{j}),\label{eq:cond-4}\\
&&\int_{T_1}\nabla\cdot(\Pi_h\bv|_{T_i}-\Pi_h\bv|_{T_1})p_k d\bx = 0,\ \forall p_k\in P_k(T_1),\ i = 2,\dots,n,\label{eq:cond-5}
\end{eqnarray}
\end{subequations}
where $e_{ij}$ is the $j$-th edge of $T_j$ with a fixed normal vector $\bn_{ij}$, $\bn_1$ is a unite vector not parallel to any internal face normal $\bn_{ij}$, ($\bn_1,\bn_2$) forms a right-hand orthonormal system, $\ljump\cdot\rjump$ denotes the jump on a edge, $\Pi_h\bv|_{T_i}$ is understood as a polynomial vector which can be used on another triangle $T_1.$ When $k = 0$, the conditions (\ref{eq:cond-2})-(\ref{eq:cond-3}) are not needed. 

%

%
%

\begin{lemma}
For $\bv=\{\bv_0,\bv_b\}\in V_h$, there exists a projection $\Pi_h:V_h\to \tilde{V}_h$ such that
\begin{eqnarray}
(\Pi_h\bv,\bq)_T&=&(\bv_0,\bq)_T\quad\forall \bq\in [P_{k-1}(T)]^d, \label{zhang0}\\
\l\Pi_h\bv\cdot\bn,\;w\r_e&=&\l\bv_b\cdot\bn,\;w\r_e 
   \quad\forall w\in P_{k}(e), \ e\subset\pT, \label{zhang1}\\
\|\Pi_h\bv-\bv_0\|&\le& C(\sum_{T\in\T_h}h_T\|\bv_0-\bv_b\|_{\pT}^2)^{1/2}.\label{zhang2}
\end{eqnarray}
\end{lemma}

\begin{proof} 
A projection operator $\Pi_h$ is defined in \cite{sf-wg-part-II} for smooth functions
  satisfying \eqref{zhang0}-\eqref{zhang1}.
Then by the definition of weak divergence $\nabla_w\cdot\bv$ For any $\bv=\{\bv_0,\bv_b\}$, we can select 
any smooth function $\bu \in\b H^1(\Omega)$  such that $Q_h \bu=\bv$, i.e.,  $Q_0 \bu =\bv_0$
  and $Q_b \bu=\bv_b$.  We note that $Q_h$ is an on-to mapping from $\b H^1(\Omega)$ to $V_h$.
Now we define
\a{ \Pi_h \bv =\Pi_h \bu.  }
It follows that
\a{  (\Pi_h\bv,\bq)_T&=(\Pi_h\bu,\bq)_T=  (Q_0\bu,\bq)_T= (\bv_0, \bq)_T 
} for all $\bq\in [P_{k-1}(T)]^d$.  For the same reason, we have
\a{  \l\Pi_h\bv\cdot\bn,\;w\r_e&=\l\Pi_h\bu\cdot\bn,\;w\r_e=
       \l \bu \cdot\bn,\;w\r_e=  \l Q_b \bu \cdot\bn,\;w\r_e  
       =\l  \bv_b\cdot\bn,\;w\r_e }
for all $w\in P_{k}(e), \ e\subset\pT$.
From interpolation/projection data and scaling argument, we have
\a{  \| \Pi_h \b u \|_T &\le C (\|\Pi_T^{k-1} \b u \|_0 + h_T \| Q_b \bu \cdot\bn\|_{\partial T}), }
where $\Pi_T^{k-1}$ is the $L^2$-projection to $[P_{k-1}(T)]^d$.
Finally we derive 
\a{ \|\Pi_h\bv-\bv_0\|^2& = \|\Pi_h (\bu -\bv_0) \|^2 \\
       & \le C \sum_{T\in \T_h}\Big(  \| \Pi_T^{k-1}  (Q_0  \bu -\bv_0) \|_T^2 
                  + h_T^2 \| Q_b (\bu-\bv_0) \cdot\bn\|_{\partial T}^2 \Big)\\
     &= Ch_T^2 \sum_{T\in \T_h} \|  (\bv_b -\bv_0) \cdot\bn\|_{\partial T}^2 \\ 
     &\le Ch_T^2 \sum_{T\in \T_h} \|  \bv_b -\bv_0 \|_{\partial T}^2,   }
and thus complete the proof.
\end{proof}

\section{Numerical Scheme}\label{Sect:Scheme}
This section contributes to develop the new weak Galerkin finite element scheme and investigate the well-posedness for the proposed scheme.
\subsection{Finite Element Scheme}

We start this section by introducing the following simple WG finite element scheme without stabilizers.

\begin{algorithm}\label{Algorithm:WG1}
Our new numerical approximation for (\ref{eq:pde-1})-(\ref{eq:bc}) is seeking $\bu_h\in V_h^0$ and $p_h\in W_h$ such that for all $\bv\in V_h^0$ and $w\in W_h$,
\begin{eqnarray}
(\nu\nabla_w\bu_h,\ \nabla_w\bv)-(\nabla_w\cdot\bv,\;p_h)&=&(f,\;\Pi_h\bv),\label{wg1}\\
(\nabla_w\cdot\bu_h,\;w)&=&0.\label{wg2}
\end{eqnarray}
\end{algorithm}

In comparison, we shall also cite the following stabilizer free weak Galerkin finite element scheme in \cite{sf-sc-wg-stokes}.
\begin{algorithm}\cite{sf-sc-wg-stokes}.\label{Algorithm:WG2}
A numerical approximation for (\ref{eq:pde-1})-(\ref{eq:bc})
is seeking $\bu_h\in V_h$ and $p_h\in W_h$ such that for all $\bv\in V_h^0$ and $w\in W_h$,
\begin{eqnarray}
(\nu\nabla_w\bu_h,\nabla_w\bv)-(\nabla_w\cdot\bv, p_h)&=&({\bf f}, \bv),\label{w1}\\
(\nabla_w\cdot\bu_h, w)&=&0.\label{w2}
\end{eqnarray}
\end{algorithm}

\begin{remark}
It is noted that Algorithm~\ref{Algorithm:WG1} and Algorithm~\ref{Algorithm:WG2} share the same stiffness matrix.
\end{remark}

\subsection{Well Posedness}\label{Sect:FEScheme}
Let $Q_0$ and $Q_b$ be the two element-wise defined $L^2$ projections onto $[P_k(T)]^d$ and $[P_{k}(e)]^d$ with $e\subset\partial T$ on $T$ respectively for velocity. Denote by $\mathcal{Q}_h$ the element-wise defined $L^2$ projection onto $P_{k}(T)$ on each element $T$ for pressure variable. Let $\Q_h$ be the element-wise defined $L^2$ projection onto $\myLambda_k(T)$ on each element $T$ for the proximation of $\nabla\bu$.  Finally we define $Q_h\bu=\{Q_0\bu,Q_b\bu\}\in V_h$ for the true solution $\bu$. The following lemma reveal commutative properties for weak gradient $\nabla_w$ and weak divergence $\nabla_w\cdot$.

\begin{lemma}
Let $\boldsymbol\phi\in [H_0^1(\Omega)]^d$, then on  $T\in\T_h$
\begin{eqnarray}
\nabla_w Q_h\boldsymbol\phi &=&\Q_h\nabla\boldsymbol\phi,\label{key1}\\
\nabla_w \cdot Q_h\boldsymbol\phi &=&\mathcal{Q}_h\nabla\cdot\boldsymbol\phi.\label{key2}
\end{eqnarray}
\end{lemma}
\begin{proof}
Using (\ref{wg}) and  integration by parts, we have that for
any $\tau\in \Lambda_k(T)$
\begin{eqnarray*}
(\nabla_w Q_h\boldsymbol\phi,\tau)_T &=& -(Q_0\boldsymbol\phi,\nabla\cdot\tau)_T
+\langle Q_b\boldsymbol\phi,\tau\cdot\bn\rangle_{\pT}\\
&=& -(\boldsymbol\phi,\nabla\cdot\tau)_T
+\langle \boldsymbol\phi,\tau\cdot\bn\rangle_{\pT}\\
&=&(\nabla \boldsymbol\phi,\tau)_T=(\Q_h\nabla\boldsymbol\phi,\tau)_T.
\end{eqnarray*}
The equation above proves the identity (\ref{key1}).

It follows from  (\ref{wd}) and  integration by parts that for any $w\in P_{k}(T)$
\begin{eqnarray*}
(\nabla_w \cdot Q_h\boldsymbol\phi,w)_T &=& -(Q_0\boldsymbol\phi,\nabla w)_T
+\langle Q_b\boldsymbol\phi\cdot\bn,w\rangle_{\pT}\\
&=& -(\boldsymbol\phi,\nabla w)_T
+\langle \boldsymbol\phi\cdot\bn,w\rangle_{\pT}\\
&=&(\nabla\cdot \boldsymbol\phi,w)_T=(\mathcal{Q}_h\nabla\cdot\boldsymbol\phi,w)_T,
\end{eqnarray*}
which proves (\ref{key2}).
\end{proof}

For any function $\varphi\in H^1(T)$, the following trace
inequality holds true (see \cite{wymix} for details):
\begin{equation}\label{trace}
\|\varphi\|_{e}^2 \leq C \left( h_T^{-1} \|\varphi\|_T^2 + h_T
\|\nabla \varphi\|_{T}^2\right).
\end{equation}

 We introduce two semi-norms $\3bar \bv\3bar$ and  $\|\bv\|_{1,h}$
   for any $\bv\in V_h$ as follows:
\begin{eqnarray}
\3bar \bv\3bar^2 &=& \sum_{T\in\T_h}(\nabla_w\bv,\nabla_w\bv)_T, \label{norm1}\\
\|\bv\|_{1,h}^2&=&\sum_{T\in \T_h}\|\nabla \bv_0\|_T^2+\sum_{T\in\T_h}h_T^{-1}\|\bv_0-\bv_b\|_{\pT}^2.\label{norm2}
\end{eqnarray}

It is easy to see that $\|\bv\|_{1,h}$ defines a norm in $V_h^0$.
Next we will show that $\3bar  \cdot \3bar$ also defines a norm in $V_h^0$ by proving the equivalence of $\3bar\cdot\3bar$ and $\|\cdot\|_{1,h}$ in $V_h$.

The following norm equivalence has been proved in  \cite{sf-wg-part-II} for each component of $\bv$,
\begin{equation}\label{happy}
C_1\|\bv\|_{1,h}\le \3bar \bv\3bar\le C_2 \|\bv\|_{1,h} \quad\forall \bv\in V_h.
\end{equation}

The inf-sup condition for the finite element  formulation (\ref{wg1})-(\ref{wg2}) is derived in the following lemma.

\smallskip
\begin{lemma}\label{Lemma:inf-sup}
There exists a positive constant $\beta$ independent of $h$ such that for all $\rho\in W_h$,
\begin{equation}\label{inf-sup}
\sup_{\bv\in V_h}\frac{(\nabla_w\cdot\bv,\rho)}{\3bar\bv\3bar}\ge \beta
\|\rho\|.
\end{equation}
\end{lemma}

\begin{lemma}\label{ue}
The weak Galerkin method (\ref{wg1})-(\ref{wg2}) has a unique solution.
\end{lemma}

\medskip

The proofs for Lemma \ref{Lemma:inf-sup} and Lemma \ref{ue} are similar to the proofs in \cite{sf-sc-wg-stokes}.

\section{Main Results}\label{Sect:MainResults}
This section contributes to the pressure-robust estimates for our proposed numerical scheme (\ref{wg1})-(\ref{wg2}). 
\subsection{Error Equations}\label{Sect:ErrorEquation}
In this section, we derive the equations that the errors satisfy.
Let $\be_h=Q_h\bu-\bu_h$ and $\varepsilon_h=\mathcal{Q}_hp-p_h$.

\begin{lemma}
The following error equations hold true for any $\bv\in V_h^0$ and $w\in W_h$,
\begin{eqnarray}
(\nu\nabla_w\be_h,\; \nabla_w\bv)-(\varepsilon_h,\;\nabla_w\cdot\bv)&=&\ell_1(\bu,\bv)+\ell_2(\bu,\bv),\label{ee1}\\
(\nabla_w\cdot\be_h,\ w)&=&0,\label{ee2}
\end{eqnarray}
where
\begin{eqnarray}
\ell_1(\bu,\bv)&=&(\nu\Delta\bu, \bv_0-\Pi_h\bv),\label{l1}\\
\ell_2(\bu,\bv)&=&\langle\nu(\nabla\bu-\Q_h\nabla\bu)\cdot\bn,\;\bv_0-\bv_b\rangle_{\partial\T_h}.\label{l2}
\end{eqnarray}
\end{lemma}
\begin{proof}
For a $\bv\in V_h$, we test (\ref{eq:pde-1}) by
$\Pi_h\bv$ to obtain
\begin{equation}\label{mm0}
-(\nu\Delta\bu,\;\Pi_h\bv)+(\nabla p,\ \Pi_h\bv)=(\bbf,\; \Pi_h\bv).
\end{equation}
Obviously,
\begin{equation}\label{mm00}
(\nu\Delta\bu,\;\Pi_h\bv)=(\nu\Delta\bu,\;\bv_0)-\ell_1(\bu,\bv).
\end{equation}
It follows from integration by parts and the fact $\langle\nabla \bu\cdot\bn,\bv_b\rangle_{\partial\T_h}=0$
\begin{equation}\label{d2}
-(\Delta\bu,\;\bv_0)=(\nabla\bu,\nabla \bv_0)_{\T_h}- \langle
\nabla \bu\cdot\bn,\bv_0-\bv_b\rangle_{\partial\T_h}.
\end{equation}
By integration by parts, (\ref{wg}) and (\ref{key1}),
\begin{eqnarray}
(\nabla \bu,\nabla \bv_0)_{\T_h}&=&(\Q_h\nabla\bu,\nabla \bv_0)_{\T_h}\nonumber\\
&=&-(\bv_0,\nabla\cdot (\Q_h\nabla \bu))_{\T_h}+\langle \bv_0, \Q_h\nabla \bu\cdot\bn\rangle_{\partial\T_h}\nonumber\\
&=&(\Q_h\nabla \bu, \nabla_w \bv)+\langle \bv_0-\bv_b,\Q_h\nabla \psi\cdot\bn\rangle_{\partial\T_h}\nonumber\\
&=&( \nabla_w Q_h\bu, \nabla_w \bv)+\langle \bv_0-\bv_b,\Q_h\nabla \bu\cdot\bn\rangle_{\partial\T_h}.\label{d3}
\end{eqnarray}
Combining (\ref{d2}) and (\ref{d3}) gives
\begin{eqnarray}
-(\nu\Delta\bu,\;\bv_0)&=&(\nu\nabla_w Q_h\bu, \nabla_w \bv)-\ell_2(\bu,\bv).\label{d4}
\end{eqnarray}
By (\ref{mm00}) and (\ref{d4}), we obtain
\begin{eqnarray}
-(\nu\Delta\bu,\;\Pi_h\bv)&=&(\nu\nabla_w Q_h\bu, \nabla_w \bv)-\ell_1(\bu,\bv)-\ell_2(\bu,\bv).\label{d44}
\end{eqnarray}
Using integration by parts and the definition of $\Pi_h$,  we have
\begin{eqnarray*}
(\nabla p,\ \Pi_h\bv)&=& -(p,\nabla\cdot\Pi_h\bv)+\l p, \Pi_h\bv\cdot\bn\r_{\partial\T_h}\\
&=& -(\mathcal{Q}_hp,\;\nabla\cdot\Pi_h\bv)\\
&=&(\nabla\mathcal{Q}_hp,\;\Pi_h\bv)_{\T_h}-\l \mathcal{Q}_hp,\; \Pi_h\bv\cdot\bn\r_{\partial\T_h}\\
&=&(\nabla\mathcal{Q}_hp,\;\bv_0)_{\T_h}-\l \mathcal{Q}_hp,\; \bv_b\cdot\bn\r_{\partial\T_h}\\
&=& -(\mathcal{Q}_hp,\;\nabla_w\cdot\bv),
\end{eqnarray*}
which implies
\begin{equation}\label{mm2}
(\nabla p,\ \Pi_h\bv)=-(\mathcal{Q}_hp,\nabla_w\cdot\bv).
\end{equation}
Substituting (\ref{d4}) and (\ref{mm2}) into (\ref{mm0}) gives
\begin{equation}\label{mm3}
(\nu\nabla_wQ_h\bu,\nabla_w\bv)-(\mathcal{Q}_hp,\nabla_w\cdot\bv)=(\bbf,\Pi_h\bv)+\ell_1(\bu,\bv)+\ell_2(\bu,\bv).
\end{equation}
The difference of (\ref{mm3}) and (\ref{wg1}) implies
\begin{equation}\label{mm4}
(\nu\nabla_w\be_h,\nabla_w\bv)-(\varepsilon_h,\nabla_w\cdot\bv)=\ell_1(\bu,\bv)+\ell_2(\bu,\bv)\quad\forall\bv\in V_h^0.
\end{equation}
Testing equation (\ref{eq:pde-2}) by $w\in W_h$ and using (\ref{key2}) give
\begin{equation}\label{mm5}
(\nabla\cdot\bu,\ w)=(\mathcal{Q}_h\nabla\cdot\bu,\ w)=(\nabla_w\cdot Q_h\bu,\ w)=0.
\end{equation}
The difference of (\ref{mm5}) and (\ref{wg2}) implies (\ref{ee2}). Thus, we have proved the lemma.
\end{proof}

\subsection{Error Estimates in Energy Norm}\label{Section:error-analysis}
In this section, we shall establish one super-convergence order
for the velocity approximation $\bu_h$ in $\3bar\cdot\3bar$-norm, $L^2$-norm and pressure approximation $p_h$ in $L^2$-norm with respective to the optimal order error estimate.

\begin{lemma}
Let $\bu\in [H^{k+2}(\Omega)]^d$ and
$\bv\in V_h$. Then, the following estimates hold true
\begin{eqnarray}
|\ell_1(\bu,\ \bv)|&\le& C\nu h^{k+1}|\bu|_{k+2}\3bar \bv\3bar,\label{mmm2}\\
|\ell_2(\bu,\ \bv)|&\le& C\nu h^{k+1}|\bu|_{k+2}\3bar \bv\3bar.\label{mmm3}
\end{eqnarray}
\end{lemma}

\begin{proof}
Let $\mathcal{P}_{k-1}$ be an element-wise $L^2$ projection onto $P_{k-1}(T)$ on each element $T\in\T_h$. It follows from (\ref{zhang0}), the definition of $\mathcal{P}_{k-1}$, (\ref{zhang2}), and (\ref{happy})
\begin{eqnarray*}
|\ell_1(\bu,\bv)|&=&|(\nu\Delta\bu, \Pi_h\bv-\bv_0)|\\
&=&|(\nu\Delta\bu-\mathcal{P}_{k-1}\Delta\bu, \Pi_h\bv-\bv_0)|\\
&\le&C\nu h^{k+1}|u|_{k+2}\|\bv\|_{1,h}\\
&\le&C\nu h^{k+1}|u|_{k+2}\3bar\bv\3bar.
\end{eqnarray*}
It follows from the Cauchy-Schwarz inequality, the trace inequality (\ref{trace}), and (\ref{happy})
\begin{eqnarray*}
|\ell_2(\bu,\bv)|&\le&\left| \langle\nu (\nabla \bu-\Q_h\nabla
\bu)\cdot\bn,\; \bv_0-\bv_b\rangle_{\pT_h} \right|\nonumber\\
&\le& \nu\left(\sum_{T\in\T_h}h_T\|\nabla \bu-\Q_h\nabla \bu\|^2_\pT\right)^{1/2}
\left(\sum_{T\in\T_h}h_T^{-1}\|\bv_0-\bv_b\|^2_\pT\right)^{1/2}\nonumber \\
&\le&  C\nu h^{k+1}|\bu|_{k+2}\3bar\bv\3bar.
\end{eqnarray*}
We have proved the lemma.
\end{proof}

\begin{theorem}\label{h1-bd}
Let $(\bu_h,p_h)\in V_h^0\times W_h$ be the solution of (\ref{wg1})-(\ref{wg2}). Then, the following error
estimates hold true
\begin{eqnarray}
\3bar Q_h\bu-\bu_h\3bar &\le& Ch^{k+1}|\bu|_{k+2},\label{errv}\\
\|\mathcal{Q}_hp-p_h\|&\le& C\nu h^{k+1}|\bu|_{k+2}.\label{errp}
\end{eqnarray}
\end{theorem}

\smallskip

\begin{proof}
By letting $\bv=\be_h$ in (\ref{ee1}) and $w=\varepsilon_h$ in
(\ref{ee2}) and then using the equation (\ref{ee2}), we have
\begin{eqnarray}
\nu\3bar \be_h\3bar^2&=&|\ell_1(\bu,\be_h)+\ell_2(\bu,\be_h)|.\label{main}
\end{eqnarray}
It then follows from (\ref{mmm2}) and (\ref{mmm3}) that
\begin{equation}\label{b-u}
\nu \3bar \be_h\3bar^2 \le C\nu h^{k+1}|\bu|_{k+2}\3bar \be_h\3bar.
\end{equation}
We have proved (\ref{errv}). To estimate
$\|\varepsilon_h\|$, we have from (\ref{ee1}) that
\[
(\varepsilon_h, \nabla\cdot\bv)=(\nu\nabla_w\be_h,\nabla_w\bv)-\ell_1(\bu, \bv)-\ell_2(\bu, \bv).
\]
Using (\ref{b-u}), (\ref{mmm2}) and
(\ref{mmm3}), we arrive at
\[
|(\varepsilon_h, \nabla\cdot\bv)|\le C\nu h^{k+1}|\bu|_{k+2}\3bar\bv\3bar.
\]
Combining the above estimate with the {\em inf-sup} condition
(\ref{inf-sup}) gives
\[
\|\varepsilon_h\|\le C\nu h^{k+1}|\bu|_{k+2}.
\]
We obtained estimate (\ref{errp}) and proved the theorem.
\end{proof}

\subsection{Error Estimates in $L^2$ Norm}\label{Sect:L2Error}

In this section, we derive super-convergence for velocity in the $L^2$-norm by duality argument.
Consider a dual problem that seeks $(\bpsi,\xi)$ satisfying
\begin{eqnarray}
-\Delta\bpsi+\nabla \xi&=\be_0 &\quad \mbox{in}\;\Omega,\label{dual-m}\\
\nabla\cdot\bpsi&=0 &\quad\mbox{in}\;\Omega,\label{dual-c}\\
\bpsi&= 0 &\quad\mbox{on}\;\partial\Omega.\label{dual-bc}
\end{eqnarray}
Assume that the dual problem (\ref{dual-m})-(\ref{dual-bc}) has the $\bH^{2}(\Omega)\times
H^1(\Omega)$-regularity property in the sense that the solution
$(\bpsi,\xi)\in \bH^{2}(\Omega)\times H^1(\Omega)$ and the
following a priori estimate holds true:
\begin{eqnarray}
\|\bpsi\|_{2}+\|\xi\|_1&\le& C\|\be_0\|.\label{reg}
\end{eqnarray}
We need the following lemma first.

\begin{lemma}
Let $\ell_1(\cdot,\cdot)$ and $\ell_2(\cdot,\cdot)$ be defined in (\ref{l1}) and (\ref{l2}), respectively. For  $\be_h=\{\be_0,\be_b\}=Q_h\bu-\bu_h\in V_h^0$ and $w\in W_h$, the following equations hold true,
\begin{eqnarray}
\|\be_0\|^2&=&\ell_1(\bu,Q_h\bpsi)+\ell_2(\bu,Q_h\bpsi)-\ell_2(\bpsi,\be_h)-\ell_3(\xi,\be_h).\label{wd1}
\end{eqnarray}
where
\[
\ell_3(\xi,\be_h)=\l \mathcal{Q}_h\xi-\xi, (\be_0-\be_b)\cdot\bn\r_{\partial\T_h}.
\]
\end{lemma}

\begin{proof}
Testing (\ref{dual-m}) by
$\be_0$ with $\be_h=\{\be_0,\be_b\}\in V_h^0$ gives
\begin{equation}\label{d1}
-(\Delta\bpsi,\;\be_0)+(\nabla \xi,\ \be_0)=(\be_0,\; \be_0).
\end{equation}
Letting $\bu=\bpsi$ and $\bv=\be_h$ in (\ref{d4}), we derive
\begin{eqnarray}
-(\Delta\bpsi,\;\be_0)&=&( \nabla_w Q_h\bpsi,\; \nabla_w \be_h)-\ell_2(\bpsi,\be_h).\label{dd1}
\end{eqnarray}
Using integration by parts and the fact $\l \xi, \bv_b\cdot\bn\r_{\partial\T_h}=0$,  we have
\begin{eqnarray*}
(\nabla \xi,\ \be_0)&=& -(\xi,\nabla\cdot\be_0)_{\T_h}+\l \xi, \be_0\cdot\bn\r_{\partial\T_h}\\
&=& -(\mathcal{Q}_h\xi,\nabla\cdot\be_0)_{\T_h}+\l \xi, (\be_0-\be_b)\cdot\bn\r_{\partial\T_h}\\
&=&(\nabla\mathcal{Q}_h\xi,\be_0)_{\T_h}-\l \mathcal{Q}_h\xi, \be_0\cdot\bn\r_{\partial\T_h}+\l \xi, (\be_0-\be_b)\cdot\bn\r_{\partial\T_h}\\
&=&-(\mathcal{Q}_h\xi,\nabla_w\cdot\be_h)-\l \mathcal{Q}_h\xi, (\be_0-\be_b)\cdot\bn\r_{\partial\T_h}+\l \xi, (\be_0-\be_b)\cdot\bn\r_{\partial\T_h}\\
&=&-(\mathcal{Q}_h\xi,\nabla_w\cdot\be_h)-\ell_3(\xi,\be_h),
\end{eqnarray*}
which yields
\begin{equation}\label{mm222}
(\nabla \xi,\ \be_0)=-(\mathcal{Q}_h\xi,\nabla_w\cdot\be_h)-\ell_3(\xi,\be_h).
\end{equation}
By (\ref{ee2}), we have
\[
(\mathcal{Q}_h\xi,\;\nabla_w\cdot\be_h)=0.
\]
Then the equation above implies
\begin{equation}\label{mm22}
(\nabla \xi,\ \be_0)=-\ell_3(\xi,\be_h).
\end{equation}
Combining (\ref{dd1}) and (\ref{mm22}) with (\ref{d1}) yields 
\begin{equation}\label{ddd1}
(\nabla_w Q_h\bpsi,\; \nabla_w\be_h)=(\be_0,\be_0)+\ell_2(\bpsi,\be_h)+\ell_3(\xi,\be_h).
\end{equation}
Testing equation (\ref{dual-c}) by $w\in W_h$ and using (\ref{key2}) give
\begin{equation}\label{d6}
(\nabla\cdot\bpsi,\ w)=(\mathcal{Q}_h\nabla\cdot\bpsi,\ w)=(\nabla_w\cdot Q_h\bpsi,\ w)=0.
\end{equation}
The equation (\ref{ee1}) implies
\begin{eqnarray}
(\nabla_w \be_h,\; \nabla_w Q_h\bpsi)-(\epsilon_h,\;\nabla_w\cdot Q_h\bpsi)&=&\ell_1(\bu,Q_h\bpsi)+\ell_2(\bu,Q_h\bpsi).\label{d9}
\end{eqnarray}
Using (\ref{d6}), we have $(\epsilon_h,\;\nabla_w\cdot Q_h\bpsi)=0$. Then (\ref{d9}) becomes
\begin{eqnarray}
(\nabla_w \be_h,\; \nabla_w Q_h\bpsi)&=&\ell_1(\bu,Q_h\bpsi)+\ell_2(\bu,Q_h\bpsi).\label{d10}
\end{eqnarray}
Combining (\ref{ddd1}) and (\ref{d10}), we derive
\[
\|\be_0\|^2=\ell_1(\bu,Q_h\bpsi)+\ell_2(\bu,Q_h\bpsi)-\ell_2(\bpsi,\be_h)-\ell_3(\xi,\be_h)
\]
and prove the lemma.
\end{proof}

\medskip

\begin{theorem}\label{them:l2}
Let  $(\bu_h,p_h)\in V_h^0\times W_h$ be the solution of
(\ref{wg1})-(\ref{wg2}). Assume that (\ref{reg}) holds true.  Then, we have
\begin{equation}\label{l2err}
\|Q_0\bu-\bu_0\|\le Ch^{k+2}|\bu|_{k+2}.
\end{equation}
\end{theorem}

\begin{proof}
The equation (\ref{wd1}) yields
\begin{eqnarray}
\|\be_0\|^2=\ell_1(\bu,Q_h\bpsi)+\ell_2(\bu,Q_h\bpsi)-\ell_2(\bpsi,\be_h)-\ell_3(\xi,\be_h).\label{d7}
\end{eqnarray}
The four terms on the right hand side of (\ref{d7}) are bounded  next.
The estimates (\ref{zhang0}), (\ref{zhang2}) and (\ref{trace}) imply
\begin{eqnarray}
\ell_1(\bu,Q_h\bpsi)&=&(\Delta\bu, Q_0\bpsi-\Pi_hQ_h\bpsi)\label{e0}\nonumber\\
&=&(\Delta\bu-\mathcal{P}_{k-1}\Delta\bu, Q_0\bpsi-\Pi_hQ_h\bpsi)_{\T_h}\nonumber\\
&\le&Ch^k|\bu|_{k+2}(\sum_{T\in\T_h}h_T\|Q_0\bpsi-Q_b\bpsi\|_{\pT}^2)^{1/2}\nonumber\\
&\le&Ch^{k+2}|\bu|_{k+2}|\bpsi|_2.
\end{eqnarray}
Using the Cauchy-Schwarz inequality, the trace inequality (\ref{trace}) and the definition of $\Q_h$
we obtain
\begin{eqnarray}
|\ell_2(\bu,Q_h\bpsi)|&\le&\left| \langle (\nabla \bu-\Q_h\nabla
\bu)\cdot\bn,\;
Q_0\bpsi-Q_b\bpsi\rangle_{\pT_h} \right|\nonumber\\
&\le& \left(\sum_{T\in\T_h}\|\nabla \bu-\Q_h\nabla \bu\|^2_\pT\right)^{1/2}
\left(\sum_{T\in\T_h}\|Q_0\bpsi-\bpsi\|^2_\pT\right)^{1/2}\nonumber \\
&\le&  Ch^{k+2}|\bu|_{k+2}|\bpsi|_2.\label{e1}
\end{eqnarray}
Using the Cauchy-Schwarz inequality, the trace inequality (\ref{happy}) and (\ref{errv}), we obtain
\begin{eqnarray}
|\ell_2(\bpsi,\be_h)|&\le&\left| \langle (\nabla \bpsi-\Q_h\nabla
\bpsi)\cdot\bn,\;
\be_0-\be_b\rangle_{\pT_h} \right|\nonumber\\
&\le& \left(\sum_{T\in\T_h}h_T\|\nabla \bpsi-\Q_h\nabla \bpsi\|^2_\pT\right)^{1/2}
\left(\sum_{T\in\T_h}h_T^{-1}\|\be_0-\be_b\|^2_\pT\right)^{1/2}\nonumber \\
&\le&  Ch|\bpsi|_2\3bar\be_h\3bar\nonumber\\
&\le&Ch^{k+2}|\bu|_{k+2}|\bpsi|_2.\label{e3}
\end{eqnarray}
Similarly, we have
\begin{eqnarray}
|\ell_3(\xi, \be_h)|&\le&\left| \langle \mathcal{Q}_h\xi-\xi,\;
(\be_0-\be_b)\cdot\bn\rangle_{\pT_h} \right|\nonumber\\
&\le& \left(\sum_{T\in\T_h}h_T\|\mathcal{Q}_h\xi-\xi\;\|^2_\pT\right)^{1/2}
\left(\sum_{T\in\T_h}h_T^{-1}\|\be_0-\be_b\|^2_\pT\right)^{1/2}\nonumber \\
&\le&  Ch^{k+2}|\bu|_{k+2}|\xi|_1.\label{e4}
\end{eqnarray}
Combining all the estimates  above with (\ref{d7}) yields
$$
\|\be_0\|^2 \leq C h^{k+2}|\bu|_{k+2}(\|\bpsi\|_2+\|\xi\|_1).
$$
The estimate (\ref{l2err}) follows from the above inequality and
the regularity assumption (\ref{reg}). We have completed the proof.
\end{proof}

\section{Implementation of $\Pi_h$ Operator}\label{Section:Implementation}
This section contributes to the details in the numerical implementation. 
In the following, for notation simplification, we shall take $k = 0$ to illustrate the derivation of velocity reconstruction into $\Lambda_0(T)$ space. In the reference triangle $\hat{T} = (0,0;1,0;0,1)$, let the basis functions of RT$_0(\hat{T})$ be:
\begin{eqnarray}
\hat{\Phi}_1(\hat{x},\hat{y}) = \sqrt{2}\begin{bmatrix}
\hat{x}\\ \hat{y}
\end{bmatrix},\
\hat{\Phi}_2(\hat{x},\hat{y}) = \begin{bmatrix}
\hat{x}-1\\ \hat{y}
\end{bmatrix},\
\hat{\Phi}_3(\hat{x},\hat{y}) = \begin{bmatrix}
\hat{x}\\ \hat{y}-1
\end{bmatrix}
\end{eqnarray}
Then the Piola transformation will be employed to map functions $\hat{\Phi}_j\in\mbox{RT}_0(\hat{T})$ to RT$_0(T_j)$ on the physical element  defined by
\begin{eqnarray}
\hat{\bf q}\rightarrow {\bf q}(x,y) = \frac{J_T}{|J_T|}\hat{\bf q}(\hat{x},\hat{y}).
\end{eqnarray}
Here we have Jacobi matrix and the determinant defined as below,
\begin{eqnarray*}
J_T = \begin{bmatrix}
\dfrac{\partial x}{\partial\hat{x}} & \dfrac{\partial x}{\partial\hat{y}}\\[1em]
\dfrac{\partial y}{\partial\hat{x}} & \dfrac{\partial y}{\partial\hat{y}}
\end{bmatrix}\mbox{ and } |J_T| = \text{det}(J_T).
\end{eqnarray*}

\begin{figure}[H]
\centering
\begin{tabular}{cc}
\includegraphics[width=0.45\textwidth]{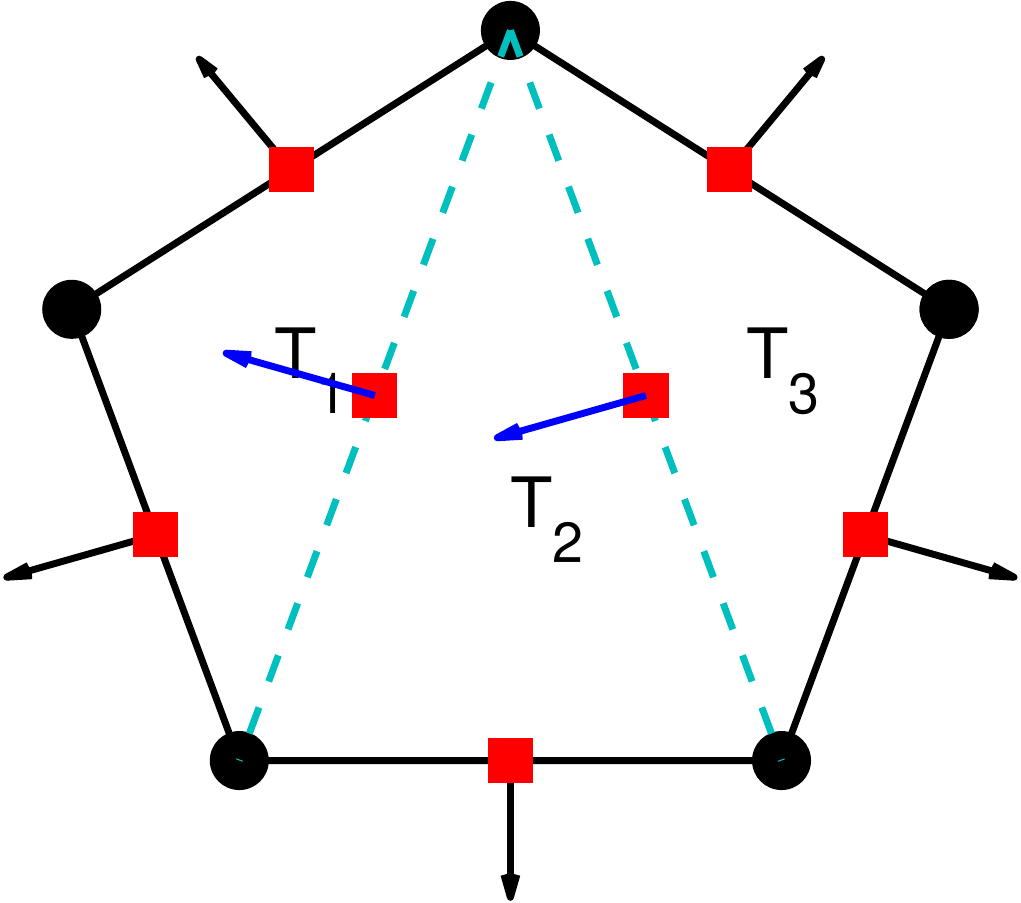}
&\includegraphics[width=0.5\textwidth]{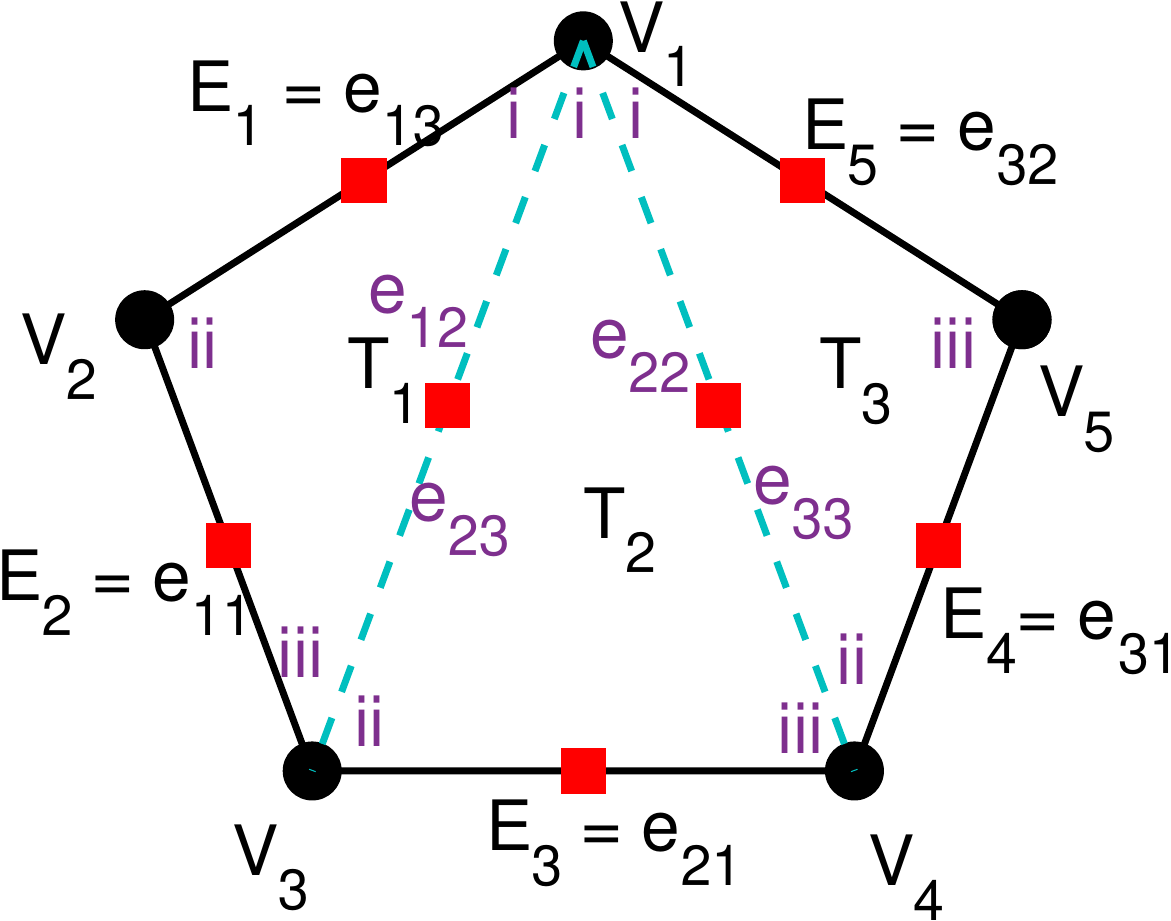}
\end{tabular}
\caption{Illustration of a possible partition of pentagon cell $T$ into three subtriangles.}\label{fig:PenFig}
\end{figure}

We shall consider a pentagon cell $T$ (shown in Figure~\ref{fig:PenFig}) for deriving the corresponding velocity reconstruction. We first perform the sub-triangulation for this cell $T$ into three small triangles $T_j$ ($j=1,\cdots,3$). The local indices for sub-triangles are denoted by Roman numbers (i,ii,iii). There are two interior edges in the cell $T$ and we shall use the notation $e_{ij}$ to denote the sub-triangle's indices $i$ and $j$.
The vertex indices are denoted as V$_j$ ($j$=1,$\cdots$,5); The boundaries for cell $T$ are denoted by E$_j$ ($j$=$1,\cdots,5$). All the labels and indexing are demonstrated in Figure~\ref{fig:PenFig}. For the piecewise polynomial basis, we shall assume the degree of freedom on each sub-triangle is three corresponding to the coefficients to be determined in front of the piecewise RT$_0(T_i)$ (i=1,2,3) basis. Thus in this particular setting, we shall derive a matrix $\mathbb{M}$ with size $9\times 9$. 
By the condition (\ref{eq:cond-1}), (\ref{eq:cond-4}), and (\ref{eq:cond-5}), we assume $\Pi_h\bv = \sum_{i=1}^3\left(\sum_{j=1}^3 c_{3(i-1)+j}\Phi_{j}^{(T_i)}(x,y)\right)$, where the basis function $\Phi_j^{(T_i)}(x,y)$ denotes the local RT$_0$ basis with index $j$ in the sub-triangle $T_i$. On the edge $E_1$, we only consider the RT$_0$ basis in the subtriangle $T_1$, and thus (\ref{eq:cond-1}) implies
\begin{eqnarray*}
\int_{E_1}\Pi_h \bv\cdot\bn_{E_1}ds = \int_{e_{13}}\sum_{j=1}^3c_{j}\Phi^{T_1}_j\cdot\bn_{13} ds
 = \int_{e_{13}} c_3\Phi_3^{T_1}\cdot\bn_{13}ds = c_3,
\end{eqnarray*}
which can be written to the first row in the matrix $\mathbb{M}$. Similarly, one can derive the 2rd to 5th rows for the matrix $\mathbb{M}$. Then for the condition  (\ref{eq:cond-4}), we have
\begin{eqnarray*}
0 &=& \int_{e_{12}} \sum_{j=1}^3c_{j}\Phi_j^{(T_1)}\cdot\bn_{e_{12}}ds + \int_{e_{23}}\sum_{j=1}^3 c_{3+j}\Phi_j^{(T_2)}\cdot\bn_{e_{23}}ds\\
&=& \int_{e_{12}}\left(c_2\Phi_2^{(T_1)}\cdot\bn_{e_{12}} + c_6\Phi_3^{(T_2)}\cdot\bn_{e_{23}}\right)ds = c_2 + c_6,
\end{eqnarray*}
which gives the 6th row of the matrix $\mathbb{M}$. Similarly, the 7th row can be obtained in the same way. Next, we shall employ the condition (\ref{eq:cond-5}) for deriving the last two rows for the matrix $\mathbb{M}$. We have 
\begin{eqnarray*}
0 &=& \int_{T_1}\nabla\cdot(\Pi_h\bv|_{T_2}-\Pi_h\bv|_{T_1}) d\bx = \int_{T_1}\nabla\cdot\left(\sum_{j=1}^3 c_{3+j}\Phi_j^{(T_2)}-\sum_{j=1}^3 c_j\Phi_j^{(T_1)}\right)d\bx\\
&=&\int_{T_1}  \left(c_4\frac{2\sqrt{2}}{|J_{T_2}|} + c_5\frac{2}{|J_{T_2}|}+ c_6\frac{2}{|J_{T_2}|} - c_1\frac{2\sqrt{2}}{|J_{T_1}|}  - c_2\frac{2}{|J_{T_1}|}  - c_3\frac{2}{|J_{T_1}|}\right) d\bx\\
&=& \frac{|J_{T_1}|}{2} \left(c_4\frac{2\sqrt{2}}{|J_{T_2}|} + c_5\frac{2}{|J_{T_2}|}+ c_6\frac{2}{|J_{T_2}|} - c_1\frac{2\sqrt{2}}{|J_{T_1}|}  - c_2\frac{2}{|J_{T_1}|}  - c_3\frac{2}{|J_{T_1}|}\right) \\
&=& |{T_1}| \left(c_4\frac{\sqrt{2}}{|{T_2}|} + c_5\frac{1}{|{T_2}|}+ c_6\frac{1}{|{T_2}|} - c_1\frac{\sqrt{2}}{|{T_1}|}  - c_2\frac{1}{|{T_1}|}  - c_3\frac{1}{|{T_1}|}\right)
\end{eqnarray*}
by using $|J_{T_i}| = 2|T_j|$ and cancelling the common factor gives row 8th. Similarly, the 9th row can be obtained. Thus the matrix can be re-written as follows,
\begin{eqnarray*}
\mathbb{M} = \begin{bmatrix}
0 & 0 & 1 & 0 & 0 & 0 & 0 & 0 & 0\\
\sqrt{2} & 0 & 0 & 0 &0 & 0 & 0 & 0 & 0\\
0 & 0 & 0 & \sqrt{2} & 0 & 0 & 0 & 0 & 0\\
0 & 0 & 0 & 0 & 0 & 0 & \sqrt{2} & 0 & 0\\
0 & 0 & 0 & 0 & 0 & 0 & 0 & 1 & 0 \\
0 & 1 & 0 & 0 & 0 & 1 & 0 & 0 & 0\\
0 & 0 & 0 & 0 & 1 & 0 & 0 & 0  & 1\\
-\sqrt{2} & -1 & -1 &\dfrac{\sqrt{2}|T_1|}{|T_2|} & \dfrac{|T_1|}{|T_2|} & \dfrac{|T_1|}{|T_2|} & 0 & 0 &0\\
-\sqrt{2} & -1 & -1 & 0 & 0 & 0 &\dfrac{\sqrt{2}|T_1|}{|T_3|} & \dfrac{|T_1|}{|T_3|} & \dfrac{|T_1|}{|T_3|}
\end{bmatrix}.
\end{eqnarray*}
By employing the matrix $\mathbb{M}$ and given $\bv = \{\bv_0,\bv_b\}$, one can construct the linear system $\mathbb{M}{\bf c} = {\bf b}$ for solving vector coefficients ${\bf c} = [c_j]_{j=1}^9$. For example, assume $\bv_b|_{E_1} = \bn_{E_1}$ and $\bv_b|_{E_j} = 0$ for $j = 2,\dots,5$, and then the vector ${\bf b}$ will be calculated by: the first component $[{\bf b}]_1 = \int_{E_1}\bv_b\cdot\bn_{E_1} ds = |E_1|$ and the other components are zeros.
The calculated $\Pi_h\bv$ corresponding to $\bv_b = \bn_{E_1}$, $\bv_b = \bn_{E_2}$, and $\bv_b = \bn_{E_3}$ are plotted in Figure~\ref{fig:Basis-NewRT0}.  As noted in this plot, on the edge $E_1$, the projection $\Pi_h\bv$ with $\bv_b = \bn_{E_1}$ preserves the quantity $\int_{E_1}\bv_b\cdot\bn|_{E_1}ds$.

\begin{figure}[H]
\centering
\begin{tabular}{ccc}
\includegraphics[width=0.3\textwidth]{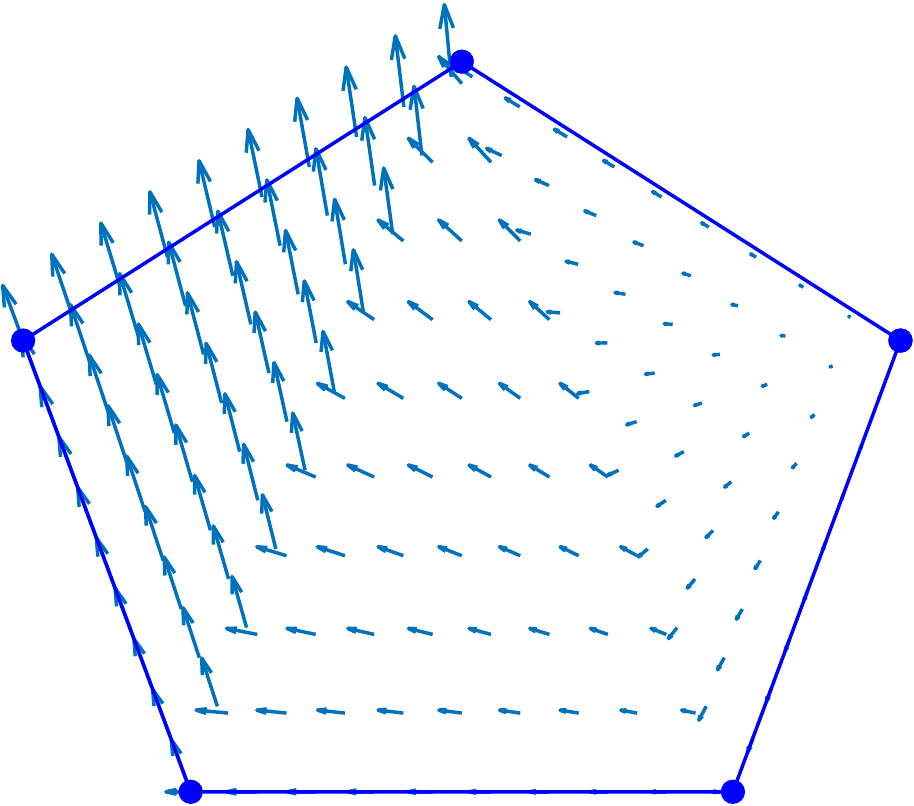}
&\includegraphics[width=0.3\textwidth]{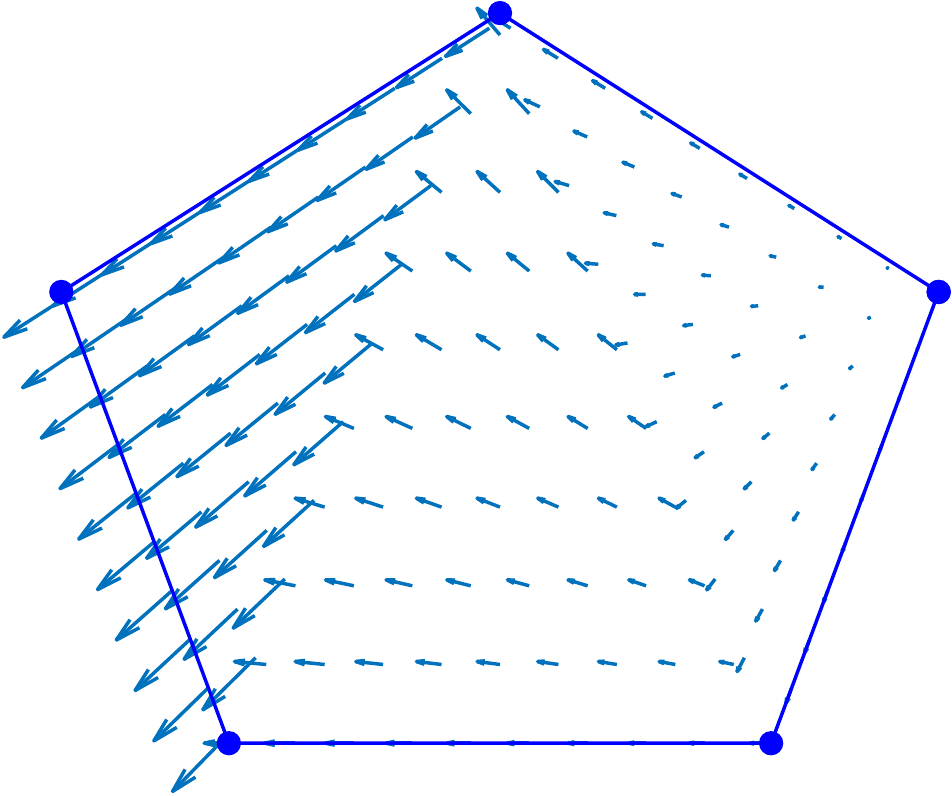}
&\includegraphics[width=0.3\textwidth]{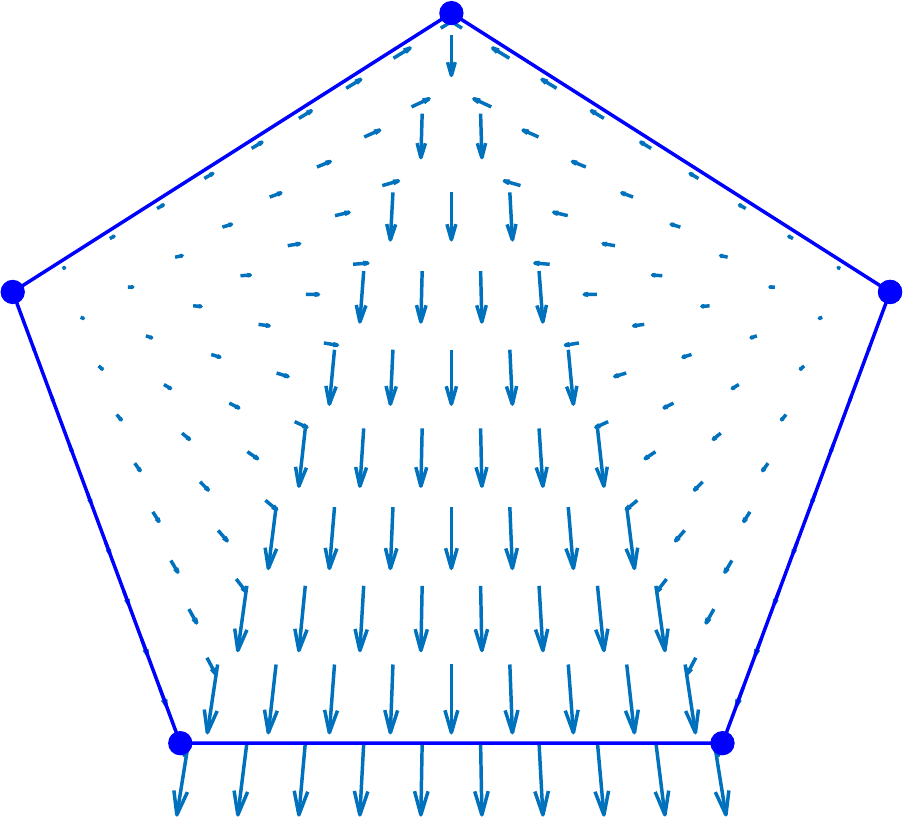}
\\
(a) & (b) & (c) 
\end{tabular}
\caption{Illustration of velocity reconstruction $\Pi_h\bv$ in the space $\Lambda_k(T)$, which is piecewise polynomial: (a) $\bv_b = \bn_{E_1}$; (b) $\bv_b=\bn_{E_2}$; (c) $\bv_b=\bn_{E_3}$; }\label{fig:Basis-NewRT0}
\end{figure}

\begin{figure}[H]
\centering
\begin{tabular}{ccc}
\includegraphics[width=0.3\textwidth]{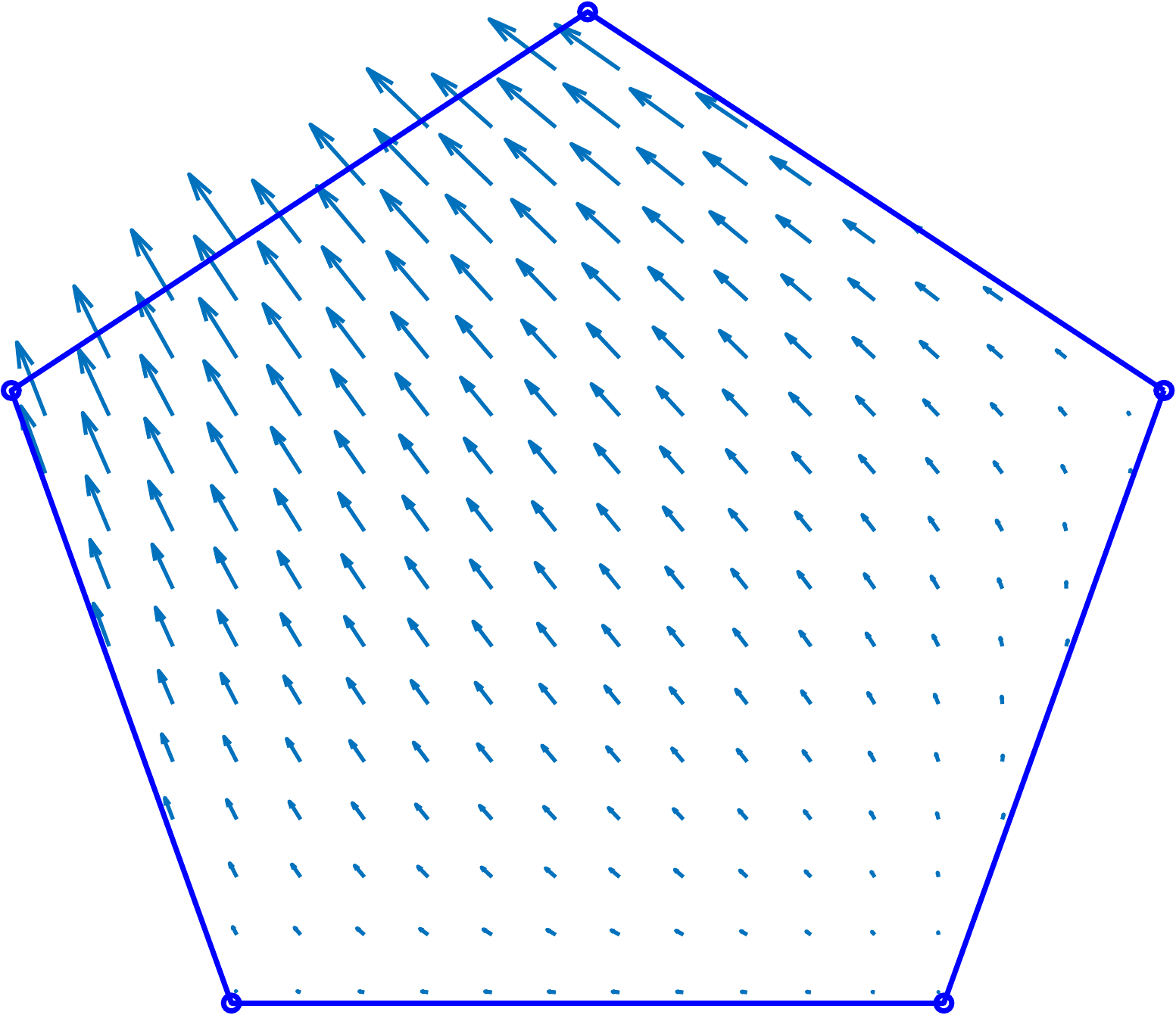}
&\includegraphics[width=0.3\textwidth]{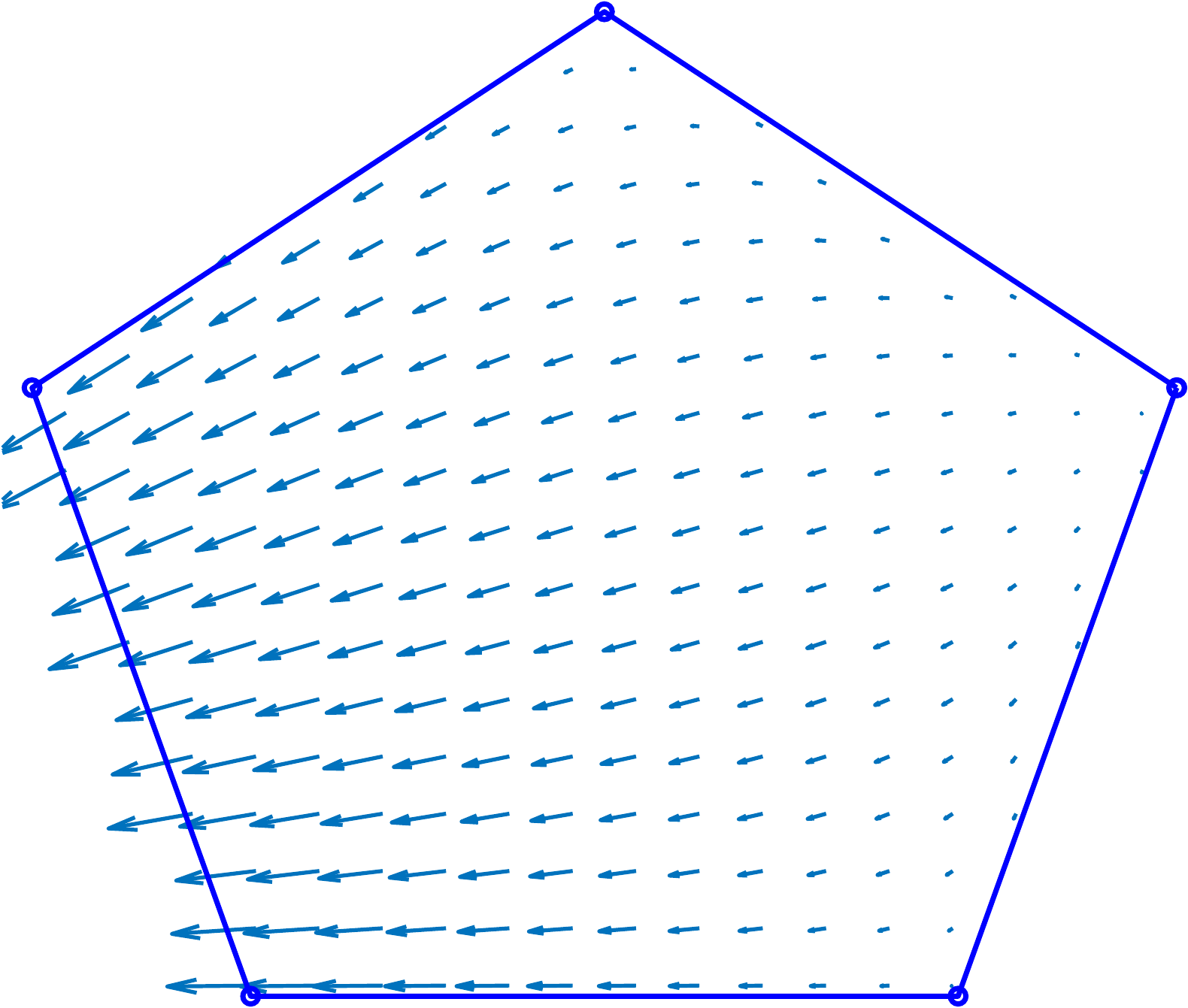}
&\includegraphics[width=0.3\textwidth]{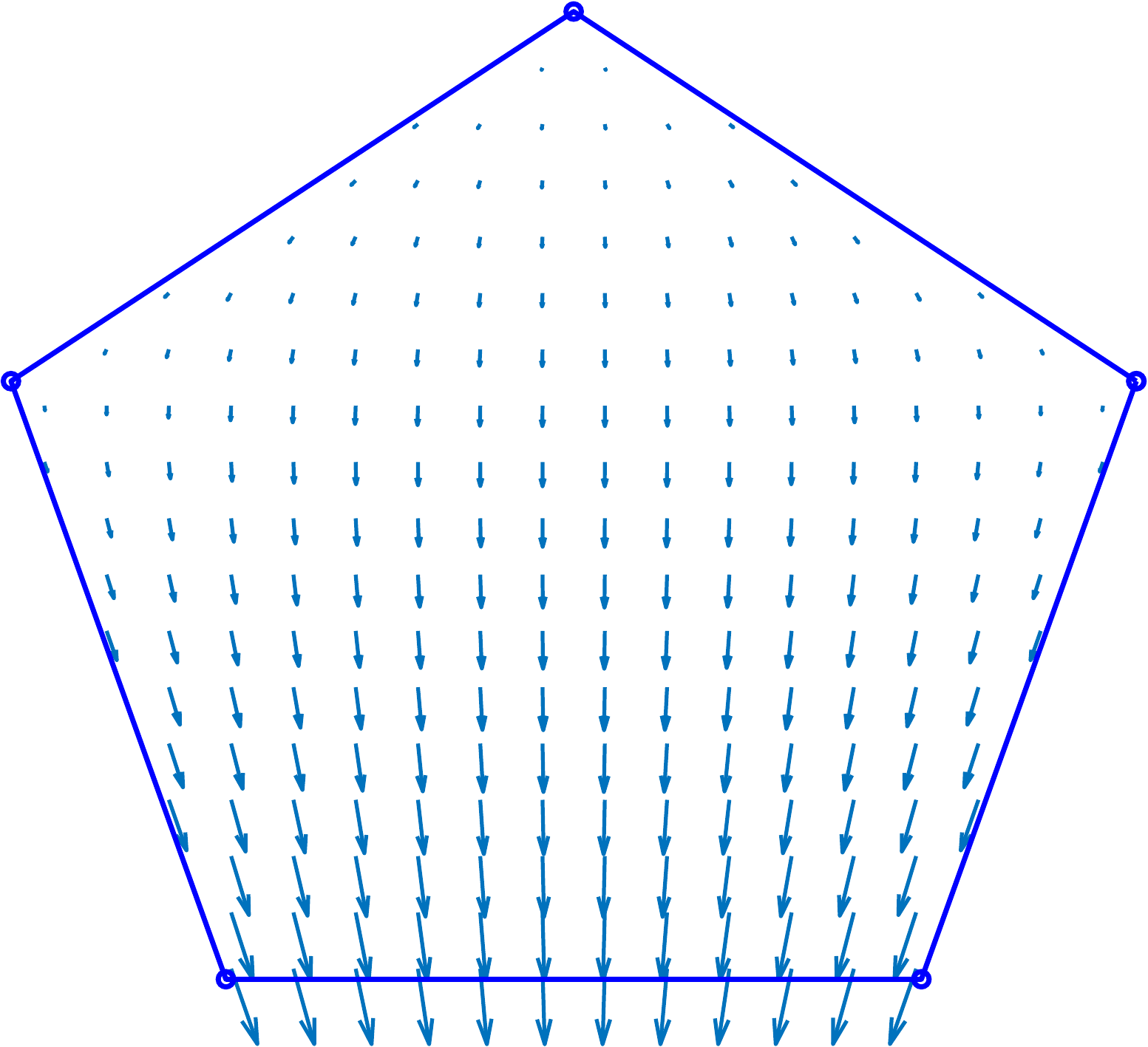}
\\
(a) & (b) & (c) 
\end{tabular}
\caption{Illustration of velocity reconstruction in the $\mathbb{CW}_0$ space in the Wachspress Coordinates \cite{mu}, which consists rational functions: (a) reconstruction of $\bn_{E_1}$; (b) reconstruction of $\bn_{E_2}$; (c) reconstruction of $\bn_{E_3}$.}\label{fig:cw0}
\end{figure}

Next, we compare the bases in $\Lambda_0(T)$ defined above with another polygonal RT$_0$-like bases in the  $\mathbb{CW}_0$ setting on the same pentagon cell, which is generated by employing the Wachspress Coordinates\cite{ChenWang}. Three velocity reconstructions based on $\mathbb{CW}_0$ space, corresponding to $\bv_b=\bn_{E_1},\bn_{E_2},\bn_{E_3}$, are plotted in Figure~\ref{fig:cw0}. As shown in Figure~\ref{fig:Basis-NewRT0}-\ref{fig:cw0}, the reconstructions in $\mathbb{CW}_0$ are one piece functions defined on the polygon $T$ consisting rational functions; the reconstructions in $\Lambda_0(T)$ are piecewise polynomials with the same normal component crossing the interior edges.

Since the stiffness matrix is the same as that of the previous stabilizer free WG scheme, we shall omit the implementation for the stiffness matrix assembling. In the body source assembling, we shall compute
\begin{eqnarray*}
({\bf f},\Pi_h\bv) &=& 
\sum_{T\in\mathcal{T}_h}\sum_{i,j} c_{3(i-1)+j}\int_{T_i} {\bf f}\Phi_j^{(T_i)}d\bx 
= \sum_{T\in\mathcal{T}_h}\left[{\bf c}\right]^\top\cdot\left[{\bf f}_{\Phi}\right],
\end{eqnarray*}
where ${\bf f}_{\Phi}$ denotes the integral of $\int_{T_i} {\bf f}\Phi_j^{(T_i)}d\bx$ for each piecewise RT$_0(T_i)$  basis ($i = 1,\cdots,\#$sub-triangles). Different as the basis in the $\mathbb{CW}_0$ space, which usually contains rational functions, here in our new numerical scheme, the integrand will be a production of ${\bf f}$ and piecewise polynomials. Thus our new constructed projection operator will not introduce extra errors when ${\bf f}$ is non-polynomial function, which is a problem for the $\mathbb{CW}_0$ space. 
Besides, the basis functions in $\mathbb{CW}_0$ are limited to the lowest order and on the convex element. In the sense of generalization to arbitrary scheme order and meshes, our new approach outperforms the scheme by employing $\mathbb{CW}_0$ spaces\cite{mu}. 

\section{Numerical Experiments}\label{Section:numerical-experiments}

This section shows several numerical examples to validate our theoretical conclusions that our proposed approach has one order super-convergence rate than the optimal rate and is also pressure-robust. We shall measure the errors by following norms for the velocity and pressure:
\begin{eqnarray*}
&&H^1\text{-Error }\3bar\be_h\3bar:=\3bar \bQ_h\bu-\bu_h\3bar=\left(\sum_{T\in\mathcal{T}_h}\int_T\left|\nabla_w(\bQ_h\bu-\bu_h)\right|^2d\bx\right)^{1/2},\\
&&L^2\text{-Error }\|\be_0\|:=\|\bQ_0\bu-\bu_0\|=\left(\sum_{T\in\mathcal{T}_h}\int_T\left|\bQ_0\bu-\bu_0\right|^2d\bx\right)^{1/2},\\
&&L^2\text{-Error }\|\epsilon_h\|:=\|\mathcal{Q}_hp-p_h\|:=\left(\sum_{T\in\mathcal{T}_h}\int_T\left|\mathcal{Q}_hp-p_h\right|^2d\bx\right)^{1/2}.
\end{eqnarray*}
As the theoretical conclusions in Theorem~\ref{h1-bd}/\ref{them:l2}, we expect the errors measured in $\3bar\be_h\3bar$, $\|\be_0\|$, and $\|\epsilon_h\|$ converge with the orders $\mathcal{O}(h^{k+1})$, $\mathcal{O}(h^{k+2})$, and $\mathcal{O}(h^{k+1})$.

\subsection{Simulations with Lowest Degree $k=0$}\label{Num-1}

Let $\Omega = (0,1)\times(0,1)$ and the exact solutions are chosen as below for testing:
\begin{eqnarray}
\bu = \begin{pmatrix}
10x^2y(x-1)^2(2y-1)(y-1)\\
-10xy^2(2x-1)(x-1)(y-1)^2
\end{pmatrix},\ p = 10x.
\end{eqnarray}
In the following, we shall test the numerical performance for the lowest order WG element $k = 0$.

\begin{itemize}
\begin{figure}[H]
\centering
\begin{tabular}{cc}
\includegraphics[width=0.45\textwidth]{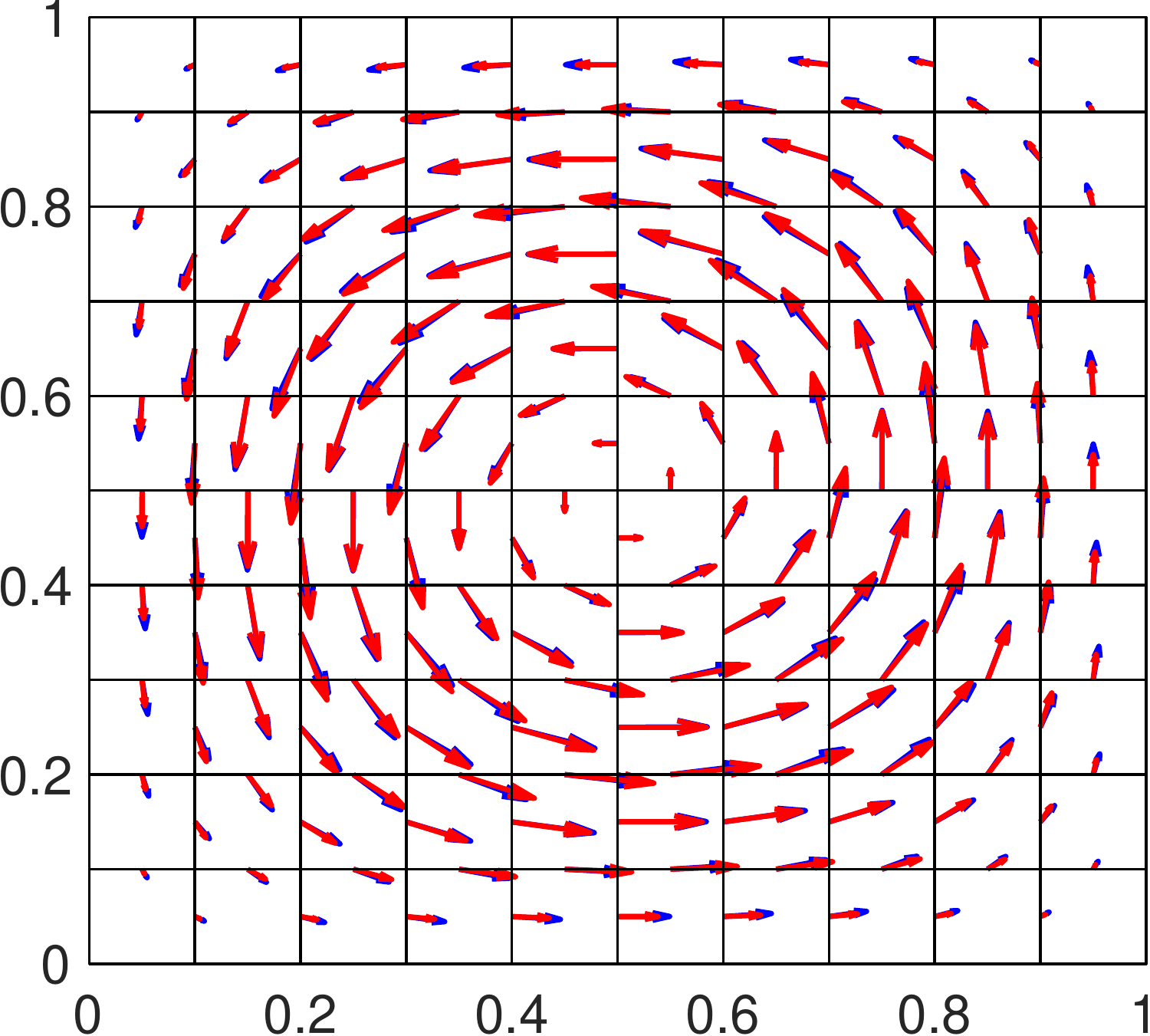}
&\includegraphics[width=0.45\textwidth]{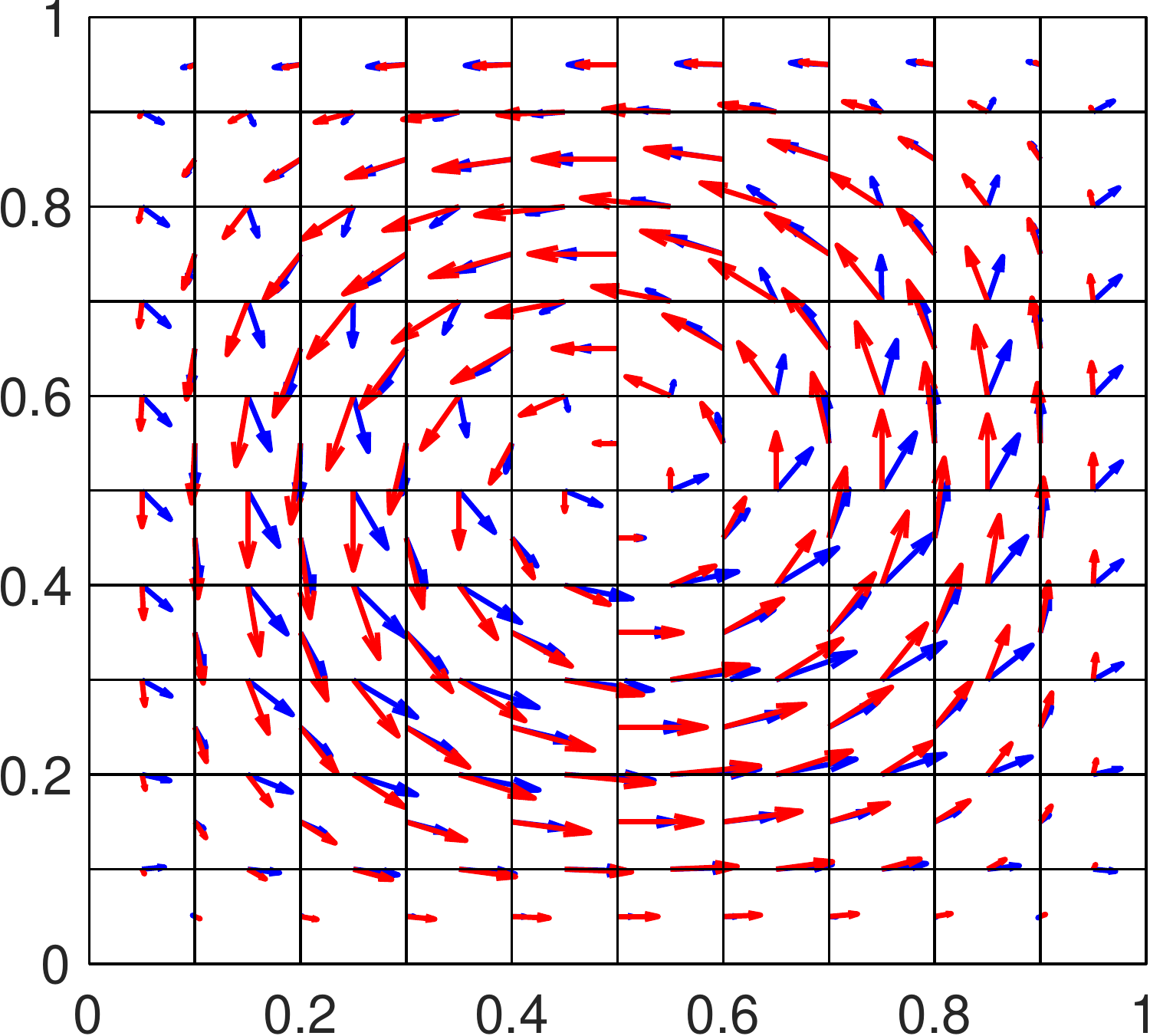}
\\
(a) & (b) 
\end{tabular}
\caption{Example~\ref{Num-1}. Plot of the vector fields of velocity. Red arrow denotes the exact solution; Blue arrow plots the WG solutions with $k=0$. (a). WG solution by pressure-robust Algorithm~\ref{Algorithm:WG1}; (b) WG solution by non-pressure-robust Algorithm~\ref{Algorithm:WG2}.} \label{fig:ex1_quiver}
\end{figure}
\item\textbf{Test case with $\nu = 1$ on rectangular grid.} In this test, we shall perform WG Algorithm~\ref{Algorithm:WG1} and WG Algorithm~\ref{Algorithm:WG2} and compare the numerical performance on the coarse mesh. Let $k = 0$, $\nu=1$,  $h = 1/10$, and the numerical solutions on the uniform rectangular mesh are illustrated in Figure~\ref{fig:ex1_quiver}. As one can see from these two figures that Algorithm~\ref{Algorithm:WG1} provides a better solution aligning with the exact velocity field. However, the right panel in this figure, produced by Algorithm~\ref{Algorithm:WG2}, plots a solution which can not preserve the vector field. It shows that as we modifying the body source assembling but remaining the same stiffness matrix, we are able to improve the simulations in velocity on the uniform rectangular meshes.

\begin{table}[H]
\caption{Example~\ref{Num-1}. Error profiles and convergence results on the uniform rectangular mesh.}\label{Tab:Ex1}
\tabcolsep=2pt
\begin{tabular}{c||cc|cc|cc||cc|cc|cc}\hline\hline
&$\3bar\be_h\3bar$& rate &$\|\be_0\|$& rate &$\|\epsilon_h\|$& rate	&$\3bar\be_h\3bar$& rate &$\|\be_0\|$& rate &$\|\epsilon_h\|$& rate	 \\ \hline
&\multicolumn{6}{c||}{$\nu$ = 1, Pressure Robust Algorithm~\ref{Algorithm:WG1} }&\multicolumn{6}{c}{$\nu$ = 1, Non-Pressure Robust Algorithm~\ref{Algorithm:WG2}	}	\\ \hline	
4	&2.42E-1	& 	&1.02E-2	&	&2.35E-2      &	
&8.92E-1  &	&7.34E-2. &	 &7.46E-1	&\\
8	&1.36E-1	&0.8	&3.55E-3	&1.5	&1.56E-2	&0.6	
&4.88E-1	&0.9	&2.28E-2	&1.7	&4.24E-1	&0.8\\
16	&7.04E-2	&1.0	&1.02E-3	&1.8	&5.62E-3	&1.5	
&2.52E-1	&1.0	&6.17E-3	&1.9	&2.27E-1	&0.9\\
32	&3.55E-2	&1.0	&2.68E-4	&1.9	&1.58E-3	&1.8	
&1.28E-1	&1.0	&1.58E-3	&2.0	&1.18E-1	&0.9\\
64	&1.78E-2	&1.0	&6.80E-5	&2.0	&4.09E-4	&2.0	
&6.40E-2	&1.0	&3.99E-4	&2.0	&6.01E-2	&1.0\\
128	&8.91E-3	&1.0	&1.71E-5	&2.0	&1.03E-4	&2.0	
&3.20E-2	&1.0	&9.99E-5	&2.0	&3.04E-2	&1.0\\ \hline
&\multicolumn{6}{c||}{$\nu$ = 1E-2, Algorithm~\ref{Algorithm:WG1}}&\multicolumn{6}{c}{$\nu$ = 1E-2, Algorithm~\ref{Algorithm:WG2}	}	\\ \hline				
4	&2.42E-1	& 	&1.02E-2	& 	&2.35E-4&		
&8.88E+1&	&7.33.      &	 &7.46E-1	&\\ 
8	&1.36E-1	&0.8	&3.55E-3	&1.5	&1.56E-4	&0.6	
&4.86E+1	&0.9	&2.28	&1.7	&4.24E-1	&0.8\\
16	&7.04E-2	&1.0	&1.02E-3	&1.8	&5.62E-5	&1.5	
&2.51E+1	&1.0	&6.15E-1	&1.9	&2.27E-1	&0.9\\
32	&3.55E-2	&1.0	&2.68E-4	&1.9	&1.58E-5	&1.8	
&1.27E+1	&1.0	&1.58E-1	&2.0	&1.18E-1	&0.9\\
64	&1.78E-2	&1.0	&6.80E-5	&2.0	&4.09E-6	&2.0	
&6.37	&1.0	&3.98E-2	&2.0	&6.01E-2	&1.0\\
128	&8.91E-3	&1.0	&1.71E-5	&2.0	&1.03E-6	&2.0	
&3.19	&1.0	&9.97E-3	&2.0	&3.04E-2	&1.0\\
\hline
&\multicolumn{6}{c||}{$\nu$ = 1E-4, Algorithm~\ref{Algorithm:WG1}}&\multicolumn{6}{c}{$\nu$ = 1E-4, Algorithm~\ref{Algorithm:WG2}}	\\ \hline					
4	&2.42E-1	& 	&1.02E-2&	&2.35E-6&		
&8.88E+3&	&7.33E+2&	 	&7.46E-1&\\	
8	&1.36E-1	&0.8	&3.55E-3	&1.5	&1.56E-6	&0.6	
&4.86E+3	&0.9	&2.28E+2	&1.7	&4.24E-1	&0.8\\
16	&7.04E-2	&1.0	&1.02E-3	&1.8	&5.62E-7	&1.5	
&2.51E+3	&1.0	&6.15E+1	&1.9	&2.27E-1	&0.9\\
32	&3.55E-2	&1.0	&2.68E-4	&1.9	&1.58E-7	&1.8
&1.27E+3	&1.0	&1.58E+1	&2.0	&1.18E-1	&0.9\\
64	&1.78E-2	&1.0	&6.80E-5	&2.0	&4.09E-8	&2.0	
&6.37E+2	&1.0	&3.98	&2.0	&6.01E-2	&1.0\\
128	&8.91E-3	&1.0	&1.71E-5	&2.0	&1.03E-8	&2.0	
&3.19E+2	&1.0	&9.97E-1	&2.0	&3.04E-2	&1.0\\ \hline\hline
\end{tabular}
\end{table}
\item{\bf Test with various values in $\nu$.} In this test, we choose $\nu = 1, $ 1E-2, and 1E-4 to validate the robustness for the proposed numerical scheme.  Table~\ref{Tab:Ex1} reports the error profiles and convergence results. The performance can be summarized as below.
	\begin{itemize}
	\item The convergence orders for velocity obtained from Algorithm~\ref{Algorithm:WG1} and Algorithm~\ref{Algorithm:WG2} agree with our theoretical conclusions. We have the errors measured in $\3bar\be_h\3bar$, $\|\be_0\|$ converging at the orders $\mathcal{O}(h)$, $\mathcal{O}(h^2)$, and $\mathcal{O}(h)$ respectively.
	
	\item As we reduce the values in $\nu$, Algorithm~\ref{Algorithm:WG1} preserve the accuracy in velocity. However, the velocity error delivered by Algorithm~\ref{Algorithm:WG2} is increasing by a factor $1/\nu$, though the convergence rate is still preserved.
	\item In the pressure simulation, the numerical solution produced by Algorithm~\ref{Algorithm:WG2} preserve the same accuracy for various viscosity and converges at the order $\mathcal{O}(h)$.  In the opposite, Algorithm~\ref{Algorithm:WG1} produces the numerical pressure with super-convergence rate $\mathcal{O}(h^2)$. Besides, as we reduce the values in $\nu$, the error in pressure is also reduced by $\nu$.
	\item Thus, the new pressure-robust scheme provides a significant better simulation for small viscosity compared to the previous scheme Algorithm~\ref{Algorithm:WG2}.
	\end{itemize}

\begin{figure}[H]
\centering
\begin{tabular}{cccc}
\includegraphics[width=0.21\textwidth,height=0.21\textwidth]{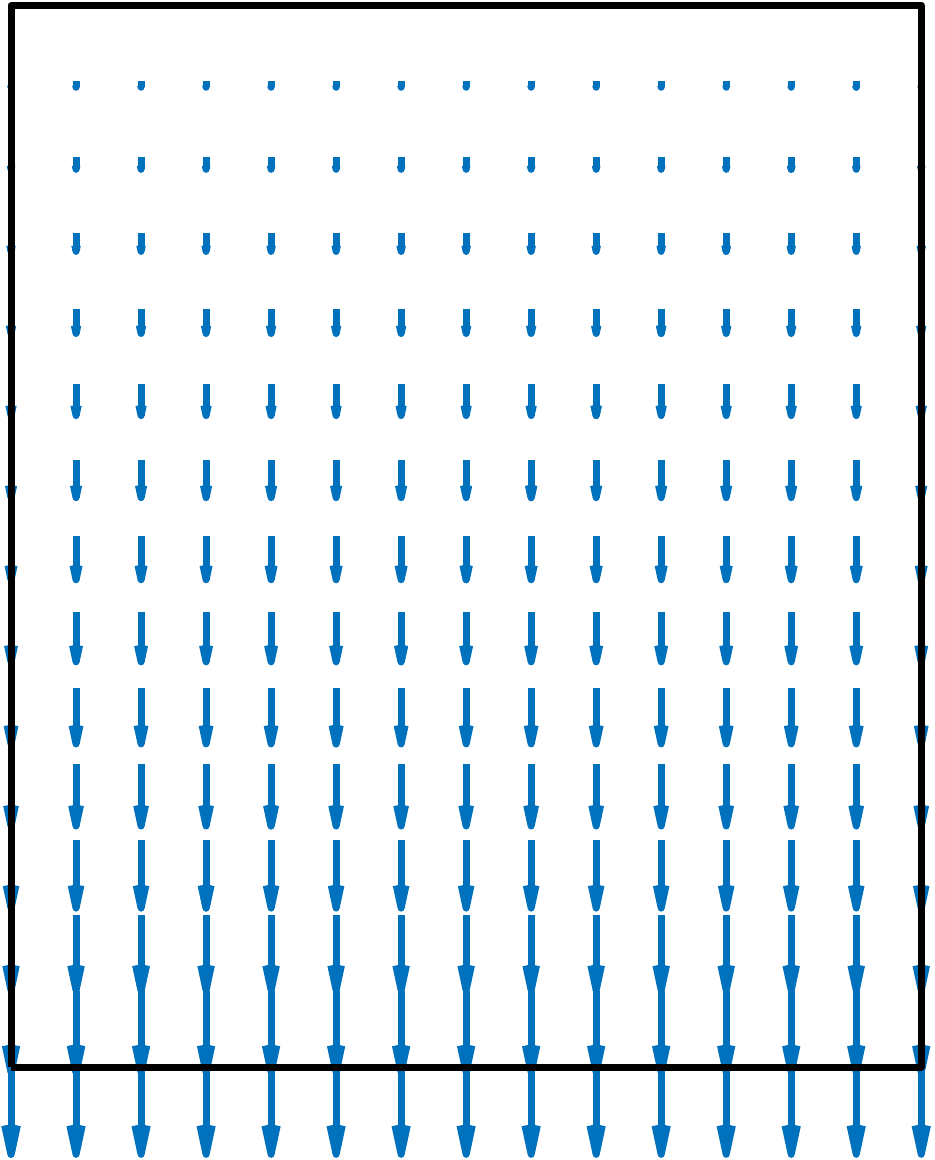}
&\includegraphics[width=0.21\textwidth,height=0.21\textwidth]{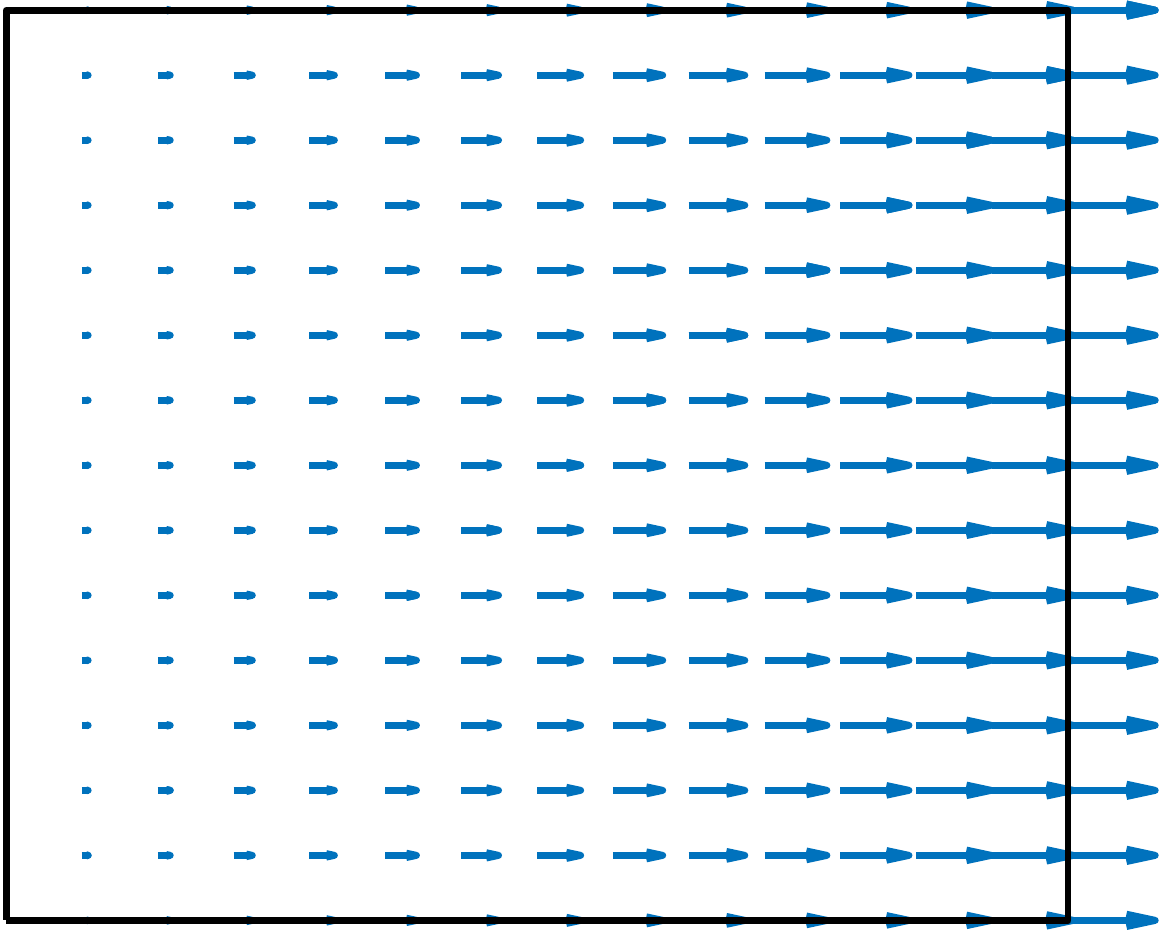}
&\includegraphics[width=0.21\textwidth,height=0.21\textwidth]{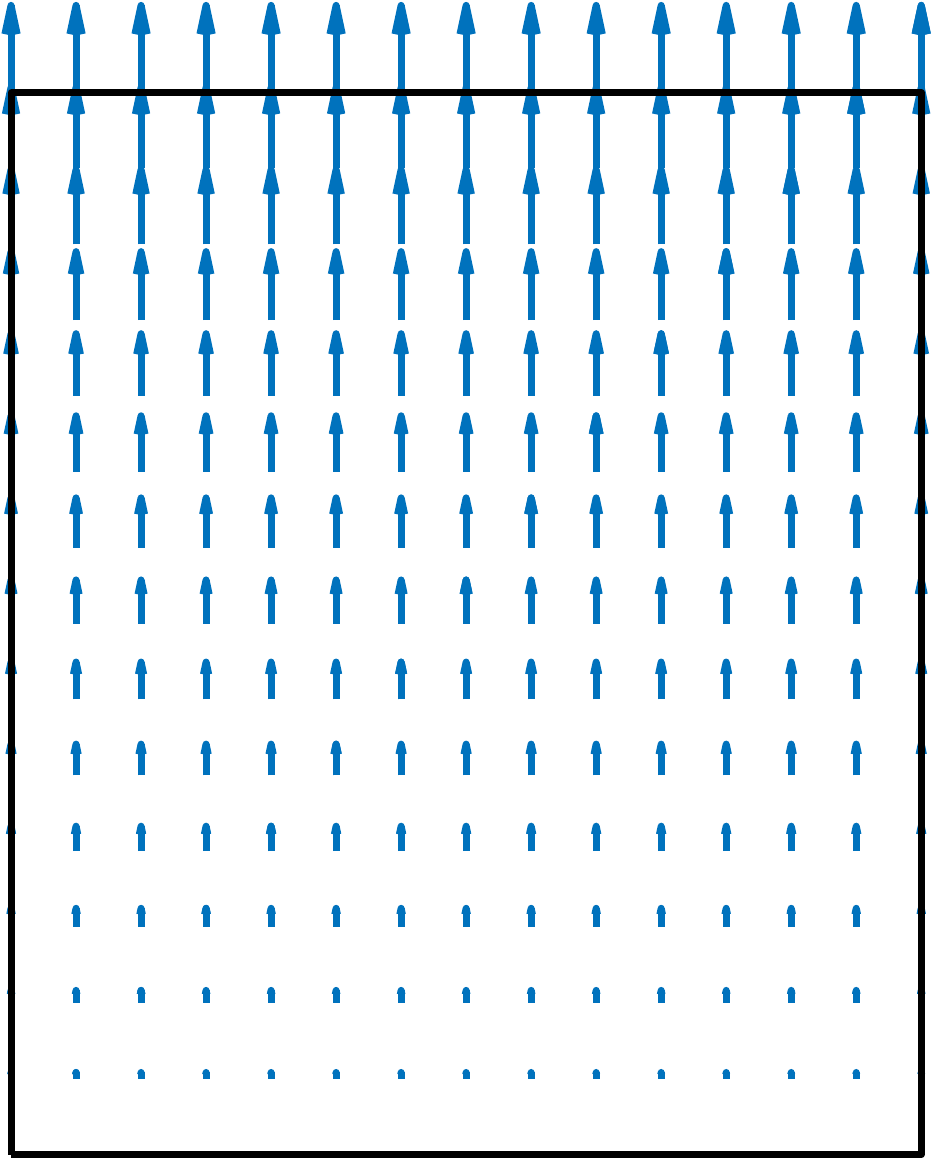}
&\includegraphics[width=0.21\textwidth,height=0.21\textwidth]{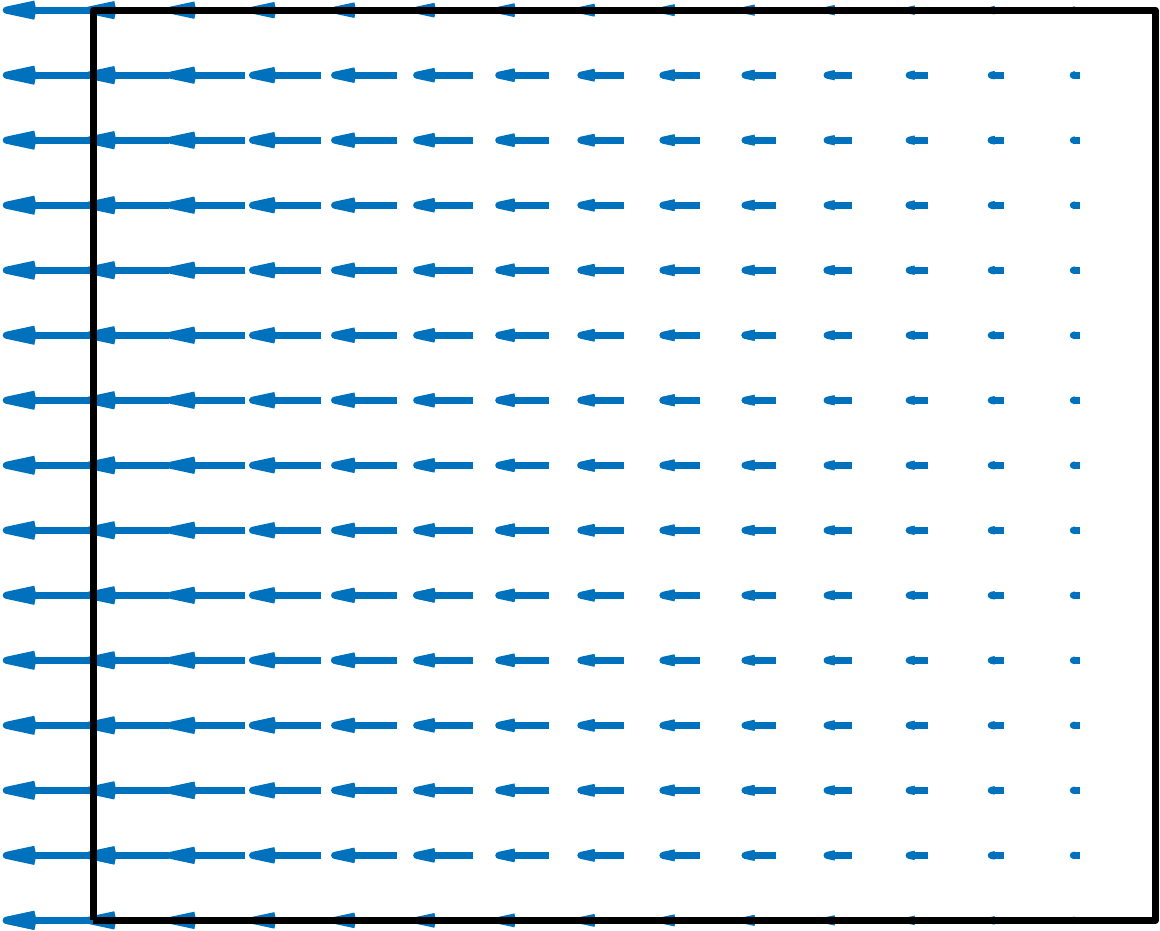}
\\
(a) & (b) & (c) &(d)\\
\includegraphics[width=0.21\textwidth,height=0.21\textwidth]{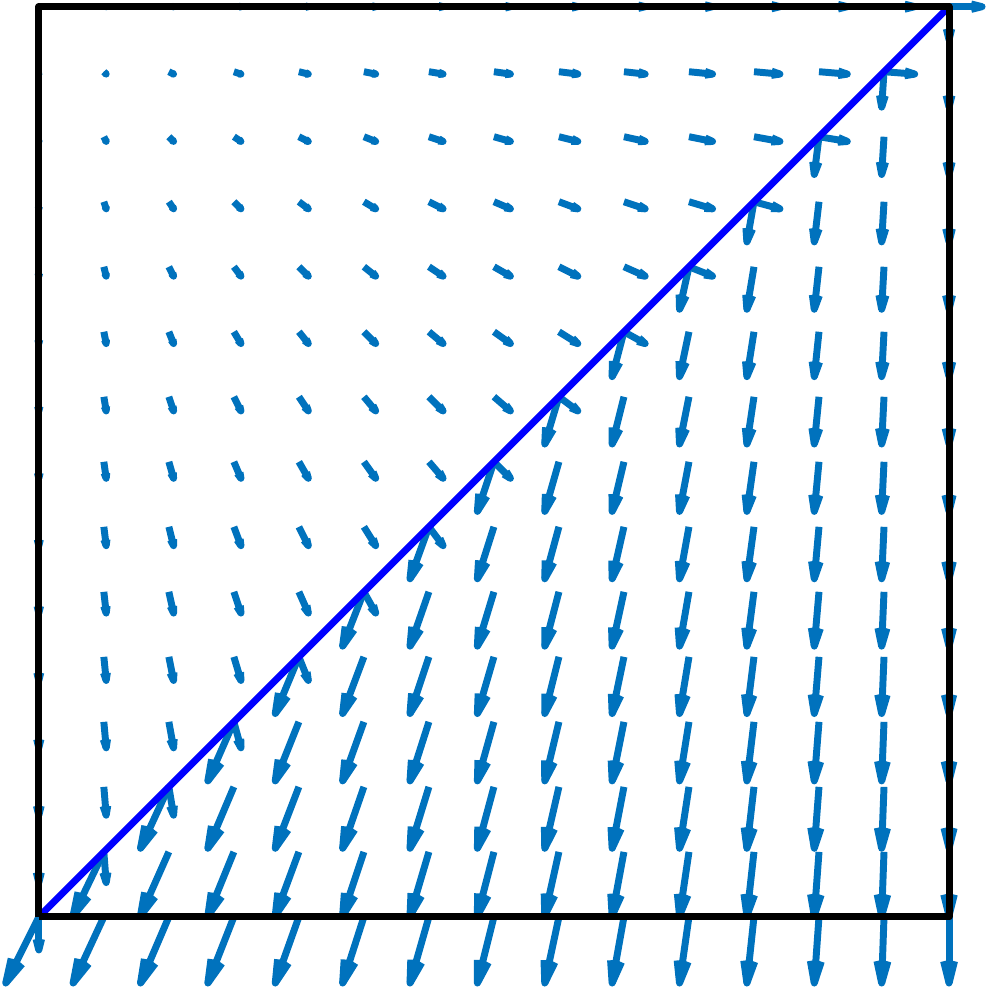}
&\includegraphics[width=0.21\textwidth,height=0.21\textwidth]{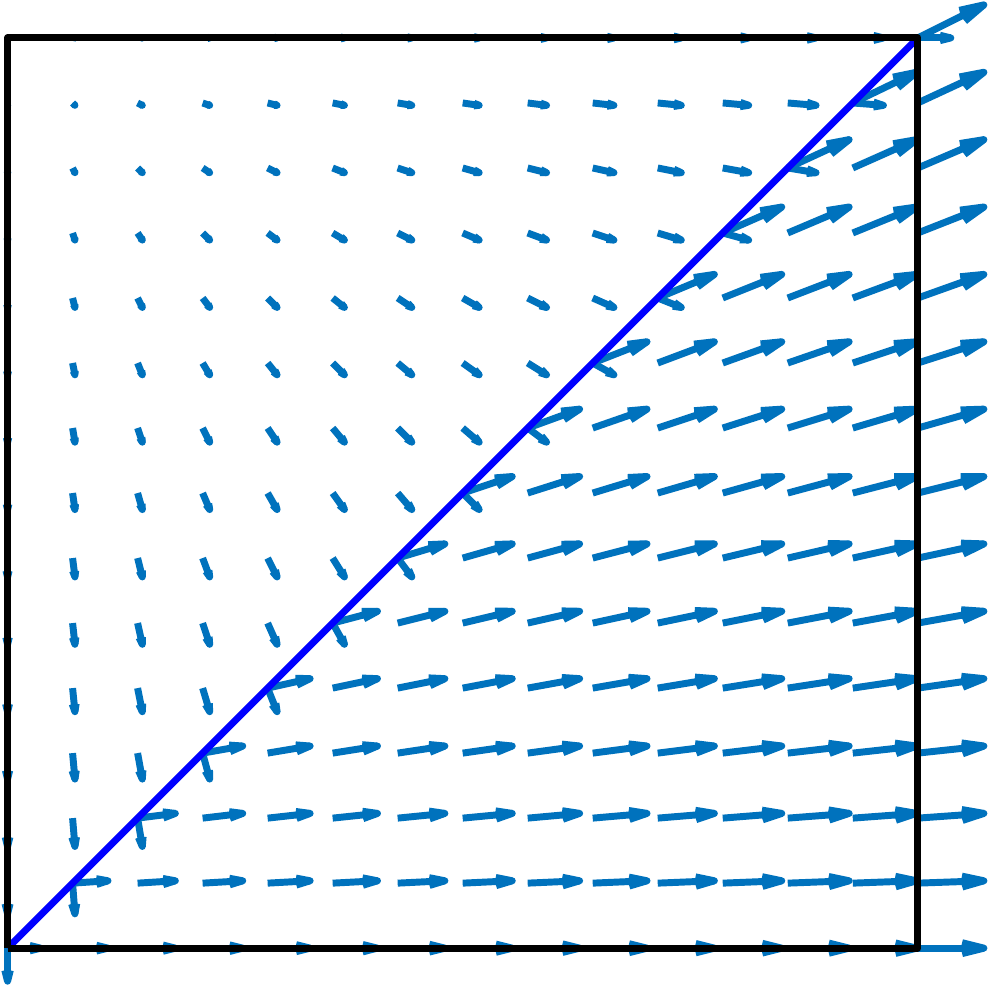}
&\includegraphics[width=0.21\textwidth,height=0.21\textwidth]{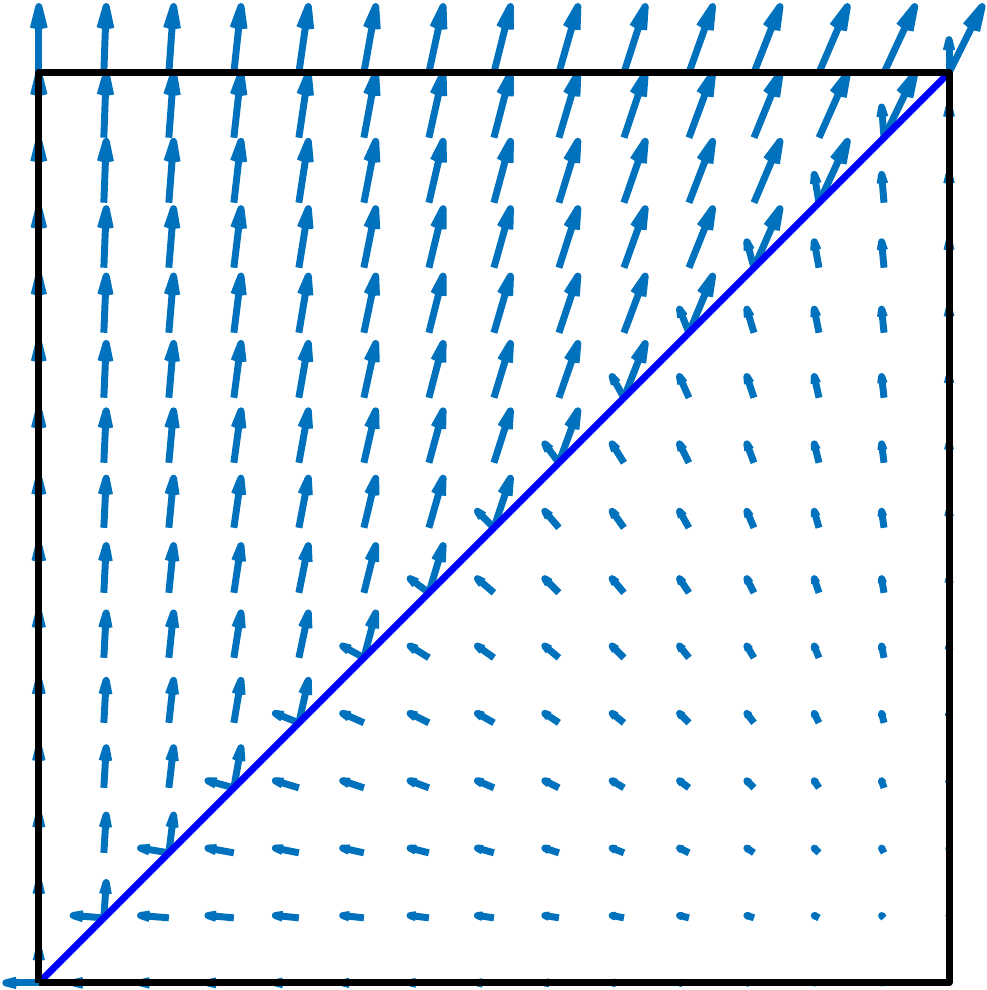}
&\includegraphics[width=0.21\textwidth,height=0.21\textwidth]{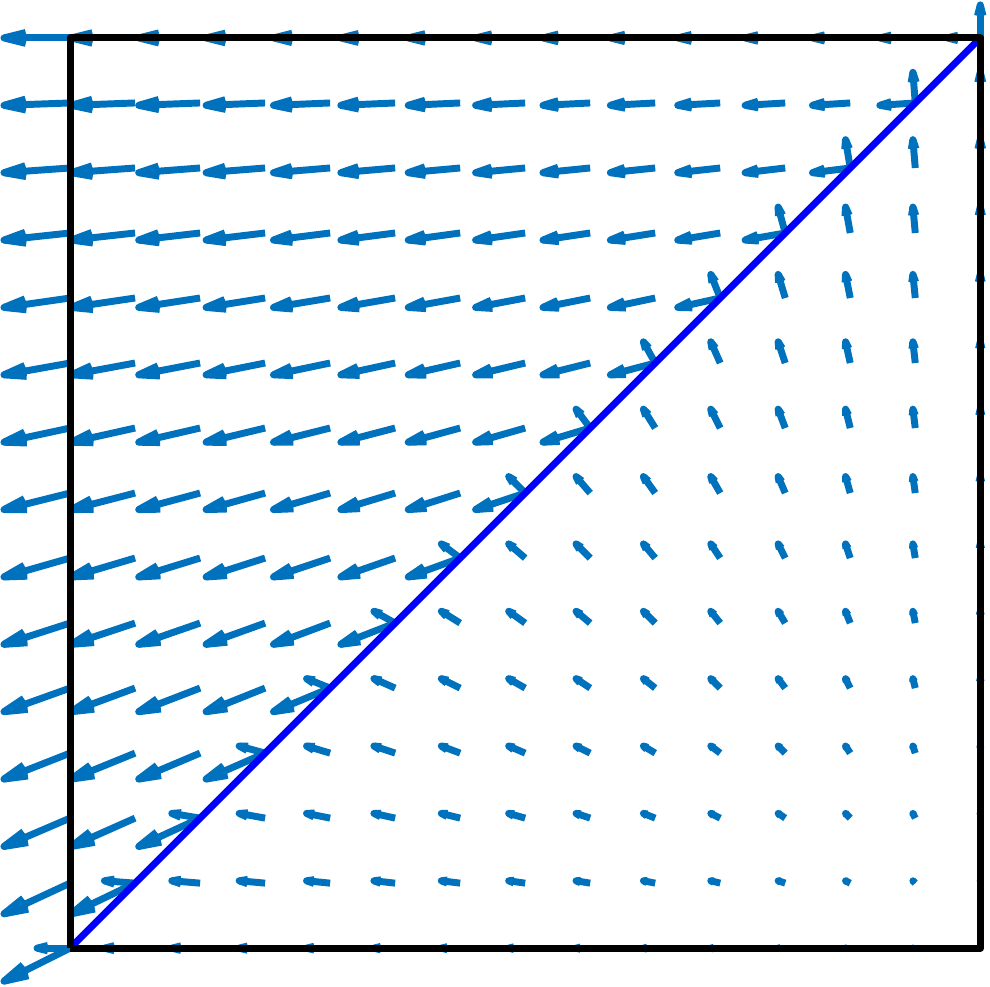}\\
(e) & (f) & (g) &(h)
\end{tabular}
\caption{Example~\ref{Num-1}. Illustration of the bases for RT$_{[0]}$(T) and $\Lambda_0(T)$.} \label{fig:ex1_quiver2_rect}
\end{figure}

\begin{figure}[H]
\centering
\includegraphics[width=0.5\textwidth]{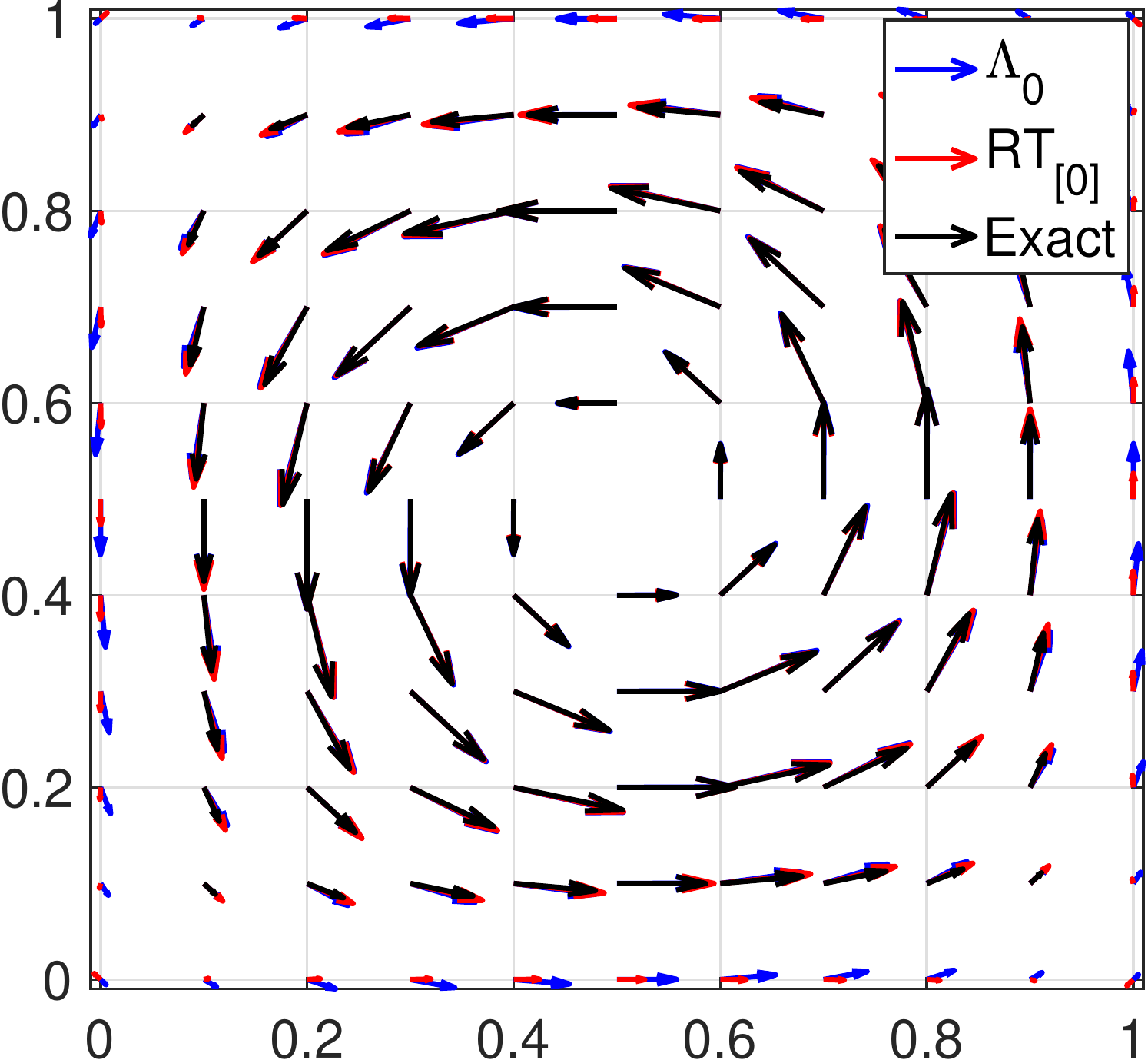}
\caption{Example~\ref{Num-1}. Vector plots for employing $\Lambda_0$, RT$_{[0]}$ spaces in the velocity reconstruction and exact solution on the mesh with $h=1/10$ and $\nu = 10^{-10}$.} \label{fig:ex1_quiver_RTLambda}
\end{figure}

\begin{table}[H]
\caption{Example~\ref{Num-1}. Error profiles for $\Lambda_0$ and RT$_{[0]}$ elements for various values in $\nu$.}\label{Tab:Ex1-2}
\centering
\tabcolsep=3pt
\begin{tabular}{l||ccc|ccc}\hline\hline
&\multicolumn{3}{c|}{$\Lambda_0$}&\multicolumn{3}{c}{$\text{RT}_{[0]}(T)$}\\ \hline
& $\3bar\be_h\3bar$ & $\|\be_0\|$ & $\|\epsilon_h\|$ & $\3bar\be_h\3bar$ & $\|\be_0\|$ & $\|\epsilon_h\|$\\ \hline
$\nu = 1$&2.8461E-2   &1.7287E-4   &1.0284E-3 
&4.3881E-2   &4.5269E-4   &6.7309E-3\\ \hline
$\nu = 1$E-2 &2.8461E-2   &1.7287E-4   &1.0284E-5
&4.3881E-2   &4.5269E-4   &6.7309E-5
\\ \hline
$\nu = 1$E-4 &2.8461E-2   &1.7287E-4   &1.0284E-7
&4.3881E-2   &4.5269E-4   &6.7309E-7
\\ \hline
$\nu = 1$E-6 &2.8461E-2   &1.7287E-4   &1.0284E-9
&4.3881E-2   &4.5269E-4   &6.7309E-9
\\ \hline
$\nu = 1$E-8 &2.8461E-2   &1.7287E-4   &1.0283E-11
&4.3881E-2   &4.5269E-4   &6.7310E-11
\\ \hline
$\nu = 1$E-10 &2.8461E-2   &1.7287E-4   &1.0284E-13
&4.3881E-2   &4.5265E-4   &7.6583E-13
\\ \hline\hline
\end{tabular}
\end{table}
\item{\bf Comparison with RT$_{[0]}(T)$ space on the rectangular mesh.}
In this test, we compare the numerical performance of employing the rectangular RT element\cite{ABF2005} with $k=0$ (denoted as RT$_{[0]}$) and our proposed projection space $\Lambda_0(T)$. We compute the velocity reconstructions $\Pi_h\bv$, respectively, into the spaces RT$_{[0]}(T)$ and $\Lambda_0(T)$.   Four basis functions on unit square are plotted in Figure~\ref{fig:ex1_quiver2_rect}. It is noted that the constructed basis in the $\Lambda_0(T)$ space is piecewise polynomial with normal component continuous cross the interior edge. We plot the numerical vector field in Figure~\ref{fig:ex1_quiver_RTLambda} on the mesh with $h=1/10$ and $\nu = 10^{-10}$ for the $\Lambda_0$ and RT$_{[0]}$ version. The numerical vector fields generated by employing $\Lambda_0$ and RT$_{[0]}$ spaces agree well with the exact vector field. This observation validates that both of the tests are pressure robust. Furthermore, the error profiles for various values in $\nu$ are reported in Table~\ref{Tab:Ex1-2}. It can be validated again from this table that by employing the velocity reconstruction both into $\Lambda_0$ and RT$_{[0]}$ spaces will produce the pressure robust simulation. Also the errors in the $\Lambda_0$ version are slightly better than that of the RT$_{[0]}$ space.

\begin{figure}[H]
\centering
\begin{tabular}{cc}
\includegraphics[width=0.45\textwidth]{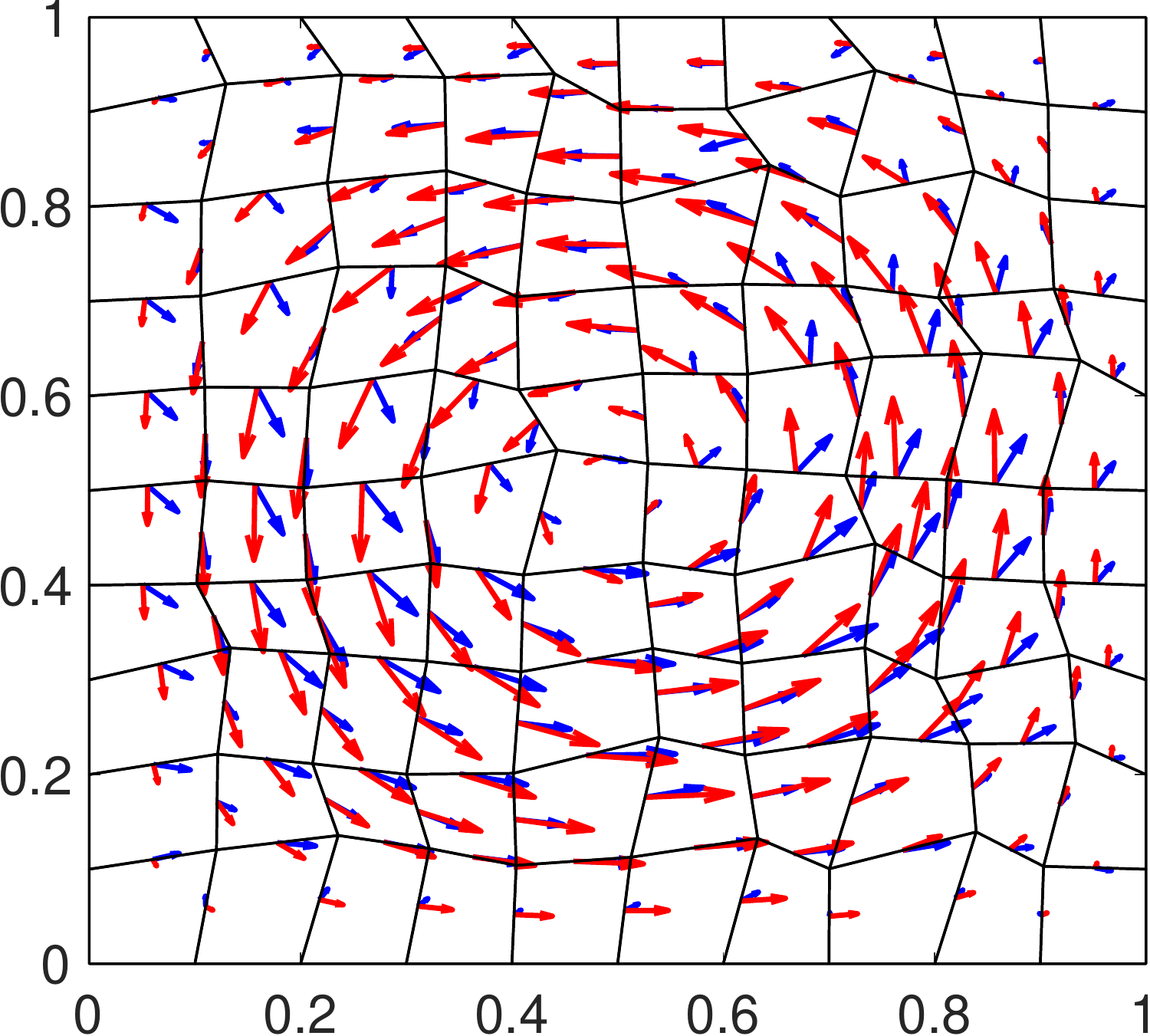}
&\includegraphics[width=0.45\textwidth]{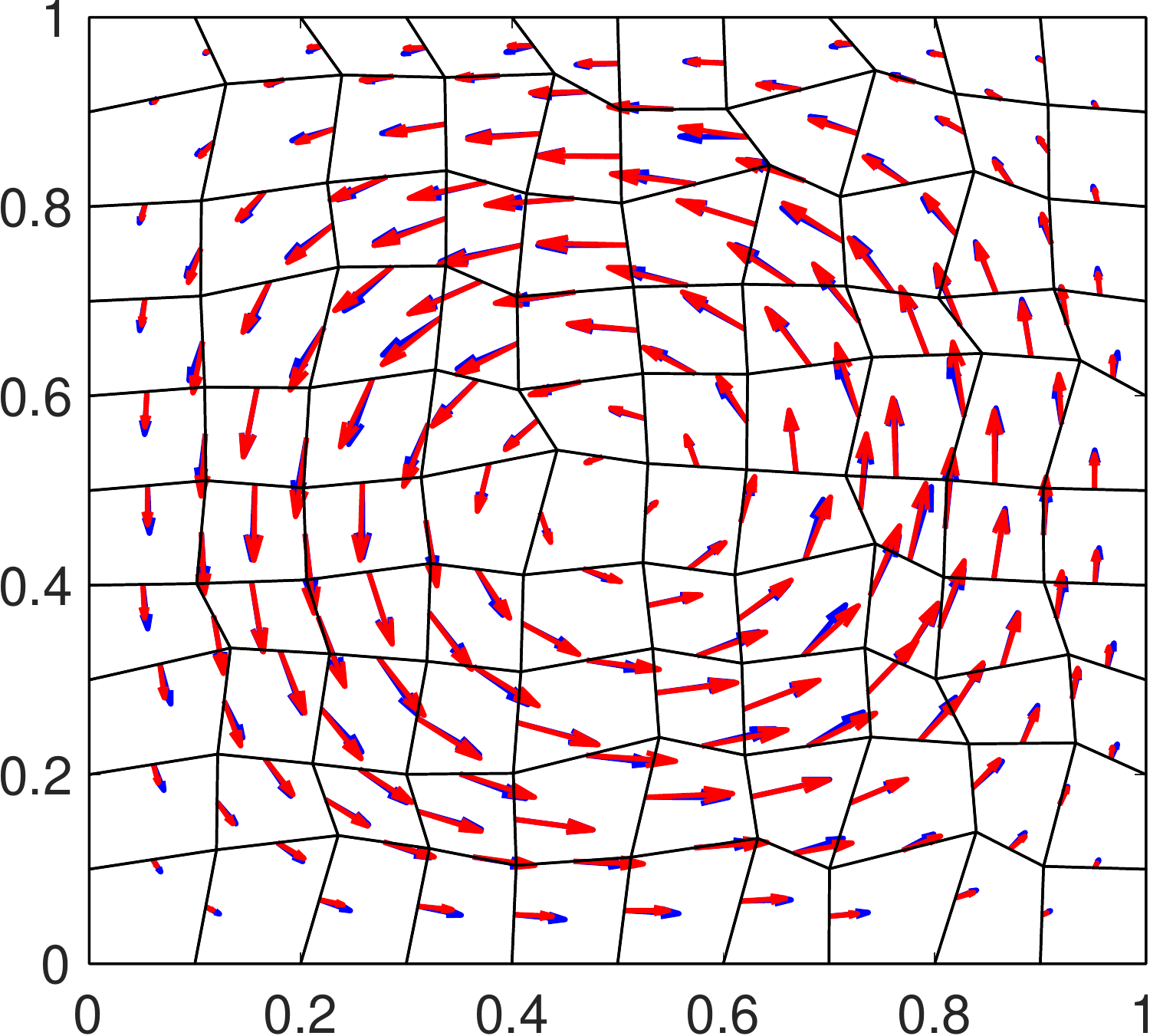}
\\
(a) & (b) 
\end{tabular}
\caption{Example~\ref{Num-1}. Plot of the velocity vector fields. Red arrow denotes the exact solution; Blue arrow plots the numerical solution. (a) Algorithm~\ref{Algorithm:WG2}; (b) Algorithm~\ref{Algorithm:WG1}.} \label{fig:ex1_quiver2}
\end{figure}
\item\textbf{Test case with $\nu = 1$ on deformed rectangular grid.} In this test, we shall perform WG Algorithm~1-2 on the deformed rectangular grid and compare the numerical performance on the coarse mesh. Let $k = 0$ and $h = 1/10$ and the numerical solutions on the uniform rectangular mesh are plotted in Figure~\ref{fig:ex1_quiver2}. 
The pressure robust scheme improves the velocity simulation and preserves very well the exact velocity vector field.

\end{itemize}

\subsection{Effects from the Approximation with Pressure}
Let $\Omega = (0,1)^2$ and the external body force ${\bf f}$ be chosen such that the exact solutions are
\begin{eqnarray}
\bu = 0,\text{ and }p(x,y) = \sum_{j=0}^7 x^jy^{7-j} - \frac{761}{1260}.
\end{eqnarray}
In this test, we perform WG Algorithm~\ref{Algorithm:WG1} on the deformed rectangular mesh and the polygonal mesh  shown in Figure~\ref{Fig:Num2-grid}. We shall employ the velocity reconstruction operator into the $\Lambda_0$ and $\mathbb{CW}_0$ spaces and check the corresponding robustness. The scheme with respect to $\mathbb{CW}_0$ space was investigated in \cite{mu}.
\begin{figure}[H]
\centering
\begin{tabular}{cc}
\includegraphics[width=0.45\textwidth,height=.45\textwidth]{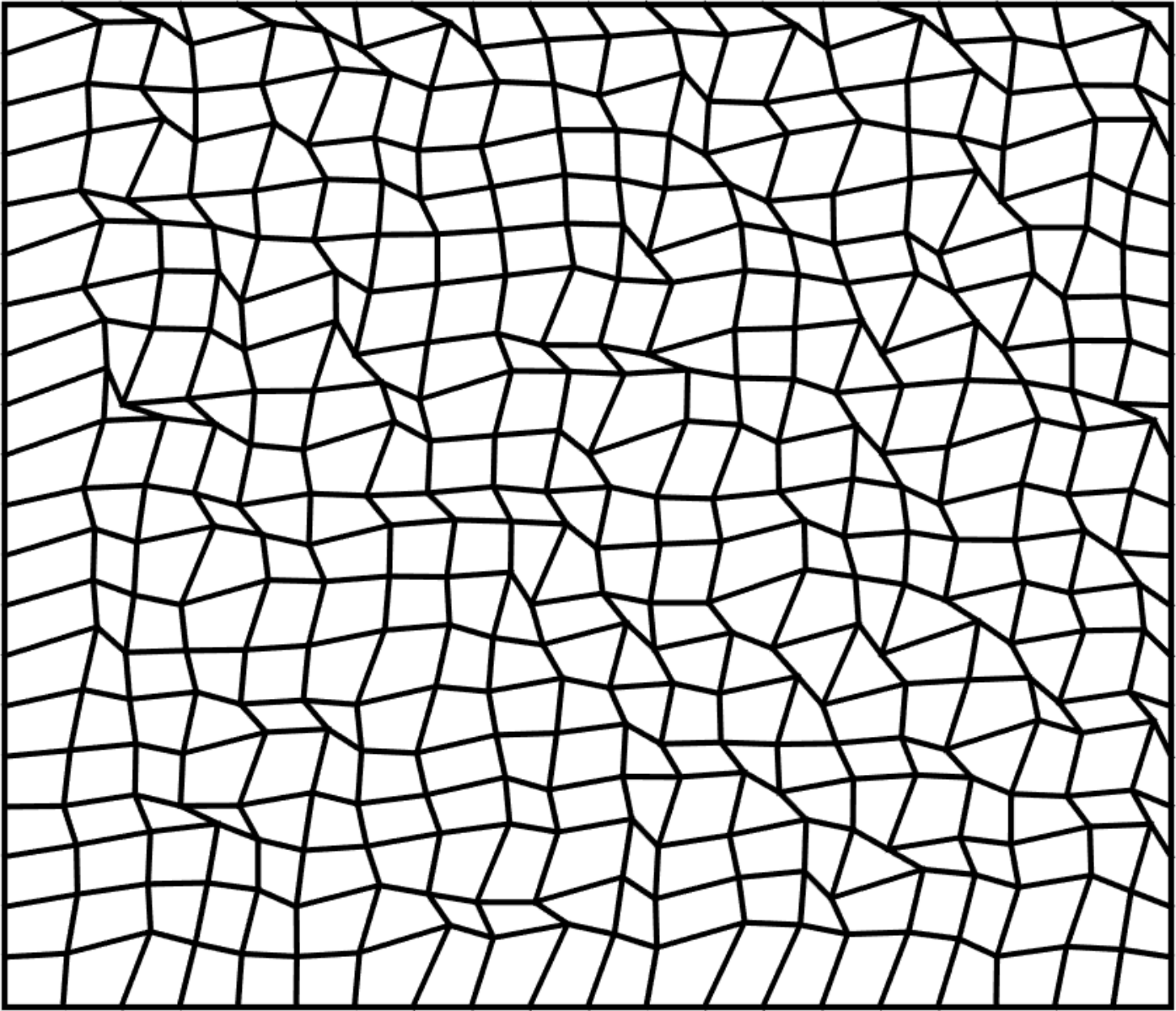}
&\includegraphics[width=0.45\textwidth,height=.45\textwidth]{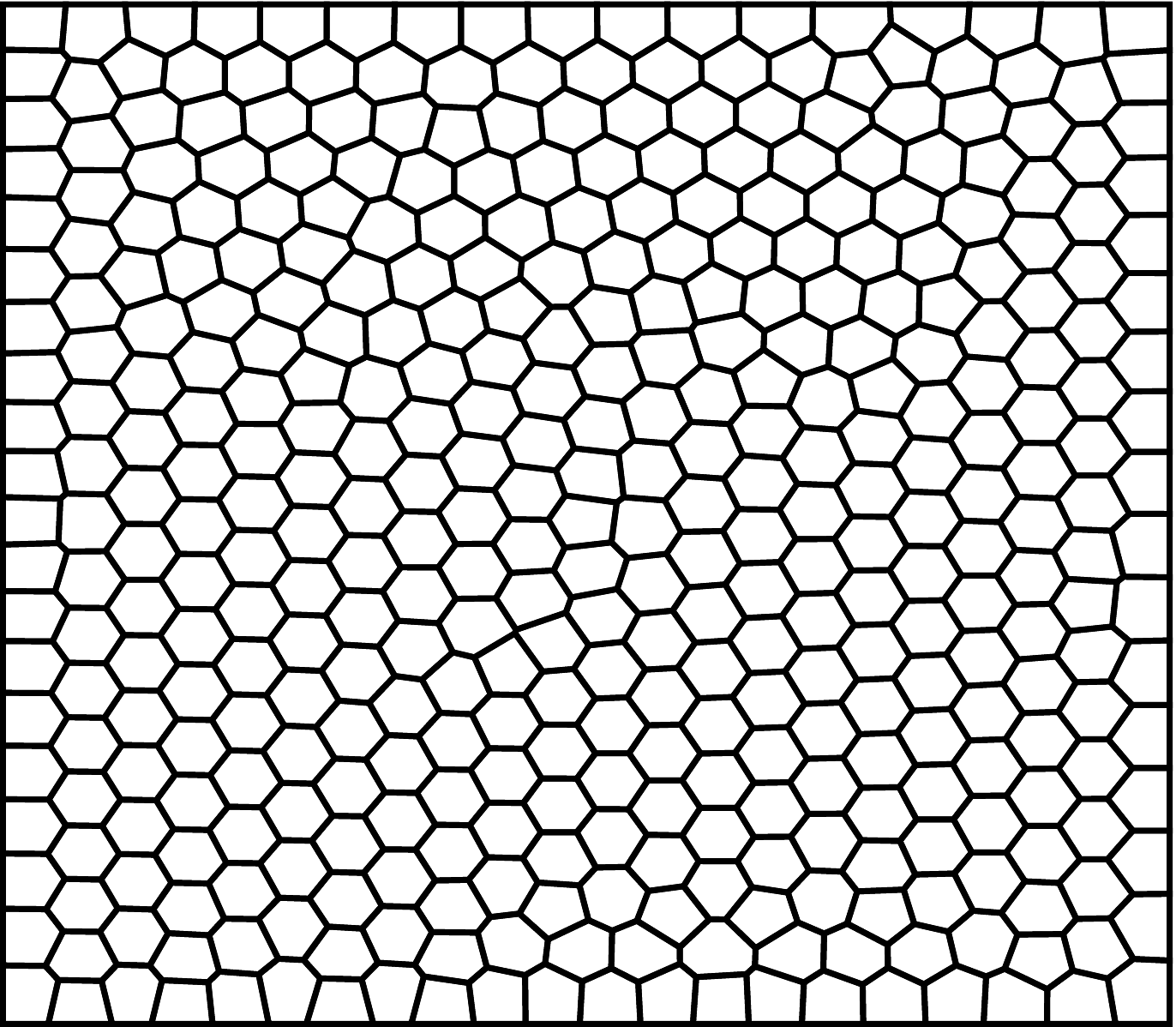}
\\
(a) & (b) 
\end{tabular}
\caption{Example~\ref{Sect:Num-2}. illustration of the polygonal grids. (a) Deformed rectangular grid; (b) Polygonal grid.}\label{Fig:Num2-grid}
\end{figure}

\subsubsection{Test - deformed rectangular mesh}\label{Sect:Num-2}
In this section, we perform the weak Galerkin Algorithm on the deformed rectangular mesh (Figure~\ref{Fig:Num2-grid}a). Since this mesh contains non-convex cell, the $\mathbb{CW}_0$ basis can not be applied as the velocity reconstruction space. However, our proposed $\Lambda_0$ space can be directly chosen for the velocity reconstruction.

The numerical results corresponding to Algorithm~\ref{Algorithm:WG2} and Algorithm~\ref{Algorithm:WG1} are plotted in Figure~\ref{Num2-rect-1}-\ref{Num2-rect-2}. We observe that:
\begin{itemize}
\item The non-pressure robust scheme produces a numerical velocity with magnitude at the order $\mathcal{O}(10^{-3})$. Thus, the error in pressure, in fact, gives rise to the inaccuracy in the velocity simulation. This visualizes the lack of pressure-robustness.
\item By only modifying the body source assembling, we can obtain the pressure robust results. The numerical solutions corresponding to Algorithm~\ref{Algorithm:WG1} are plotted in Figure~\ref{Num2-rect-2}. It is noted that the velocity simulation is almost zero and very closed to the expected zero flow, which is at the order $\mathcal{O}(10^{-18})$.
\item The above observations validate our enhanced discretization on the non-convex meshes. 
\end{itemize}

\begin{figure}[H]
\centering
\begin{tabular}{ccc}
\includegraphics[width=0.32\textwidth]{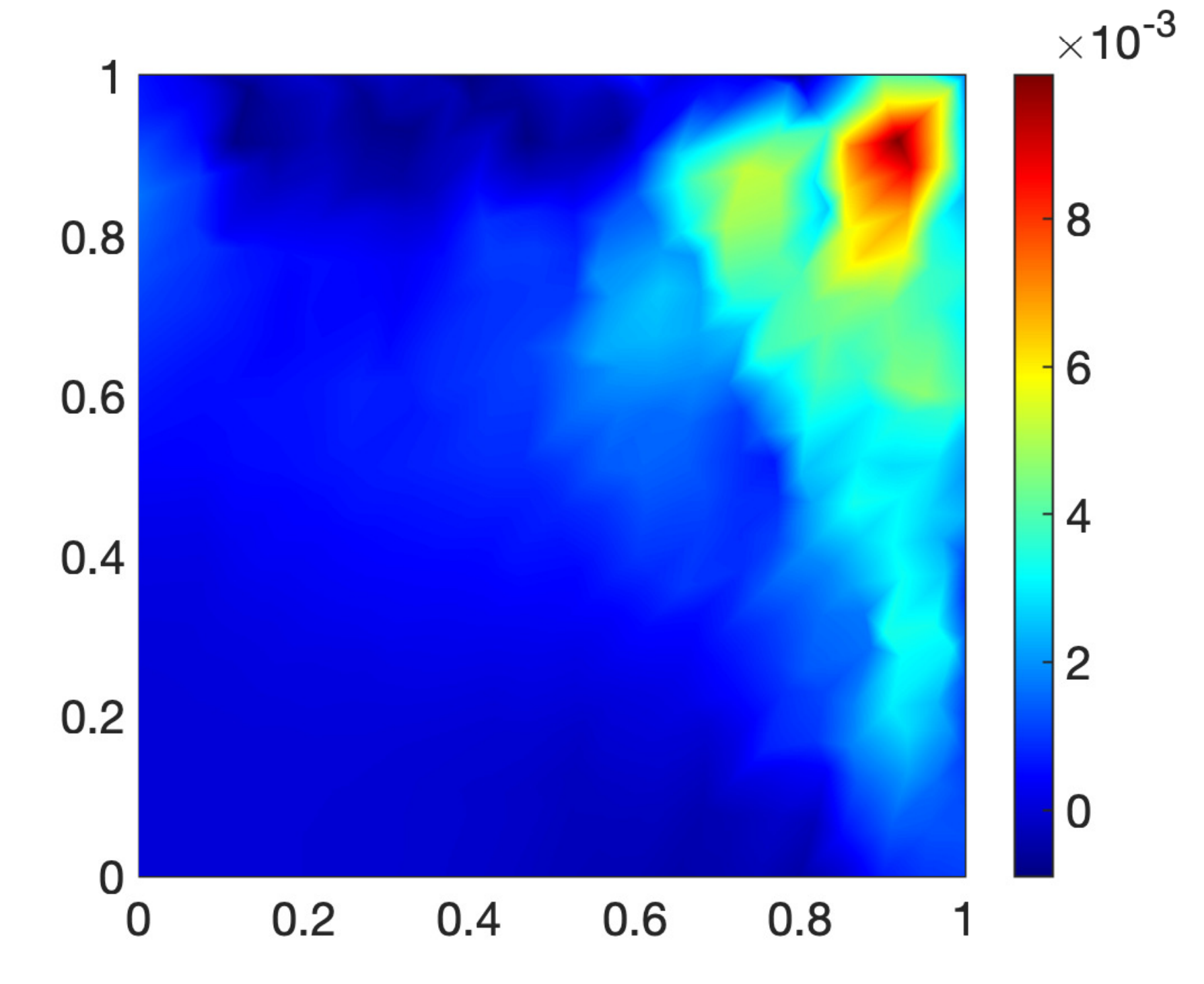}
&\includegraphics[width=0.32\textwidth]{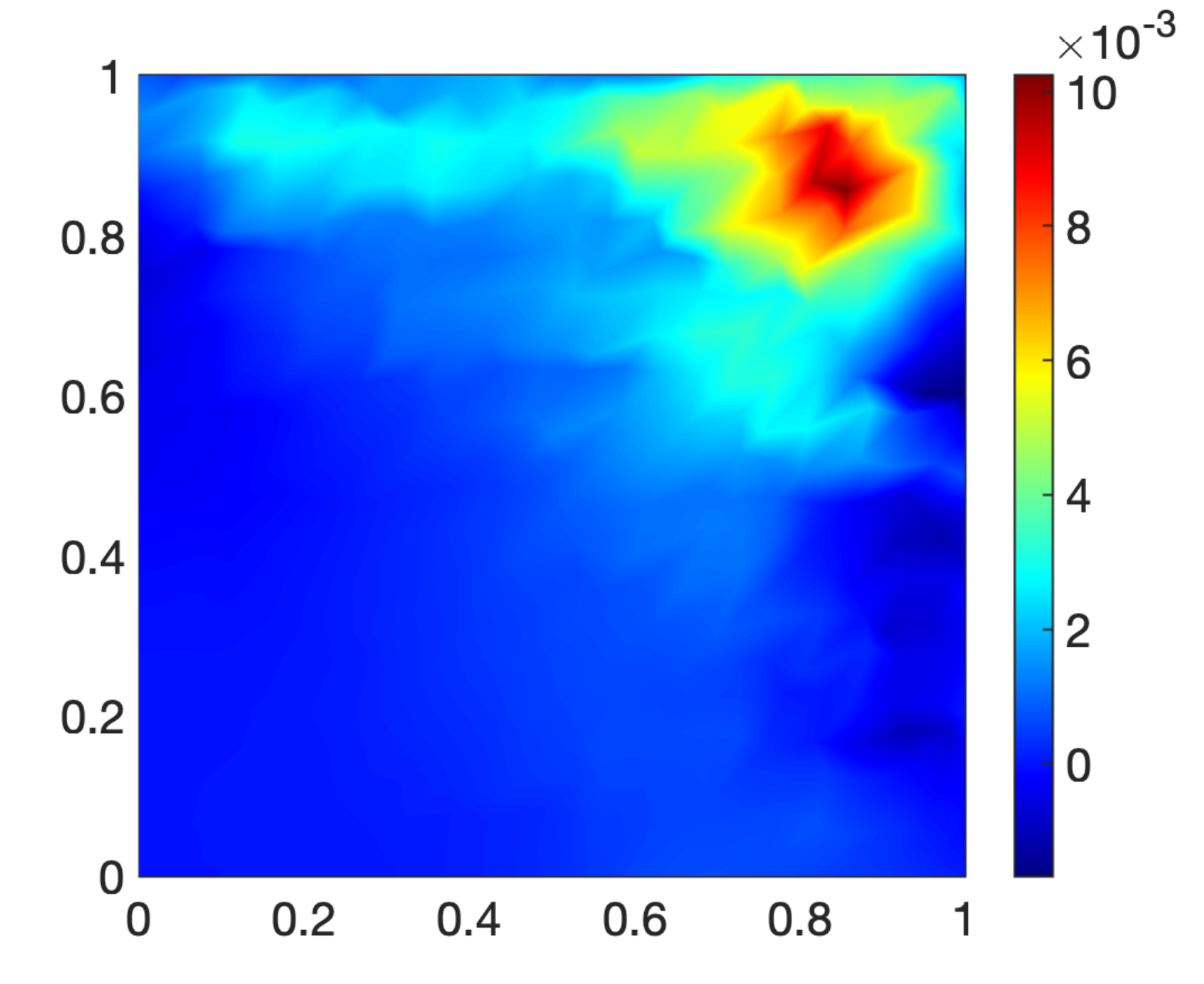}
&\includegraphics[width=0.31\textwidth]{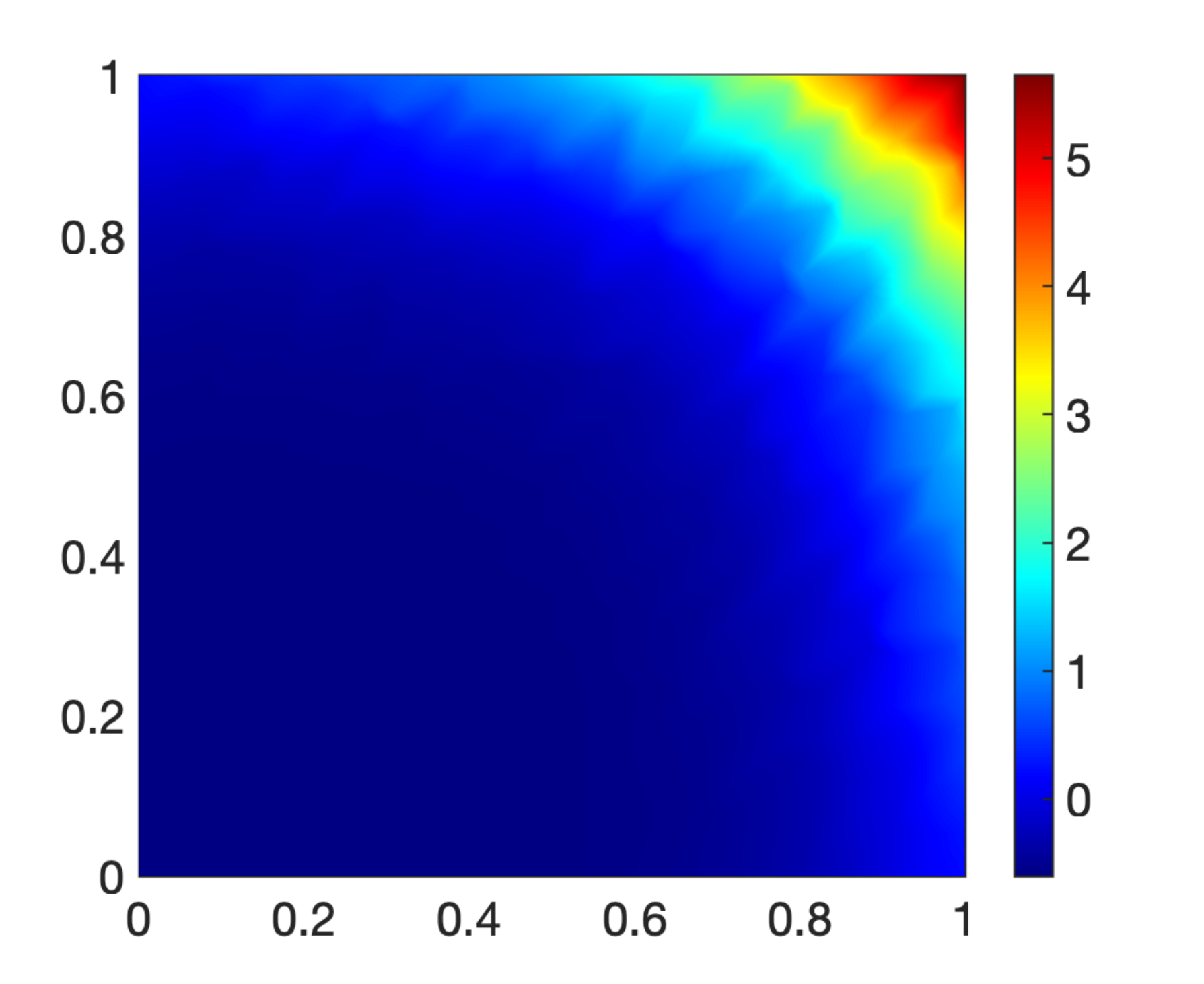}
\\
(a) & (b)  &(c)
\end{tabular}
\caption{Test~\ref{Sect:Num-2}. Illustration of numerical solutions from Algorithm~\ref{Algorithm:WG2} on the deformed rectangular grid. (a). x-component of velocity; (b) y-component of velocity; (c) pressure.}\label{Num2-rect-1}
\end{figure}

\begin{figure}[H]
\centering
\begin{tabular}{ccc}
\includegraphics[width=0.32\textwidth]{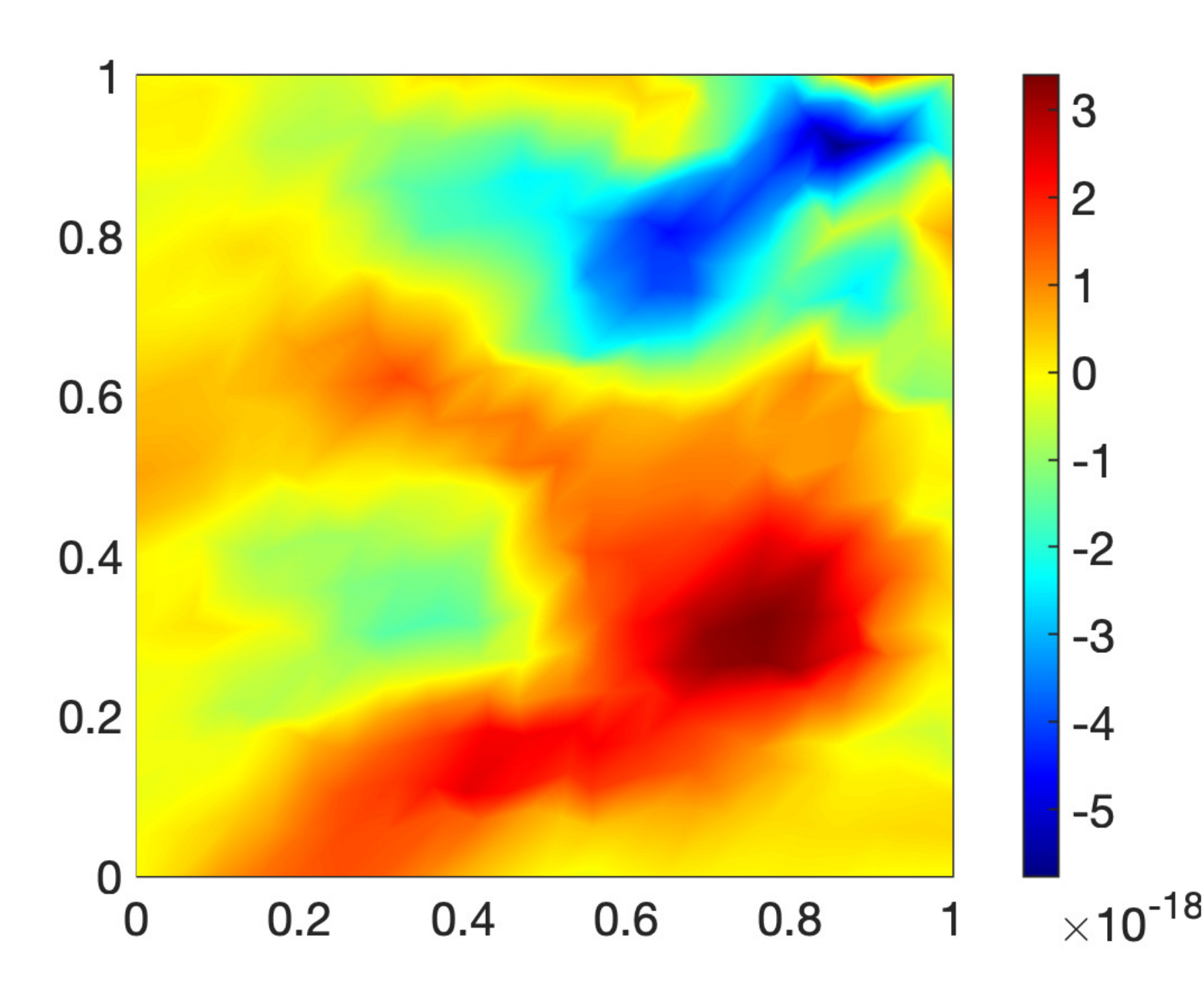}
&\includegraphics[width=0.32\textwidth]{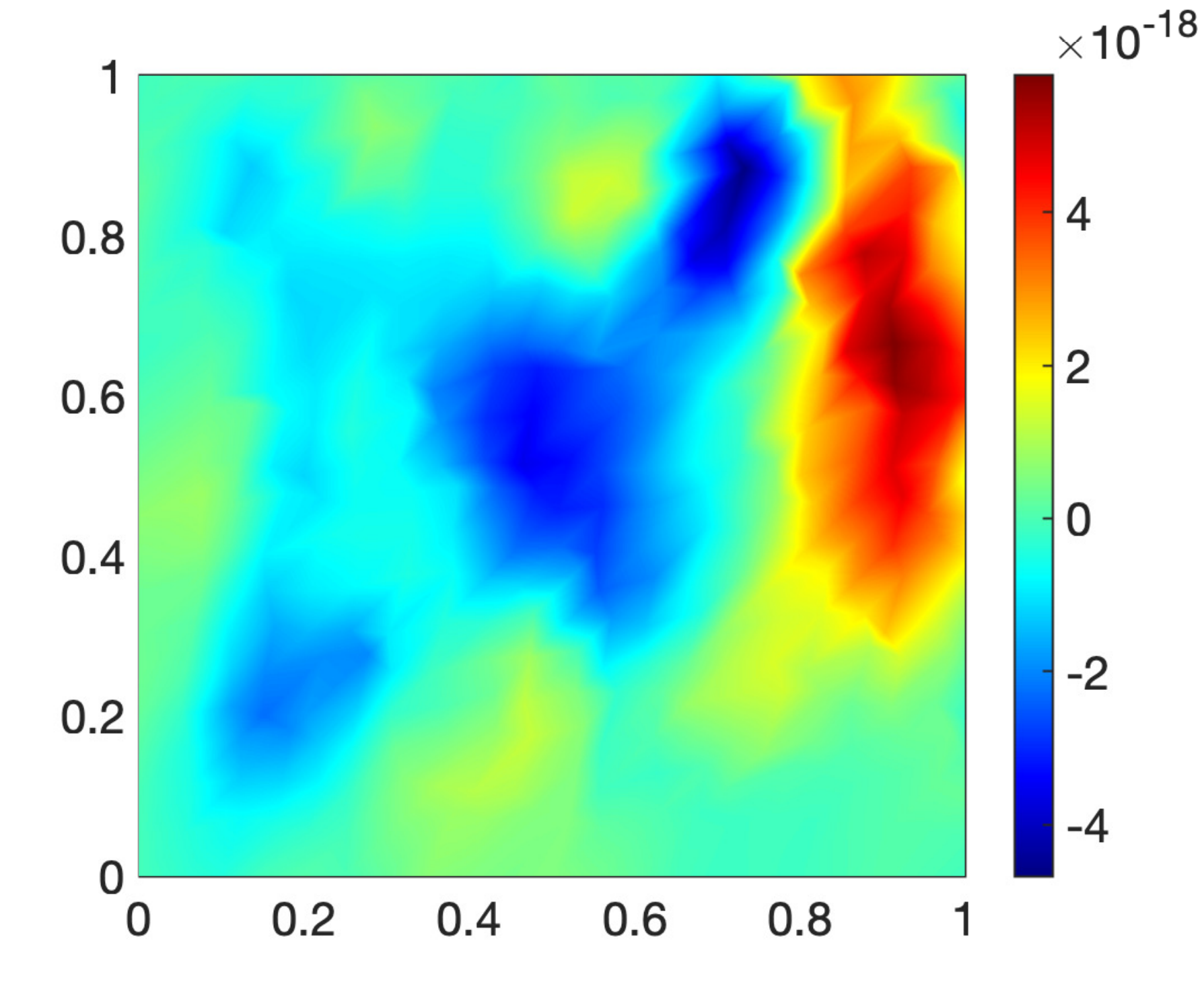}
&\includegraphics[width=0.31\textwidth]{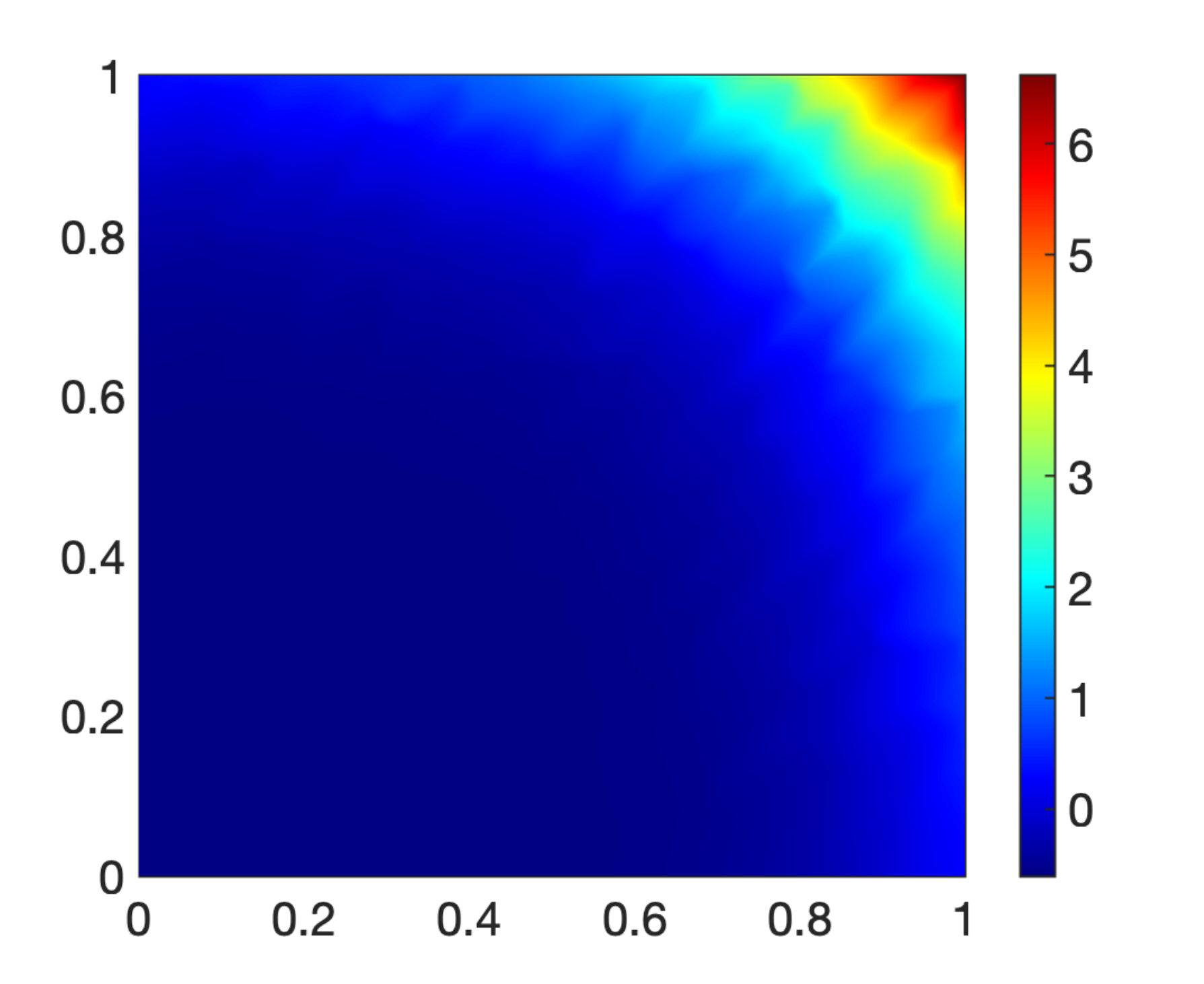}
\\
(a) & (b) & (c)
\end{tabular}
\caption{Test~\ref{Sect:Num-2}. Illustration of numerical solutions from Algorithm~\ref{Algorithm:WG1} on the deformed rectangular grid. (a). x-component of velocity; (b) y-component of velocity; (c) pressure.}\label{Num2-rect-2}
\end{figure}

\subsubsection{Test - polygonal mesh}\label{Sect:Num-2-2}
In this test, we shall perform the simulations on the polygonal grid and compare the numerical performance with the scheme in $\mathbb{CW}_0$ setting\cite{mu}. It is noted that $\mathbb{CW}_0$ basis on the polygonal mesh usually consists of rational functions and thus in some cases the numerical integration cannot calculated exactly. Such inaccuracy in numerical integration will introduce errors in velocity simulation even though the scheme is designed with pressure-robust property. This test will show that be employing the $\Lambda_0$ space we proposed in this paper, we can overcome such inaccuracy integral and achieve complete independency in the velocity simulation.

\begin{figure}[H]
\centering
\begin{tabular}{ccc}
\includegraphics[width=0.32\textwidth]{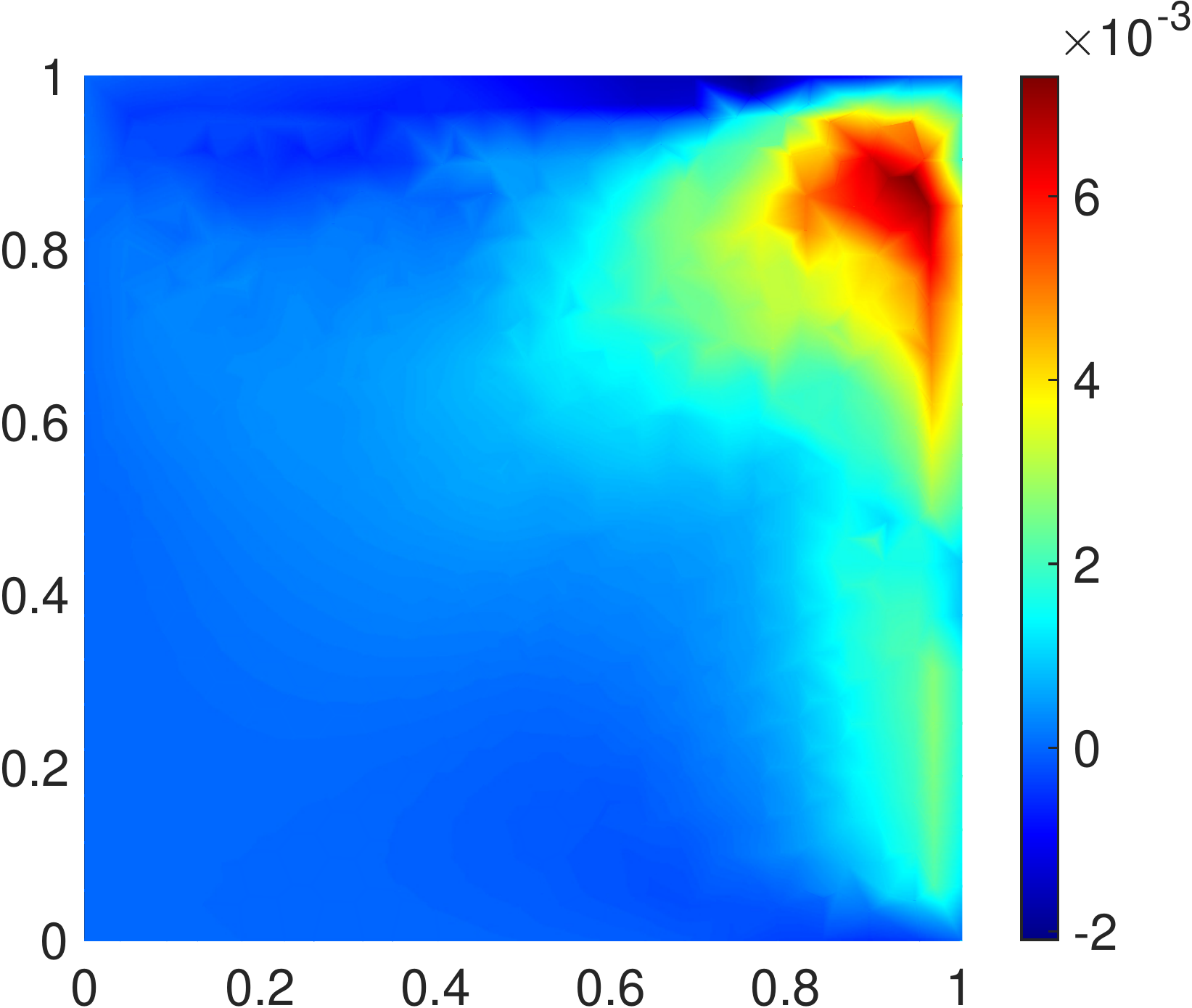}
&\includegraphics[width=0.32\textwidth]{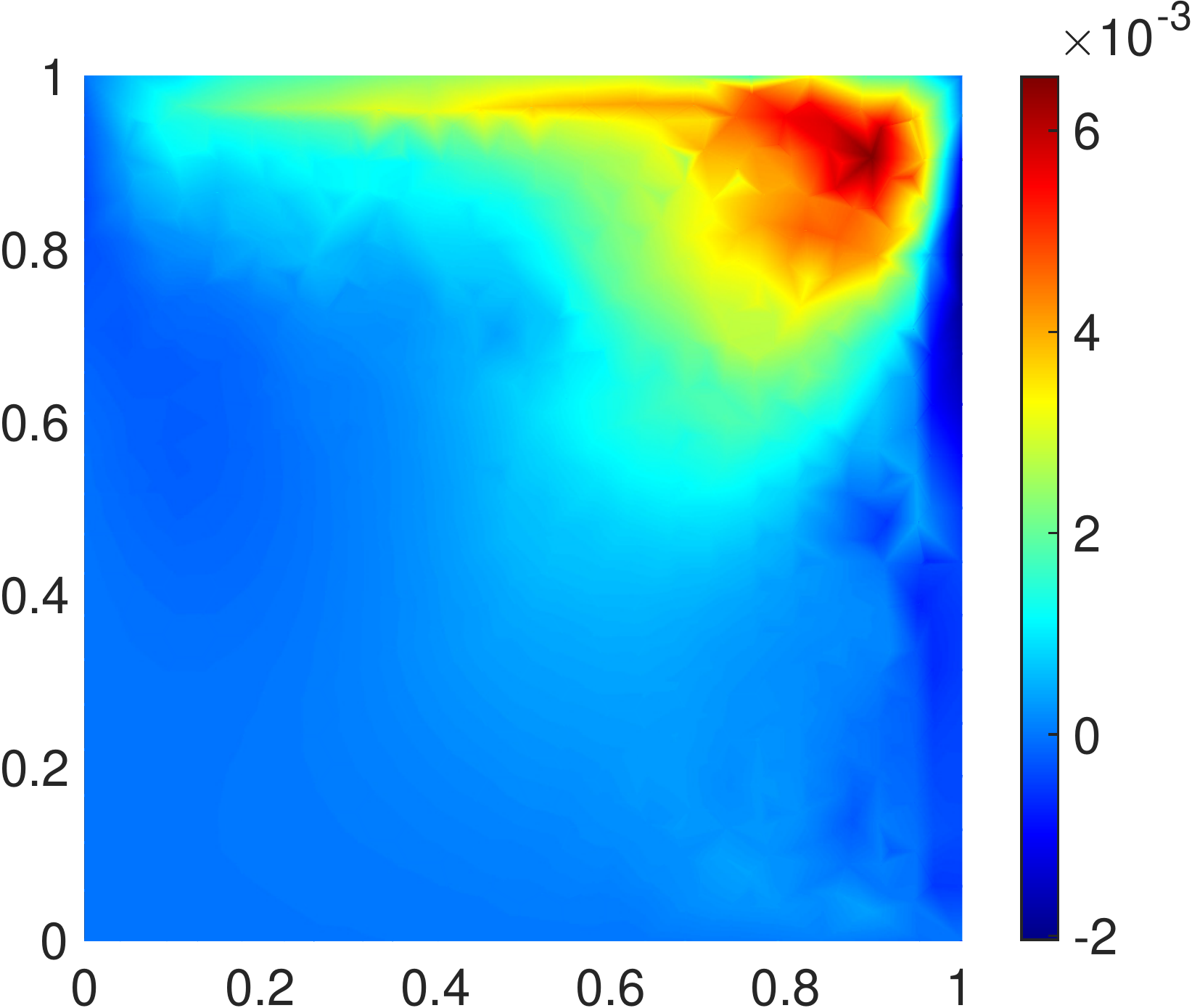}
&\includegraphics[width=0.31\textwidth]{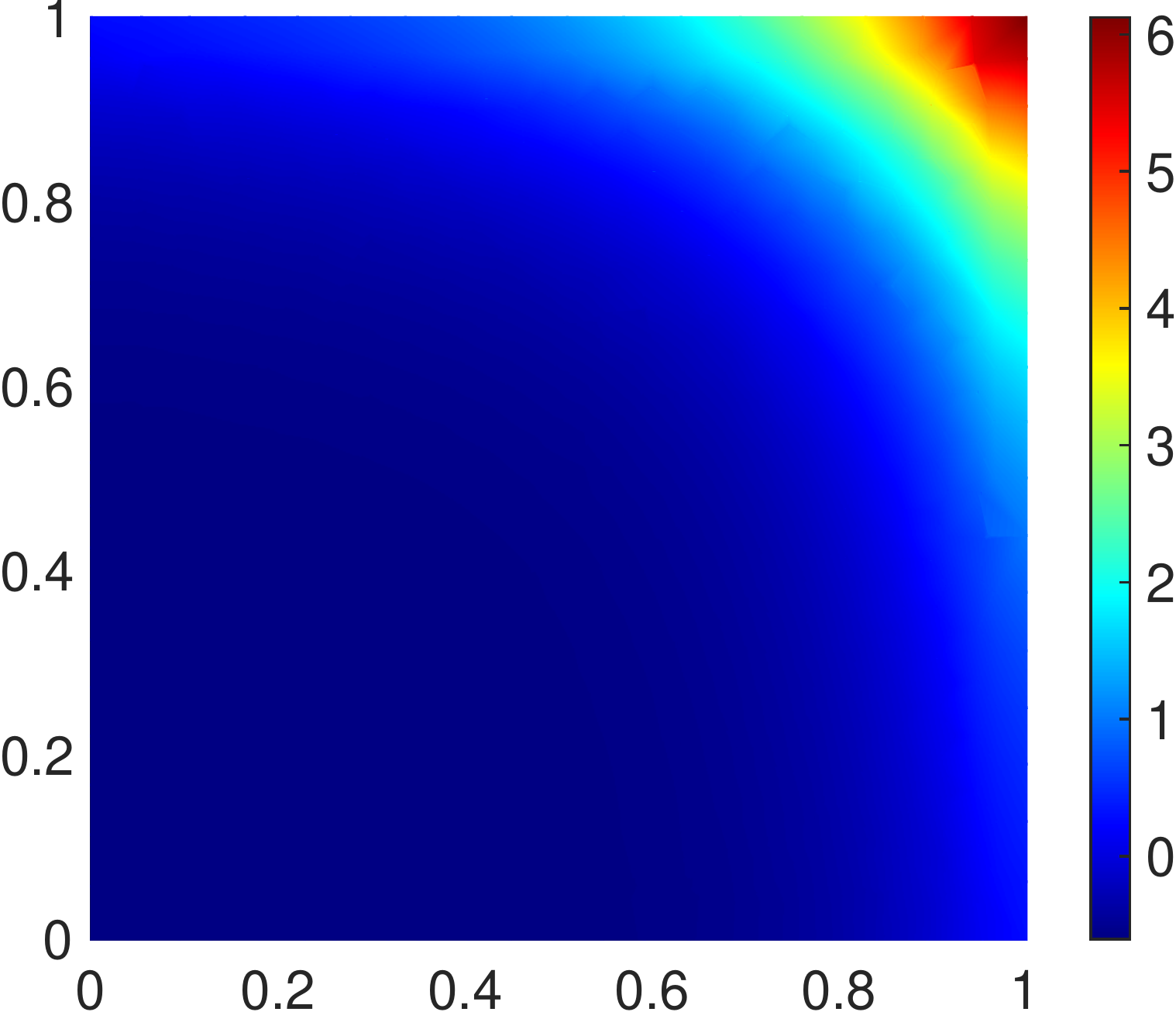}
\\
(a) & (b) & (c)
\end{tabular}
\caption{Test~\ref{Sect:Num-2-2}. Illustration of numerical solutions from Algorithm~\ref{Algorithm:WG2} on the polygonal grid. (a). x-component of velocity; (b) y-component of velocity; (c) pressure.}\label{Num2-1}
\end{figure}

\begin{figure}[H]
\centering
\begin{tabular}{ccc}
\includegraphics[width=0.32\textwidth]{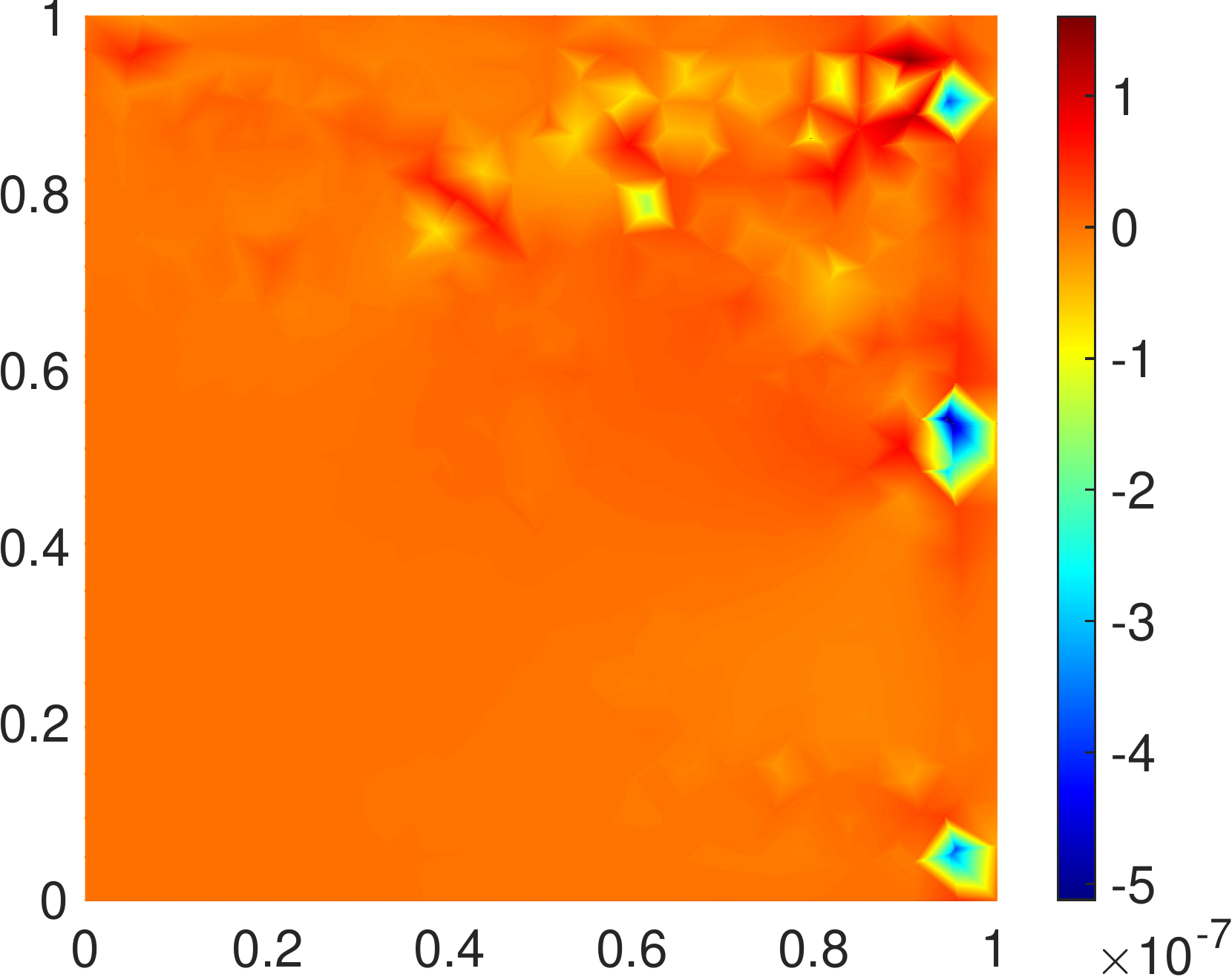}
&\includegraphics[width=0.32\textwidth]{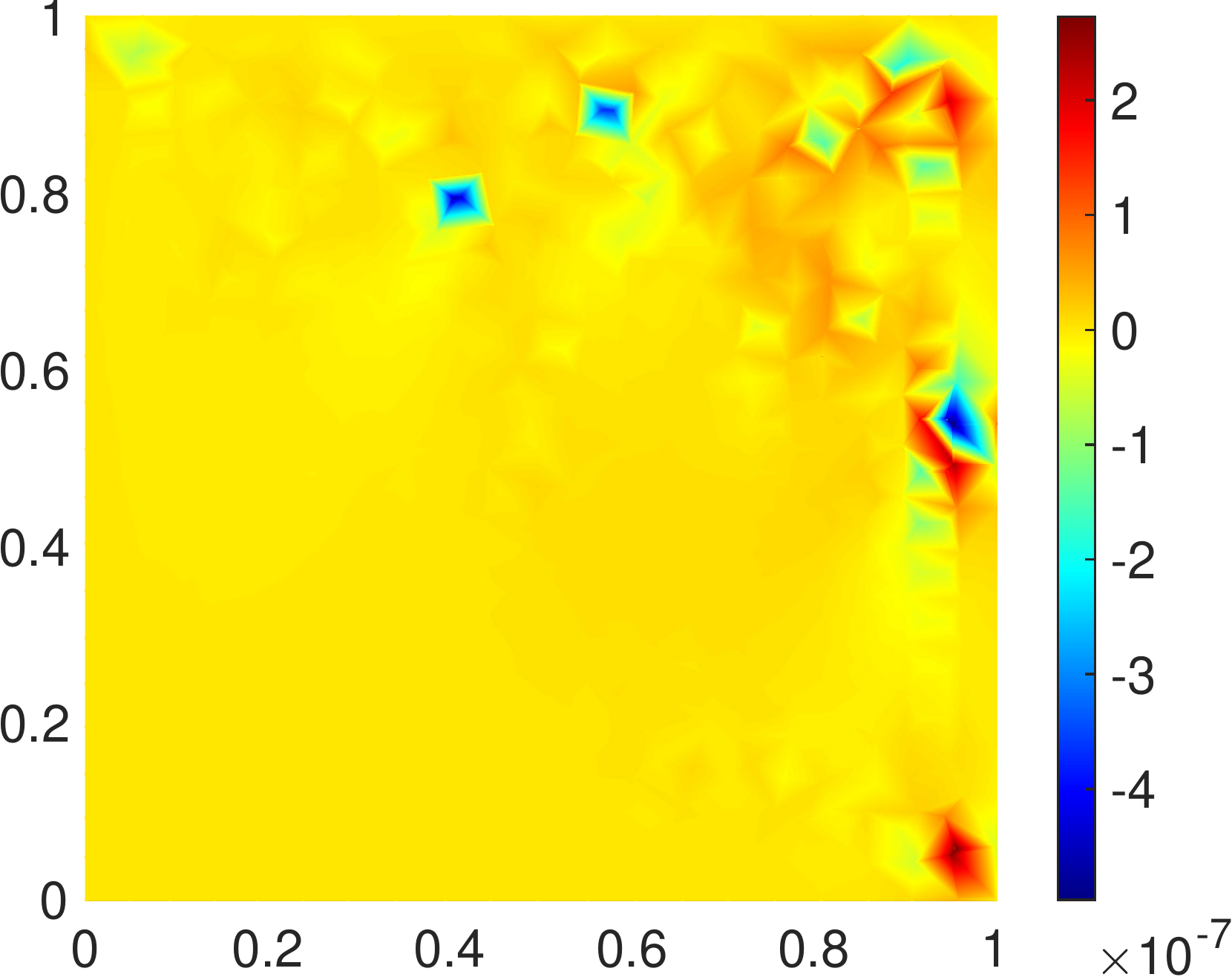}
&\includegraphics[width=0.31\textwidth]{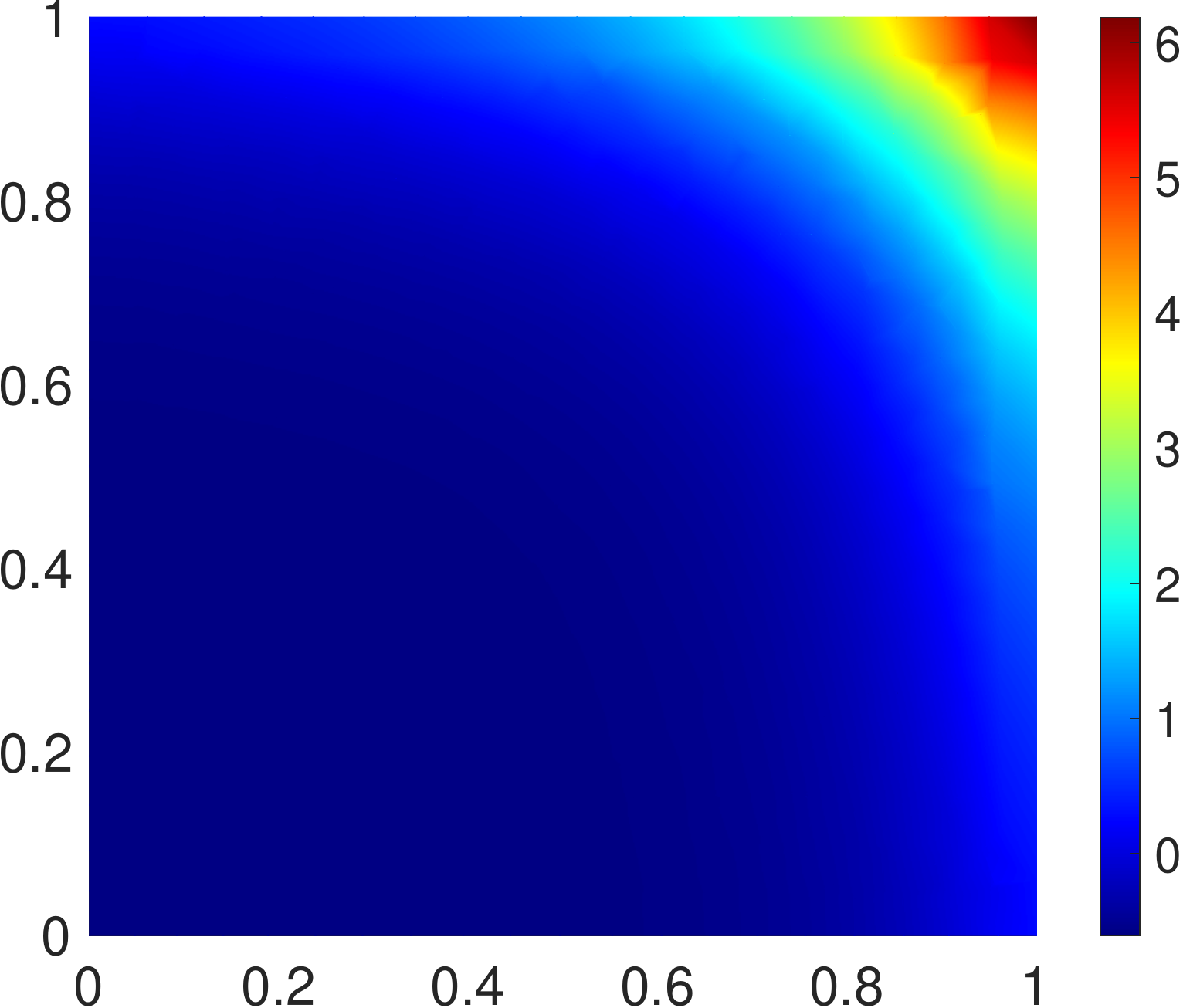}
\\
(a) & (b) & (c)
\end{tabular}
\caption{Test~\ref{Sect:Num-2-2}. Illustration of numerical solutions from Algorithm on the polygonal grid with $\mathbb{CW}_0$\cite{mu}. (a). x-component of velocity; (b) y-component of velocity; (c) pressure.}\label{Num2-2}
\end{figure}

\begin{figure}[H]
\centering
\begin{tabular}{ccc}
\includegraphics[width=0.32\textwidth]{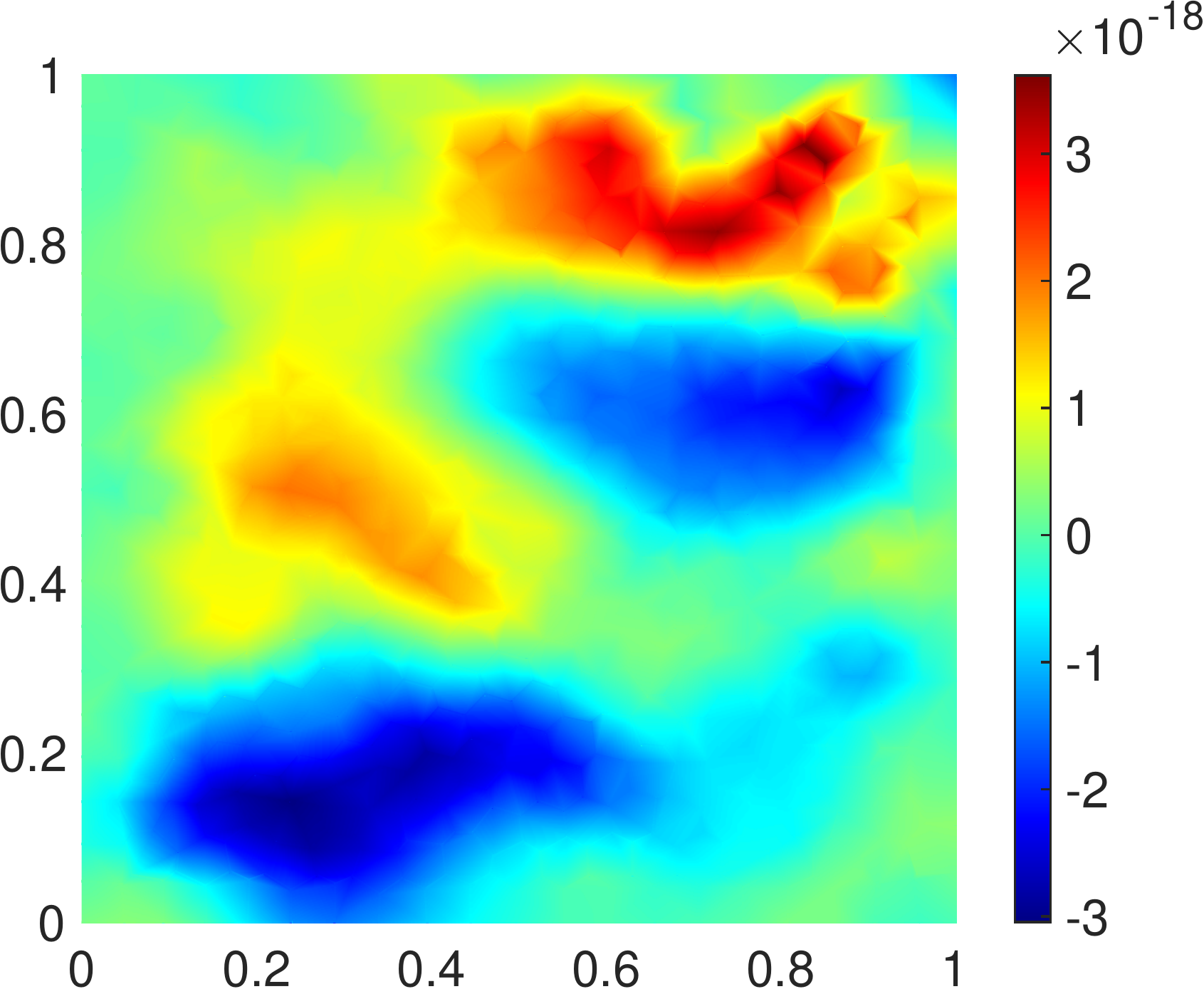}
&\includegraphics[width=0.32\textwidth]{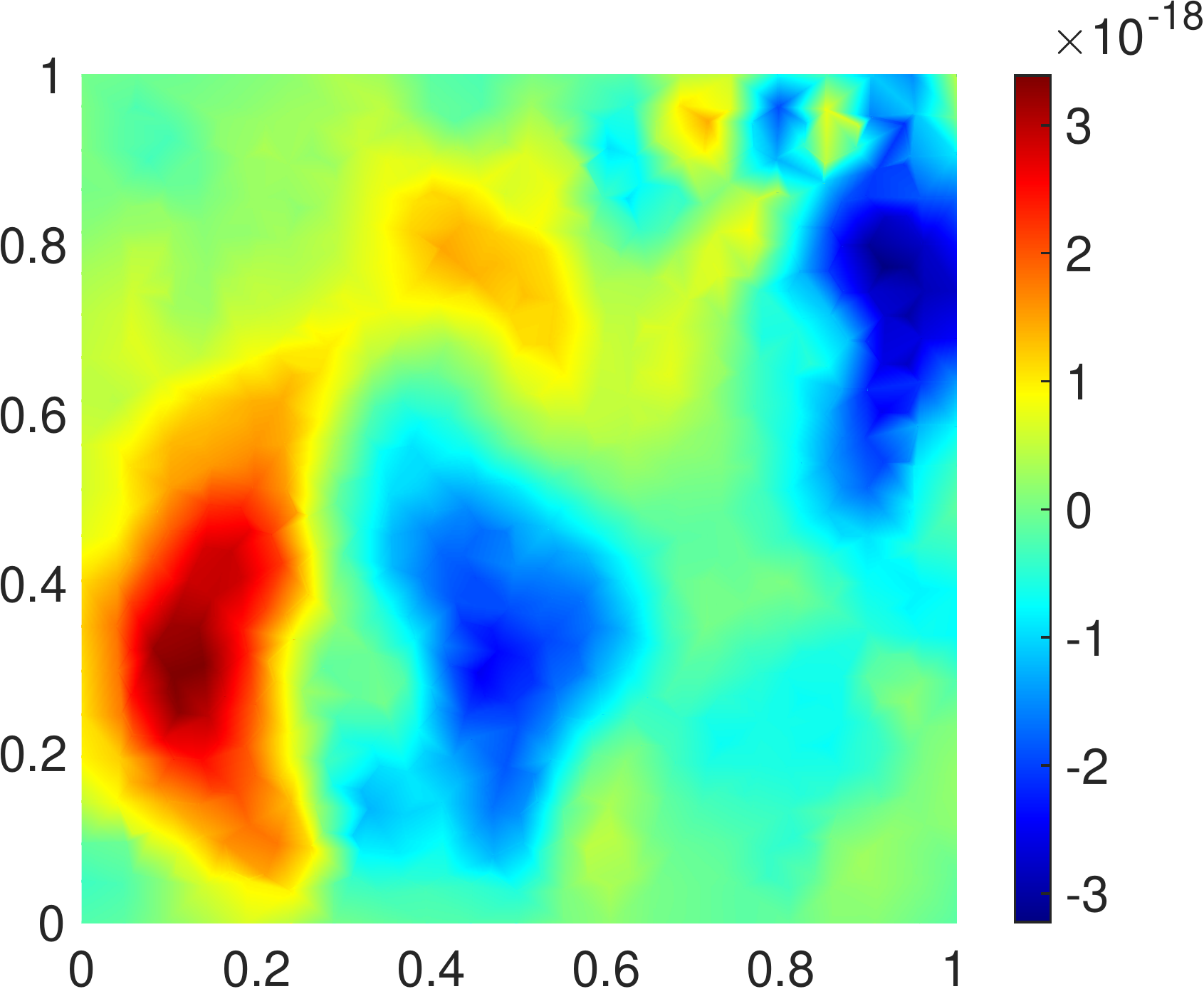}
&\includegraphics[width=0.31\textwidth]{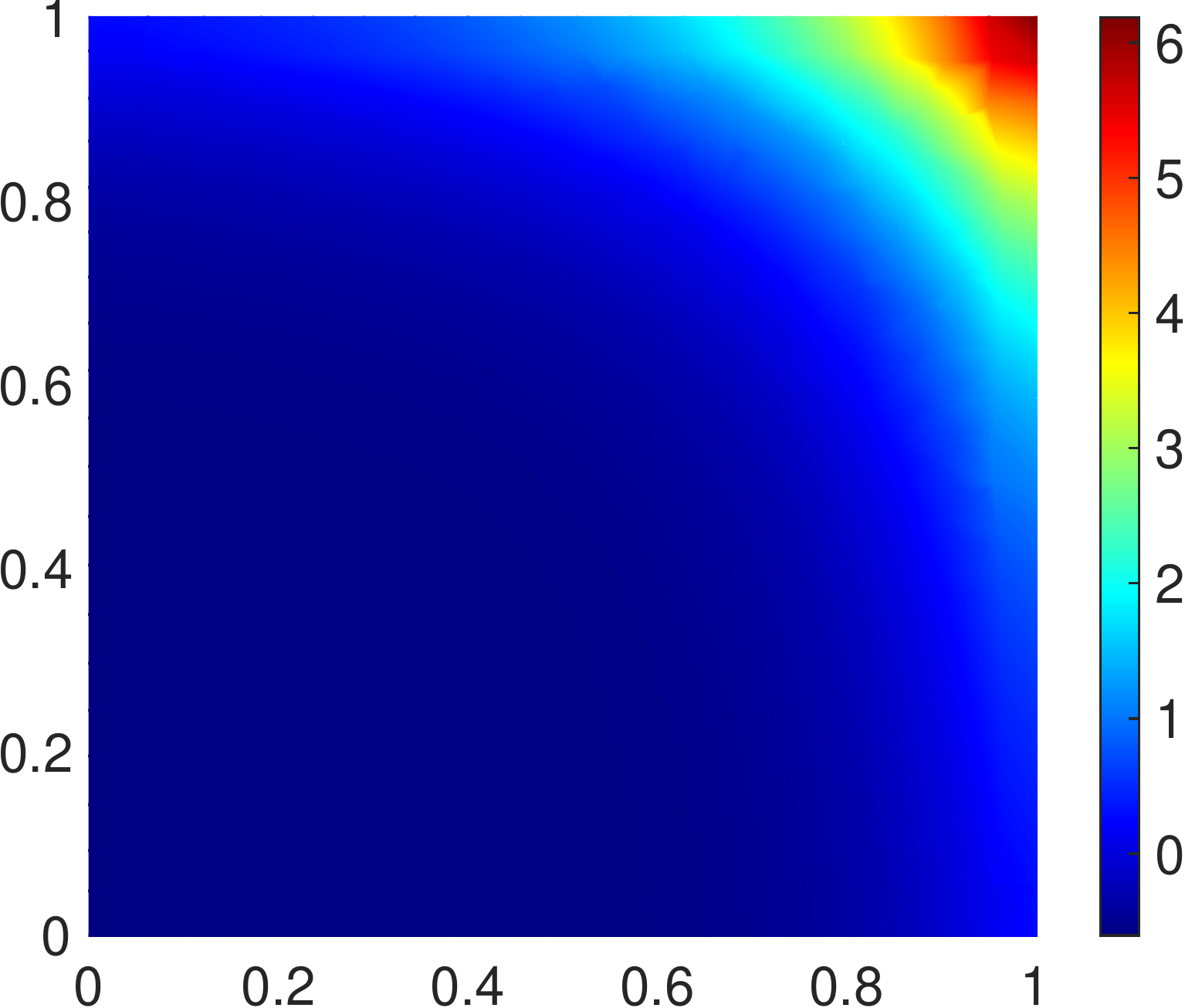}
\\
(a) & (b) & (c)
\end{tabular}
\caption{Test~\ref{Sect:Num-2-2}. Illustration of numerical solutions from Algorithm~\ref{Algorithm:WG1} on the polygonal grid with $\Lambda_0$. (a). x-component of velocity; (b) y-component of velocity; (c) pressure.}\label{Num2-3}
\end{figure}

We perform the simulations by the non-pressure robust and pressure robust Algorithms and the numerical solutions are plotted in Figure~\ref{Num2-1}-\ref{Num2-3}. When the non-pressure robust has been employed, both reconstruction spaces in $\mathbb{CW}_0$ and $\Lambda_0$ version will deliver poor numerical velocity. The velocity plots from $\Lambda_0$ setting are shown in Figure~\ref{Num2-1}a-b and the velocity from $\mathbb{CW}_0$ is very similar and we omit the plotting here. Now we modify the body force simulation by employing the velocity reconstruction in the $\mathbb{CW}_0$ and $\Lambda_0$ spaces to retrieve the pressure robustness. Due to the invariant of pressure, we expect the numerical velocity is close to zero. The results by employing $\mathbb{CW}_0$ space are plotted in Figure~\ref{Num2-2}. As one can observe that on the same level of the mesh, the magnitude of velocity is reduce from the order $\mathcal{O}(10^{-3})$ to $\mathcal{O}(10^{-7})$. The difference from zero flow is due to the inaccuracy for the involving rational functions in the integration for the $\mathbb{CW}_0$ space. In contrast, as employing the velocity reconstruction operator in the proposed $\Lambda_0$ space consisting piecewise polynomials, we overcome the limitation generated by rational functions and retrieve enhanced velocity simulation with the magnitude at the order $\mathcal{O}(10^{-18})$.  This example justified one advantage by our scheme than the work in \cite{mu}.

\subsection{Simulations with Higher Order WG Element}
In this section, we shall perform the WG algorithm with high order WG elements and validate our theoretical conclusions.
\subsubsection{Convergence test - smooth solutions}\label{Num-3-1}
Let $\Omega = (0,1)^2$ and we choose the following analytical solutions for the convergence test:
\begin{eqnarray*}
\bu = \begin{pmatrix}
\sin(\pi x)\sin(\pi y)\\
\cos(\pi x)\cos(\pi y)
\end{pmatrix},\ p(x,y) = 2\cos(\pi x)\sin(\pi y).
\end{eqnarray*}

The performance on a sequence of  deformed rectangular mesh (similar as Figure~\ref{Fig:Num2-grid}a)  is summarized as below.
\begin{itemize}
\begin{table}[H]
\caption{Example~\ref{Num-3-1}. Error profiles and convergence results for $\nu = 1$ on deformed rectangular mesh.}\label{Tab:Ex3-nu1}
\tabcolsep=2pt
\begin{tabular}{c||cc|cc|cc||cc|cc|cc}\hline\hline
&\multicolumn{6}{c||}{Non-Pressure Robust Algorithm~\ref{Algorithm:WG2} }&\multicolumn{6}{c}{Pressure Robust Algorithm~\ref{Algorithm:WG1}	}	\\ \hline
$1/h$&$\3bar\be_h\3bar$& rate &$\|\be_0\|$& rate &$\|\epsilon_h\|$& rate	&$\3bar\be_h\3bar$& rate &$\|\be_0\|$& rate &$\|\epsilon_h\|$& rate	 \\ \hline
&\multicolumn{12}{c}{$k$ = 0}	\\ \hline	
4	&5.01E-1	 &	&3.16E-2	& 	&3.11E-1	&      &1.01	& 	&7.00E-2& 	     &1.15&\\	
8	&2.72E-1	&0.9	&1.03E-2	&1.6	&2.64E-1	&0.2	&5.84E-1	&0.8	&2.36E-2&	1.6&	8.21E-1&	0.5\\
16	&1.41E-1	&0.9	&2.97E-3	&1.8	&1.60E-1	&0.7	&3.10E-1	&0.9	&7.02E-3&	1.7&	4.87E-1&	0.8\\
32	&7.16E-2	&1.0	&7.86E-4	&1.9	&8.53E-2	&0.9	&1.59E-1	&1.0	&1.88E-3&	1.9&	2.62E-1&	0.9\\
64	&3.60E-2	&1.0	&2.00E-4	&2.0	&4.37E-2	&1.0	&8.00E-2	&1.0	&4.82E-4&	2.0&	1.36E-1&	1.0\\ \hline
&\multicolumn{12}{c}{$k$ = 1}	\\ \hline									
4	&8.02E-2	& 	&2.19E-3	& 	&5.79E-2	&	&2.15E-1  &     &6.61E-3&	 	&7.80E-2	\\
8	&2.31E-2	&1.8	&2.70E-4	&3.0	&9.57E-3	&2.6	&5.84E-2	&1.9&8.71E-4&	2.9	&1.06E-2	&2.9\\
16	&6.12E-3	&1.9	&3.35E-5	&3.0	&1.40E-3	&2.8	&1.50E-2	&2.0&1.11E-4&	3.0	&1.35E-3	&3.0\\
32	&1.57E-3	&2.0	&4.20E-6	&3.0	&2.07E-4	&2.8	&3.77E-3	&2.0&1.40E-5&	3.0	&1.70E-4	&3.0\\
64	&3.97E-4	&2.0	&5.29E-7	&3.0	&3.19E-5	&2.7	&9.45E-4	&2.0&1.75E-6&	3.0	&2.14E-5	&3.0\\ \hline
&\multicolumn{12}{c}{$k$ = 2}	\\ \hline											4	&9.03E-3	& 	&1.79E-4	& 	&4.01E-3&           &2.57E-2&	& 	5.10E-4&	 	&5.84E-3	&\\
8	&1.22E-3	&2.9	&1.05E-5	&4.1	&5.56E-4&	2.8&	3.34E-3	&2.9	&3.23E-5	&4.0	&5.56E-4	&3.4\\
16	&1.58E-4	&2.9	&6.44E-7	&4.0	&7.02E-5&	3.0&	4.23E-4	&3.0	&2.04E-6	&4.0	&6.84E-5	&3.0\\
32	&2.01E-5	&3.0	&4.00E-8	&4.0	&8.71E-6&	3.0&	5.32E-5	&3.0	&1.28E-7	&4.0	&8.52E-6	&3.0\\
64	&2.54E-6	&3.0	&2.50E-9	&4.0	&1.08E-6&	3.0&	6.66E-6	&3.0	&8.00E-9	&4.0	&1.06E-6	&3.0\\ \hline
&\multicolumn{12}{c}{$k$ = 3}	\\ \hline												4	&7.96E-4	& 	&1.42E-5	& 	&2.54E-4&    &2.42E-3&	 	&4.23E-5&	 &	6.47E-4&	\\
8	&5.32E-5	&3.9	&4.86E-7	&4.9	&1.28E-5&4.3&	1.59E-4&	3.9	&1.42E-6&	4.9&	3.11E-5&	4.4\\
16	&3.43E-6	&4.0	&1.58E-8	&4.9	&6.69E-7&4.3&	1.01E-5&	4.0	&4.52E-8&	5.0&	1.59E-6&	4.3\\
32	&2.18E-7	&4.0	&5.01E-10&5.0	&3.74E-8&4.2&	6.36E-7&	4.0	&1.42E-9&	5.0&	8.83E-8&	4.2\\
64	&1.37E-8	&4.0	&1.58E-11	&5.0	&2.20E-9&4.1	&3.99E-8&	4.0	&4.45E-11	&5.0&	5.17E-9&	4.1\\ \hline
&\multicolumn{12}{c}{$k$ = 4}	\\ \hline												
4	&5.84E-5	& 	&9.04E-7	& 	&1.27E-5&		&1.86E-4&	 	&2.67E-6&	 	&3.57E-5	&\\
8	&1.88E-6	&5.0	&1.44E-8	&6.0	&4.48E-7&	4.8	&5.91E-6&	5.0	&4.27E-8&	6.0	&9.45E-7	&5.2\\
16	&5.95E-8	&5.0	&2.27E-10	&6.0	&1.42E-8&	5.0	&1.86E-7&	5.0	&6.72E-10&	6.0	&2.79E-8	&5.1\\
32	&1.87E-9	&5.0	&3.56E-12	&6.0	&4.36E-10&	5.0	&5.83E-9&	5.0	&1.05E-11	&     6.0	&8.40E-10	&5.1\\ \hline\hline
\end{tabular}
\end{table}
\item \textbf{Test case with $\nu = 1$ on the deformed rectangular meshes.} First the non-pressure robust Algorithm \ref{Algorithm:WG2} and pressure robust Algorithm \ref{Algorithm:WG1} have been carried out for various WG elements with $k=0,1,2,3,4$. Table~\ref{Tab:Ex3-nu1} reports the error profiles and convergence results. It can be seen from this table that the velocity errors measured in $\3bar\cdot\3bar$-norm and $L^2$-norm converge at the orders $\mathcal{O}(h^{k+1})$ and $\mathcal{O}(h^{k+2})$, which is one order higher than the optimal rate in convergence. The pressure error measured in $L^2$-norm converges at least $\mathcal{O}(h^{k+1})$. The errors from these two algorithms are the same magnitude.

\begin{table}[H]
\caption{Example~\ref{Num-1}. Error profiles and convergence results for $\nu = 1$E-2 on deformed rectangular mesh.}\label{Tab:Ex3-nu1e-2}
\tabcolsep=2pt
\begin{tabular}{c||cc|cc|cc||cc|cc|cc}\hline\hline
&\multicolumn{6}{c||}{Non-Pressure Robust Algorithm~\ref{Algorithm:WG2} }&\multicolumn{6}{c}{Pressure Robust Algorithm~\ref{Algorithm:WG1}	}	\\ \hline
$1/h$&$\3bar\be_h\3bar$& rate &$\|\be_0\|$& rate &$\|\epsilon_h\|$& rate	&$\3bar\be_h\3bar$& rate &$\|\be_0\|$& rate &$\|\epsilon_h\|$& rate	 \\ \hline
&\multicolumn{12}{c}{$k$ = 0}	\\ \hline
4	&3.46E+1&	 	&3.01&	 	&1.84E-1&		&1.01&	 	&7.00E-2&	 	&1.15E-2&	\\
8	&2.06E+1&	0.7	&1.03&	1.6	&1.58E-1&	0.2	&5.84E-1&	0.8	&2.36E-2&	1.6	&8.21E-3&	0.5\\
16	&1.10E+1&	0.9	&2.91E-1&	1.8	&1.26E-1&	0.3	&3.10E-1&	0.9	&7.02E-3&	1.7	&4.87E-3&	0.8\\
32	&5.61&	1.0	&7.65E-2&	1.9	&7.61E-2&	0.7	&1.59E-1&	1.0	&1.88E-3&	1.9	&2.62E-3&	0.9\\
64	&2.83&	1.0	&1.94E-2&	2.0	&4.14E-2&	0.9	&8.00E-2&	1.0	&4.82E-4&	2.0	&1.36E-3&	1.0\\ \hline
	&\multicolumn{12}{c}{$k$ = 1}	\\ \hline											
4	&7.39&	 	&2.05E-1&	 	&4.12E-2&		&2.15E-1&	 	&6.61E-3&	 	&7.80E-4	&\\
8	&2.12&	1.8	&2.47E-2&	3.1	&7.51E-3&	2.5	&5.84E-2&	1.9	&8.71E-4&	2.9	&1.06E-4	&2.9\\
16	&5.57E-1&	1.9	&3.01E-3&	3.0	&1.15E-3&	2.7	&1.50E-2&	2.0	&1.11E-4&	3.0	&1.35E-5	&3.0\\
32	&1.42E-1&	2.0	&3.76E-4&	3.0	&1.69E-4&	2.8	&3.77E-3&	2.0	&1.40E-5&	3.0	&1.70E-6	&3.0\\
64	&3.59E-2&	2.0	&4.72E-5&	3.0	&2.55E-5&	2.7	&9.45E-4&	2.0	&1.75E-6&	3.0	&2.14E-7	&3.0\\ \hline
&\multicolumn{12}{c}{$k$ = 2}	\\ \hline												
4	&8.58E-1&	 	&1.61E-2&	 	&2.68E-3&		&2.57E-2&	 	&5.10E-4&	 	&5.84E-5	&\\
8	&1.14E-1&	2.9	&9.47E-4&	4.1	&4.48E-4&	2.6	&3.34E-3&	2.9	&3.23E-5&	4.0	&5.56E-6	&3.4\\
16	&1.46E-2&	3.0	&5.76E-5&	4.0	&6.00E-5&	2.9	&4.23E-4&	3.0	&2.04E-6&	4.0	&6.84E-7	&3.0\\
32	&1.85E-3&	3.0	&3.57E-6&	4.0	&7.58E-6&	3.0	&5.32E-5&	3.0	&1.28E-7&	4.0	&8.52E-8	&3.0\\
64	&2.33E-4&	3.0	&2.22E-7&	4.0	&9.47E-7&	3.0	&6.66E-6&	3.0	&8.00E-9&	4.0	&1.06E-8	&3.0\\ \hline
&\multicolumn{12}{c}{$k$ = 3}	\\ \hline												
4	&7.64E-2&	 	&1.37E-3&	 	&2.46E-4&		&2.42E-3&	 	&4.23E-5&	 	&6.47E-6	&\\
8	&5.12E-3&	3.9	&4.70E-5&	4.9	&1.20E-5&	4.4	&1.59E-4&	3.9	&1.42E-6&	4.9	&3.11E-7	&4.4\\
16	&3.29E-4&	4.0	&1.52E-6&	4.9	&6.29E-7&	4.2	&1.01E-5&	4.0	&4.52E-8&	5.0	&1.59E-8	&4.3\\
32	&2.08E-5&	4.0	&4.83E-8&	5.0	&3.59E-8&	4.1	&6.36E-7&	4.0	&1.42E-9&	5.0	&8.83E-10	&4.2\\
64	&1.31E-6&	4.0	&1.52E-9&	5.0	&2.14E-9&	4.1	&3.99E-8&	4.0	&4.45E-11	&5.0	&5.17E-11	&4.1\\ \hline
&\multicolumn{12}{c}{$k$ = 4}	\\ \hline												
4	&5.65E-3&	 	&8.69E-5&	 	&1.27E-5&		&1.86E-4&	 	&2.67E-6&	 	&3.57E-7	&\\
8	&1.83E-4&	4.9	&1.40E-6&	6.0	&4.13E-7&	4.9	&5.91E-6&	5.0	&4.27E-8&	6.0	&9.45E-9	&5.2\\
16	&5.79E-6&	5.0	&2.21E-8&	6.0	&1.30E-8&	5.0	&1.86E-7&	5.0	&6.72E-10&	6.0	&2.79E-10	&5.1\\
32	&1.82E-7&	5.0	&3.47E-10&	6.0	&3.99E-10&	5.0	&5.83E-9&	5.0	&1.05E-11	&6.0	&8.47E-12	&5.0
\\ \hline\hline
\end{tabular}
\end{table}
\item\textbf{Test case with $\nu$=1E-2 and 1E-4 on the deformed rectangular meshes.} The numerical performance corresponding to small values in $\nu$ is reported in Table~\ref{Tab:Ex3-nu1e-2}-\ref{Tab:Ex3-nu1e-4}. As one reduces the value in $\nu$, the velocity error produced by non-pressure robust scheme is increasing by a factor $1/\nu$, though preserves the expected convergence rate. In contrast, the velocity error produced by the pressure robust scheme preserve the accuracy and convergence order. Thus our proposed Algorithm~\ref{Algorithm:WG1} shows the significant enhancement in velocity simulation. For the pressure simulation, we observe opposite results. As we decreasing the values in $\nu$, the pressure errors produced by Algorithm~\ref{Algorithm:WG2} remain the same magnitude. However, the pressure errors generated by pressure-robust scheme Algorithm~\ref{Algorithm:WG1} is also reduced by a factor $\nu$.

\begin{table}[H]
\caption{Example~\ref{Num-1}. Error profiles and convergence results for $\nu = 1$E-4 on deformed rectangular mesh.}\label{Tab:Ex3-nu1e-4}
\tabcolsep=2pt
\begin{tabular}{c||cc|cc|cc||cc|cc|cc}\hline\hline
&\multicolumn{6}{c||}{Non-Pressure Robust Algorithm~\ref{Algorithm:WG2} }&\multicolumn{6}{c}{Pressure Robust Algorithm~\ref{Algorithm:WG1}	}	\\ \hline
$1/h$&$\3bar\be_h\3bar$& rate &$\|\be_0\|$& rate &$\|\epsilon_h\|$& rate	&$\3bar\be_h\3bar$& rate &$\|\be_0\|$& rate &$\|\epsilon_h\|$& rate	 \\ \hline
&\multicolumn{12}{c}{$k$ = 0}	\\ \hline
4	&3.46E+3&	 	&3.01E+2&	 	&1.86E-1&		&1.01&	 	&7.00E-2&	 	&1.15E-4&	\\
8	&2.06E+3&	0.7	&1.03E+2&	1.6	&1.57E-1&	0.2	&5.84E-1&	0.8	&2.36E-2&	1.6	&8.21E-5&	0.5\\
16	&1.10E+3&	0.9	&2.91E+1&	1.8	&1.26E-1&	0.3	&3.10E-1&	0.9	&7.02E-3&	1.7	&4.87E-5&	0.8\\
32	&5.61E+2&	1.0	&7.65&	1.9	&7.60E-2&	0.7	&1.59E-1&	1.0	&1.88E-3&	1.9	&2.62E-5&	0.9\\
64	&2.83E+2&	1.0	&1.94&	2.0	&4.13E-2&	0.9	&8.00E-2&	1.0	&4.82E-4&	2.0	&1.36E-5&	1.0\\ \hline
&\multicolumn{12}{c}{$k$ = 1}	\\ \hline								
4	&7.40E+2&	 	&2.05E+1&	 	&4.10E-2&		&2.15E-1&	 	&6.61E-3&	 	&7.80E-6&	\\
8	&2.12E+2&	1.8	&2.47&	3.1	&7.49E-3&	2.5	&5.84E-2&	1.9	&8.71E-4&	2.9	&1.06E-6&	2.9\\
16	&5.57E+1&	1.9	&3.01E-1&	3.0	&1.15E-3&	2.7	&1.50E-2&	2.0	&1.11E-4	&      3.0	&1.35E-7&	3.0\\
32	&1.42E+1&	2.0	&3.76E-2&	3.0	&1.69E-4&	2.8	&3.77E-3&	2.0	&1.40E-5&	3.0	&1.70E-8&	3.0\\
64	&3.59E&	2.0	&4.72E-3&	3.0	&2.55E-5&	2.7	&9.45E-4&	2.0	&1.75E-6&	3.0	&2.14E-9&	3.0\\ \hline
&\multicolumn{12}{c}{$k$ = 2}	\\ \hline												
4	&8.58E+1&	 	&1.61&	 	&2.67E-3&		&2.57E-2&	 	&5.10E-4&	 	&5.84E-7&	 \\
8	&1.14E+1&	2.9	&9.46E-2&	4.1	&4.47E-4&	2.6	&3.34E-3&	2.9	&3.23E-5&	4.0	&5.56E-8&	3.4\\
16	&1.46E&	3.0	&5.76E-3&	4.0	&5.99E-5&	2.9	&4.23E-4&	3.0	&2.04E-6&	4.0	&6.84E-9&	3.0\\
32	&1.85E-1&	3.0	&3.57E-4&	4.0	&7.57E-6&	3.0	&5.32E-5&	3.0	&1.28E-7&	4.0	&8.52E-10&	3.0\\
64	&2.33E-2&	3.0	&2.22E-5&	4.0	&9.45E-7&	3.0	&6.66E-6&	3.0	&8.00E-9&	4.0	&1.06E-10&	3.0\\ \hline
&\multicolumn{12}{c}{$k$ = 3}	\\ \hline												
4	&7.64E&	 	&1.37E-1&	 	&2.46E-4&		&2.42E-3&	 	&4.23E-5&	 	&6.47E-8&	\\
8	&5.12E-1&	3.9	&4.70E-3&	4.9	&1.20E-5&	4.4	&1.59E-4&	3.9	&1.42E-6&	4.9	&3.11E-9	&	4.4\\
16	&3.29E-2&	4.0	&1.52E-4&	4.9	&6.29E-7&	4.2	&1.01E-5&	4.0	&4.52E-8&	5.0	&1.59E-10&	4.3\\
32	&2.08E-3&	4.0	&4.83E-6&	5.0	&3.59E-8&	4.1	&6.36E-7&	4.0	&1.42E-9&	5.0	&8.84E-12&	4.2\\
64	&1.31E-4&	4.0	&1.52E-7&	5.0	&2.14E-9&	4.1	&3.99E-8&	4.0	&4.45E-11	&	5.0	&5.18E-13&	4.1\\ \hline
&\multicolumn{12}{c}{$k$ = 4}	\\ \hline												
4	&5.65E-1&	 	&8.69E-3&	 	&1.27E-5&		&1.86E-4&	 	&2.67E-6&	 	&3.57E-9&	\\
8	&1.83E-2&	4.9	&1.40E-4&	6.0	&4.13E-7&	4.9	&5.91E-6&	5.0	&4.27E-8&	6.0	&9.45E-11	&	5.2\\
16	&5.79E-4&	5.0	&2.21E-6&	6.0	&1.29E-8&	5.0	&1.86E-7&	5.0	&6.72E-10&	6.0	&2.82E-12&	5.1\\
32	&1.82E-5&	5.0	&3.47E-8&	6.0	&3.98E-10&	5.0	&5.83E-9&	5.0	&1.55E-11&	5.4	&1.14E-13&	4.6
\\ \hline\hline
\end{tabular}
\end{table}

\end{itemize}

Next, two algorithms are performed on a sequence of polygonal mesh (Mesh Level 3 is shown in Figure \ref{Fig:Num2-grid}b). Similar numerical conclusions as follows can be obtained. 
\begin{itemize}
\item \textbf{Test case with $\nu = 1$ on polygonal mesh.} All errors generated by Algorithm~\ref{Algorithm:WG1} and Algorithm~\ref{Algorithm:WG2} converge at the expected order. Though some inconsistency error involved in Algorithm~\ref{Algorithm:WG1}, the numerical errors are at comparable magnitude. This observation validate the accuracy of our proposed scheme.
\begin{table}[H] 
\caption{Example~\ref{Num-1}. Error profiles and convergence results for $\nu = 1$ on the polygonal grids.}\label{Tab:Ex3-nu1-case2}
\tabcolsep=1pt
\begin{tabular}{c||cc|cc|cc||cc|cc|cc}\hline\hline
&\multicolumn{6}{c||}{Non-Pressure Robust Algorithm~\ref{Algorithm:WG2} }&\multicolumn{6}{c}{Pressure Robust Algorithm~\ref{Algorithm:WG1}	}	\\ \hline
Mesh&$\3bar\be_h\3bar$& rate &$\|\be_0\|$& rate &$\|\epsilon_h\|$& rate	&$\3bar\be_h\3bar$& rate &$\|\be_0\|$& rate &$\|\epsilon_h\|$& rate	 \\ \hline
&\multicolumn{12}{c}{$k$ = 0}	\\ \hline
Level 1 &6.92E-1&	 	&5.22E-2&	 	&1.55E-1&		&1.27&	 	&9.21E-2&	 	&1.29&	\\
Level 2 &3.61E-1&	0.9	&1.38E-2&	1.9	&1.37E-1&	0.2	&7.93E-1&	0.7	&2.98E-2&	1.6	&4.18E-1	&1.6\\
Level 3 &1.76E-1&	1.0	&3.60E-3&	1.9	&6.40E-2&	1.1	&4.18E-1&	0.9	&9.12E-3&	1.7	&1.41E-1	&1.6\\
Level 4 &8.67E-2&	1.0	&8.66E-4&	2.1	&2.16E-2&	1.6	&2.17E-1&	0.9	&2.48E-3&	1.9	&8.01E-2	&0.8\\
Level 5 &4.33E-2&	1.0	&2.17E-4&	2.0	&7.64E-3&	1.5	&1.10E-1&	1.0	&6.40E-4&	2.0	&1.39E-2	&2.5\\ \hline
&\multicolumn{12}{c}{$k$ = 1}	\\ \hline											
Level 1 &6.21E-2&	 	&1.89E-3&	 	&4.10E-2&		&1.85E-1&	 	&5.86E-3&	 	&6.04E-2&	\\
Level 2 &1.59E-2&	2.0	&2.28E-4&	3.0	&2.88E-3&	3.8	&5.41E-2&	1.8	&5.41E-4&	3.4	&1.09E-2&	2.5\\
Level 3 &3.92E-3&	2.0	&2.53E-5&	3.2	&3.44E-4&	3.1	&1.34E-2&	2.0	&7.03E-5&	2.9	&2.38E-3&	2.2\\
Level 4 &1.02E-3&	1.9	&3.11E-6&	3.0	&8.80E-5&	2.0	&3.44E-3&	2.0	&8.88E-6&	3.0	&8.50E-4&	1.5\\
Level 5 &2.49E-4&	2.0	&3.90E-7&	3.0	&2.03E-5&	2.1	&8.68E-4&	2.0	&1.10E-6&	3.0	&1.45E-4&	2.6\\ \hline
&\multicolumn{12}{c}{$k$ = 2}	\\ \hline											
Level 1 &8.11E-3&	 	&1.52E-4&	 	&4.18E-3&		&2.18E-2&	 	&4.86E-4&	 	&5.96E-3&	\\
Level 2 &8.57E-4&	3.2	&9.25E-6&	4.0	&3.40E-4&	3.6	&2.89E-3&	2.9	&2.53E-5&	4.3	&6.76E-4&	3.1\\
Level 3 &1.06E-4&	3.0	&5.40E-7&	4.1	&2.24E-5&	3.9	&3.60E-4&	3.0	&1.74E-6&	3.9	&8.23E-5&	3.0\\
Level 4 &1.28E-5&	3.1	&3.37E-8&	4.0	&3.73E-6&	2.6	&4.61E-5&	3.0	&1.16E-7&	3.9	&1.23E-5&	2.7\\
Level 5 &1.57E-6&	3.0	&2.09E-9&	4.0	&2.27E-7&	4.0	&5.90E-6&	3.0	&7.49E-9&	4.0	&1.25E-6&	3.3\\ \hline
&\multicolumn{12}{c}{$k$ = 3}	\\ \hline											
Level 1 &6.31E-4&		&1.12E-5&	 	&1.81E-4&		&1.85E-3&	 	&3.30E-5&	 	&5.73E-4&	\\
Level 2 &3.29E-5&	4.3	&2.95E-7&	5.2	&1.05E-5&	4.1	&1.27E-4&	3.9	&9.89E-7&	5.1	&3.03E-5&	4.2\\
Level 3 &1.98E-6&	4.1	&8.50E-9&	5.1	&5.65E-7&	4.2	&7.75E-6&	4.0	&2.95E-8&	5.1	&1.73E-6&	4.1\\
Level 4 &1.23E-7&	4.0	&2.61E-10&	5.0	&1.56E-8&	5.2	&5.00E-7&	4.0	&9.57E-10&	4.9	&1.06E-7&	4.0\\
Level 5 &7.63E-9&	4.0	&7.95E-12&	5.0	&1.27E-9&	3.6	&3.08E-8&	4.0	&2.92E-11&	5.0	&6.42E-9&	4.0
\\ \hline\hline
\end{tabular}
\end{table}
\item \textbf{Test case with $\nu = 1$E-2 and 1E-4 on polygonal mesh.} As in the small viscosity case, the accuracy in Algorithm~\ref{Algorithm:WG2} is destroyed with largely increasing error though still produce the expected convergence order. However, by our enhanced numerical discretization Algorithm~\ref{Algorithm:WG1}, we preserve the accuracy and order in velocity and improve the simulation in pressure along with decreasing the values in $\nu$. This test validate our proposed scheme in the high order computational scheme based on the polygonal mesh.
\begin{table}[H]
\caption{Example~\ref{Num-1}. Error profiles and convergence results for $\nu = 1$E-2 on the polygonal grids.}\label{Tab:Ex3-nu1E-2-case2}
\tabcolsep=1pt
\begin{tabular}{c||cc|cc|cc||cc|cc|cc}\hline\hline
&\multicolumn{6}{c||}{Non-Pressure Robust Algorithm~\ref{Algorithm:WG2} }&\multicolumn{6}{c}{Pressure Robust Algorithm~\ref{Algorithm:WG1}	}	\\ \hline
Mesh&$\3bar\be_h\3bar$& rate &$\|\be_0\|$& rate &$\|\epsilon_h\|$& rate	&$\3bar\be_h\3bar$& rate &$\|\be_0\|$& rate &$\|\epsilon_h\|$& rate	 \\ \hline
&\multicolumn{12}{c}{$k$ = 0}	\\ \hline
Level 1 &3.82E+1&	 	&3.73&	 	&2.21E-1&		&1.27&	 	&9.21E-2&	 	&1.29E-2&	\\
Level 2 &2.20E+1&	0.8	&1.17&	1.7	&1.49E-1&	0.6	&7.93E-1&	0.7	&2.98E-2&	1.6	&4.18E-3&	1.6\\
Level 3 &1.16E+1&	0.9	&3.21E-1&	1.9	&5.72E-2&	1.4	&4.18E-1&	0.9	&9.12E-3&	1.7	&1.41E-3&	1.6\\
Level 4 &5.91&	1.0	&8.37E-2&	1.9	&1.64E-2&	1.8	&2.17E-1&	0.9	&2.48E-3&	1.9	&8.01E-4&	0.8\\
Level 5 &2.98&	1.0	&2.11E-2	&2.0	&3.04E-3&	2.4	&1.10E-1&	1.0	&6.40E-4&	2.0	&1.39E-4&	2.5\\ \hline
&\multicolumn{12}{c}{$k$ = 1}	\\ \hline											
Level 1 &6.37&	 	&1.74E-1&	 	&2.41E-2&		&1.85E-1&	 	&5.86E-3&	 	&6.04E-4&	\\
Level 2 &1.51&	2.1	&1.97E-2&	3.1	&1.61E-3&	3.9	&5.41E-2&	1.8	&5.41E-4&	3.4	&1.09E-4&	2.5\\
Level 3 &3.81E-1&	2.0	&2.26E-3&	3.1	&2.07E-4&	3.0	&1.34E-2&	2.0	&7.03E-5&	2.9	&2.38E-5&	2.2\\
Level 4 &9.46E-2&	2.0	&2.80E-4&	3.0	&3.77E-5&	2.5	&3.44E-3&	2.0	&8.88E-6&	3.0	&8.50E-6&	1.5\\
Level 5 &2.34E-2&	2.0	&3.48E-5&	3.0	&8.58E-6&	2.1	&8.68E-4&	2.0	&1.10E-6&	3.0	&1.45E-6&	2.6\\ \hline
&\multicolumn{12}{c}{$k$ = 2}	\\ \hline											
Level 1 &7.29E-1&	 	&1.41E-2&	 	&2.03E-3&		&2.18E-2&	 	&4.86E-4&	 	&5.96E-5&	\\
Level 2 &7.86E-2&	3.2	&8.48E-4&	4.1	&2.27E-4&	3.2	&2.89E-3&	2.9	&2.53E-5&	4.3	&6.76E-6&	3.1\\
Level 3 &9.82E-3&	3.0	&5.09E-5&	4.1	&1.55E-5&	3.9	&3.60E-4&	3.0	&1.74E-6&	3.9	&8.23E-7&	3.0\\
Level 4 &1.20E-3&	3.0	&3.20E-6&	4.0	&3.77E-6&	2.0	&4.61E-5&	3.0	&1.16E-7&	3.9	&1.23E-7&	2.7\\
Level 5 &1.47E-4&	3.0	&2.01E-7&	4.0	&1.14E-7&	5.0	&5.90E-6&	3.0	&7.49E-9&	4.0	&1.25E-8&	3.3\\ \hline
&\multicolumn{12}{c}{$k$ = 3}	\\ \hline											
Level 1 &6.23E-2&	 	&1.06E-3&	 	&1.67E-4&		&1.85E-3&	 	&3.30E-5&	 	&5.73E-6&	\\
Level 2 &3.22E-3&	4.3	&2.85E-5&	5.2	&6.94E-6&	4.6	&1.27E-4&	3.9	&9.89E-7&	5.1	&3.03E-7&	4.2\\
Level 3 &1.96E-4&	4.0	&8.35E-7&	5.1	&3.80E-7&	4.2	&7.75E-6&	4.0	&2.95E-8&	5.1	&1.73E-8&	4.1\\
Level 4 &1.21E-5&	4.0	&2.57E-8&	5.0	&1.23E-8&	4.9	&5.00E-7&	4.0	&9.57E-10&	4.9	&1.06E-9&	4.0\\
Level 5 &7.34E-7&	4.0	&7.67E-10&	5.1	&7.20E-10&	4.1	&3.08E-8&	4.0	&2.92E-11&	5.0	&6.42E-11	&4.0
\\ \hline\hline
\end{tabular}
\end{table}
\begin{table}[H]
\caption{Example~\ref{Num-1}. Error profiles and convergence results for $\nu = 1$E-4 on the polygonal grids.}\label{Tab:Ex3-nu1E-4-case2}
\tabcolsep=2pt
\begin{tabular}{c||cc|cc|cc||cc|cc|cc}\hline\hline
&\multicolumn{6}{c||}{Non-Pressure Robust Algorithm~\ref{Algorithm:WG2} }&\multicolumn{6}{c}{Pressure Robust Algorithm~\ref{Algorithm:WG1}	}	\\ \hline
Mesh&$\3bar\be_h\3bar$& rate &$\|\be_0\|$& rate &$\|\epsilon_h\|$& rate	&$\3bar\be_h\3bar$& rate &$\|\be_0\|$& rate &$\|\epsilon_h\|$& rate	 \\ \hline
&\multicolumn{12}{c}{$k$ = 0}	\\ \hline
Level 1	&3.81E+3&	 	&3.71E+2&	 	&2.22E-1&		&1.27&	 	&9.21E-2&	 	&1.29E-4&\\	
Level 2	&2.19E+3&	0.8	&1.16E+2&	1.7	&1.49E-1&	0.6	&7.93E-1&	0.7	&2.98E-2&	1.6	&4.18E-5&	1.6\\
Level 3	&1.15E+3&	0.9	&3.20E+1&	1.9	&5.72E-2&	1.4	&4.18E-1&	0.9	&9.12E-3&	1.7	&1.41E-5&	1.6\\
Level 4	&5.91E+2&	1.0	&8.37&	1.9	&1.64E-2&	1.8	&2.17E-1&	0.9	&2.48E-3&	1.9	&8.01E-6&	0.8\\
Level 5	&2.97E+2&	1.0	&2.11&	2.0	&3.05E-3&	2.4	&1.10E-1&	1.0	&6.40E-4&	2.0	&1.39E-6&	2.5\\ \hline
&\multicolumn{12}{c}{$k$ = 1}	\\ \hline												
Level 1	&6.39E+2&	 	&1.74E+1&	 	&2.40E-2&		&1.85E-1&	 	&5.86E-3&	 	&6.04E-6&	\\
Level 2	&1.52E+2&	2.1	&1.97&	3.1	&1.61E-3&	3.9	&5.41E-2&	1.8	&5.41E-4&	3.4	&1.09E-6&	2.5\\
Level 3	&3.81E+1&	2.0	&2.26E-1&	3.1	&2.08E-4&	3.0	&1.34E-2&	2.0	&7.03E-5&	2.9	&2.38E-7&	2.2\\
Level 4	&9.47&	2.0	&2.80E-2&	3.0	&3.78E-5&	2.5	&3.44E-3&	2.0	&8.88E-6&	3.0	&8.50E-8&	1.5\\
Level 5	&2.34&	2.0	&3.48E-3&	3.0	&8.58E-6&	2.1	&8.68E-4&	2.0	&1.10E-6&	3.0	&1.45E-8&	2.6\\ \hline
&\multicolumn{12}{c}{$k$ = 2}	\\ \hline												
Level 1	&7.28E+1&	 	&1.41&	 	&2.01E-3&		&2.18E-2&	 	&4.86E-4&	 	&5.96E-7&	\\
Level 2	&7.86&	3.2	&8.48E-2&	4.1	&2.26E-4&	3.2	&2.89E-3&	2.9	&2.53E-5&	4.3	&6.76E-8&	3.1\\
Level 3	&9.82E-1&	3.0	&5.09E-3&	4.1	&1.54E-5&	3.9	&3.60E-4&	3.0	&1.74E-6&	3.9	&8.23E-9&	3.0\\
Level 4	&1.20E-1&	3.0	&3.20E-4&	4.0	&3.77E-6&	2.0	&4.61E-5&	3.0	&1.16E-7&	3.9	&1.23E-9&	2.7\\
Level 5	&1.47E-2&	3.0	&2.01E-5&	4.0	&1.13E-7&	5.1	&5.90E-6&	3.0	&7.49E-9&	4.0	&1.25E-10&	3.3\\ \hline
&\multicolumn{12}{c}{$k$ = 3}	\\ \hline												
Level 1	&6.23&	 	&1.06E-1&	 	&1.67E-4&		&1.85E-3&	 	&3.30E-5&	 	&5.73E-8&	\\
Level 2	&3.22E-1&	4.3	&2.85E-3&	5.2	&6.91E-6&	4.6	&1.27E-4&	3.9	&9.89E-7&	5.1	&3.03E-9&	4.2\\
Level 3	&1.96E-2&	4.0	&8.35E-5&	5.1	&3.78E-7&	4.2	&7.75E-6&	4.0	&2.95E-8&	5.1	&1.73E-10&	4.1\\
Level 4	&1.22E-3&	4.0	&2.57E-6&	5.0	&1.23E-8&	4.9	&5.00E-7&	4.0	&9.57E-10&	4.9	&1.06E-11	&     4.0
\\ \hline\hline
\end{tabular}
\end{table}
\end{itemize}

\subsubsection{Test - Low regularity in pressure}\label{Sect:Num-4}
In this test, let the Lshape domain be $\Omega = [-1,1]^2\backslash [0,1]\times[-1,0]$ and the exact solutions be chosen as follows:
\begin{eqnarray*}
\bu = \begin{pmatrix}
\sin(\pi x)\sin(\pi y),\\
\cos(\pi x)\cos(\pi y)
\end{pmatrix},\ p = r^{2/3}\sin(\frac{2\theta}{3}),
\end{eqnarray*}
where $r,\theta$ are in the polar coordinates. It is known that the velocity is smooth and the regularity of pressure is approximately $H^{1.67}$.

In this example, we shall employ the polygonal meshes to validate the numerical performance and the enhancement of pressure robustness. A sequence of meshes have been employed in the simulation. Figure~\ref{Num4-Grids} shows the first two levels of the meshes. The numerical experiments have been carried out based on the WG Algorithm \ref{Algorithm:WG1} (denoted as WG1) and WG Algorithm \ref{Algorithm:WG2} (denoted as WG2) for various polynomial degrees. 
Figure~\ref{Num4-pL2} plots the error profiles and convergence results for pressure. It is noted that by Algorithm~\ref{Algorithm:WG2}, the error, measured in $\|\epsilon_h\|$, converges at the order $\mathcal{O}(h^{\min(k+1,2)})$, which is limited by the regularity of pressure. However, we can break such limitation induced by low pressure regularity suggested as Algorithm~\ref{Algorithm:WG1}. The numerical pressure error produced by Algorithm~\ref{Algorithm:WG1} converges to 0 at the order $\mathcal{O}(h^{k+1})$.

\begin{figure}[H]
\centering
\begin{tabular}{cc}
\includegraphics[width=0.4\textwidth,height=.4\textwidth]{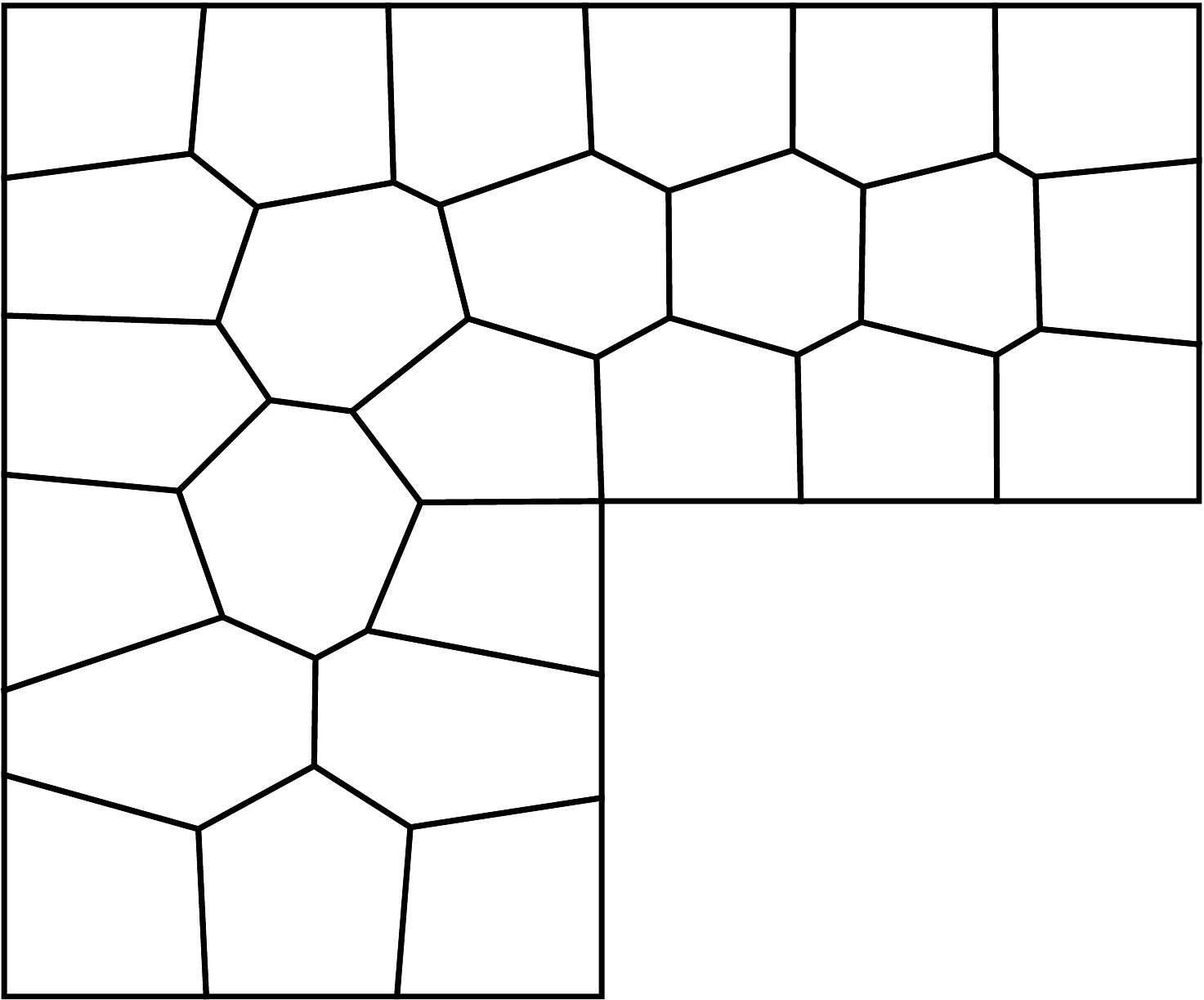}
&\includegraphics[width=0.4\textwidth,height=.4\textwidth]{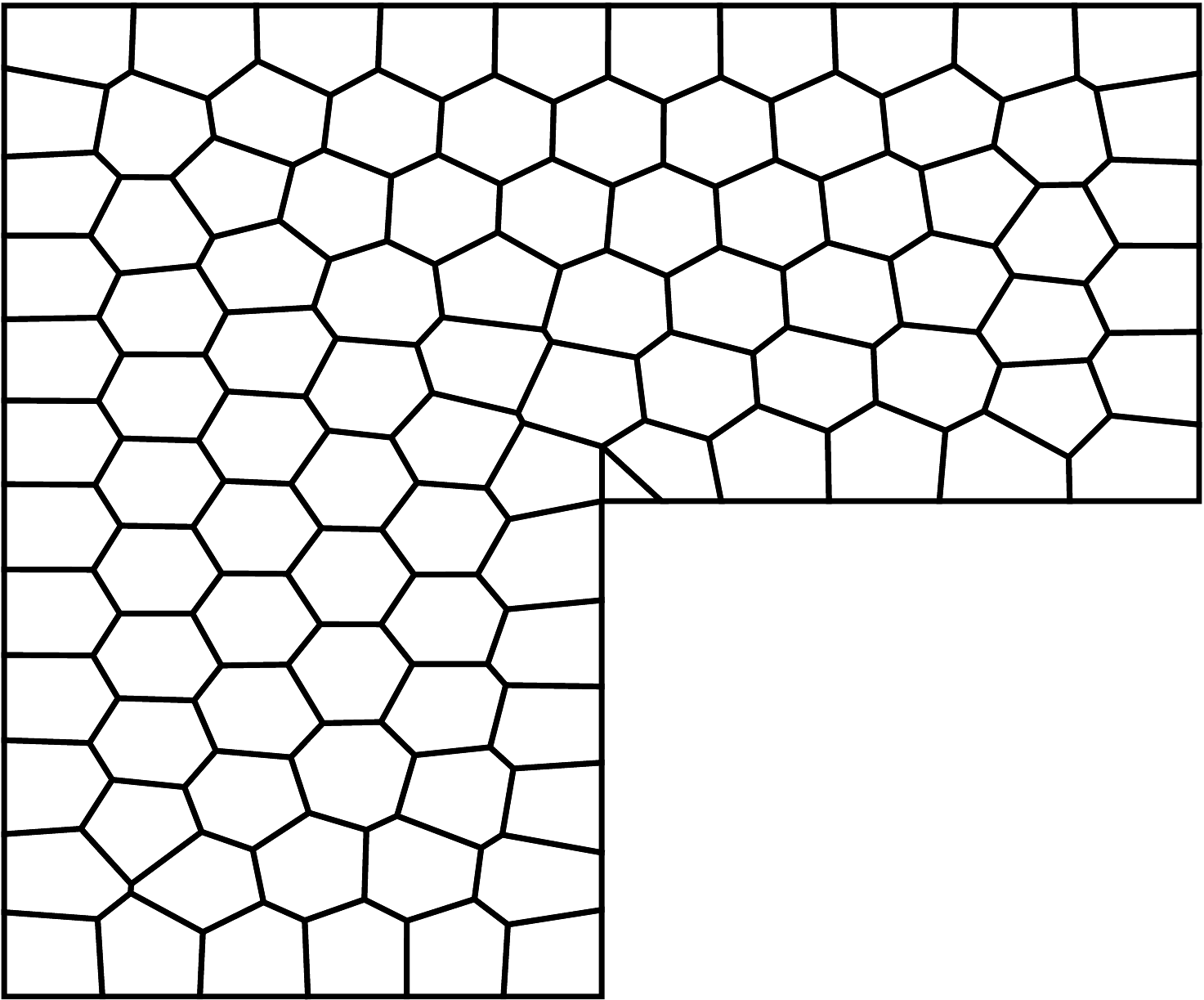}
\\
(a) & (b) 
\end{tabular}
\caption{Test~\ref{Sect:Num-4}. Illustration of computational polygonal grids: (a) Level 1; (b) Level 2.}\label{Num4-Grids}
\end{figure}

\begin{figure}[H]
\centering
\includegraphics[width=0.7\textwidth]{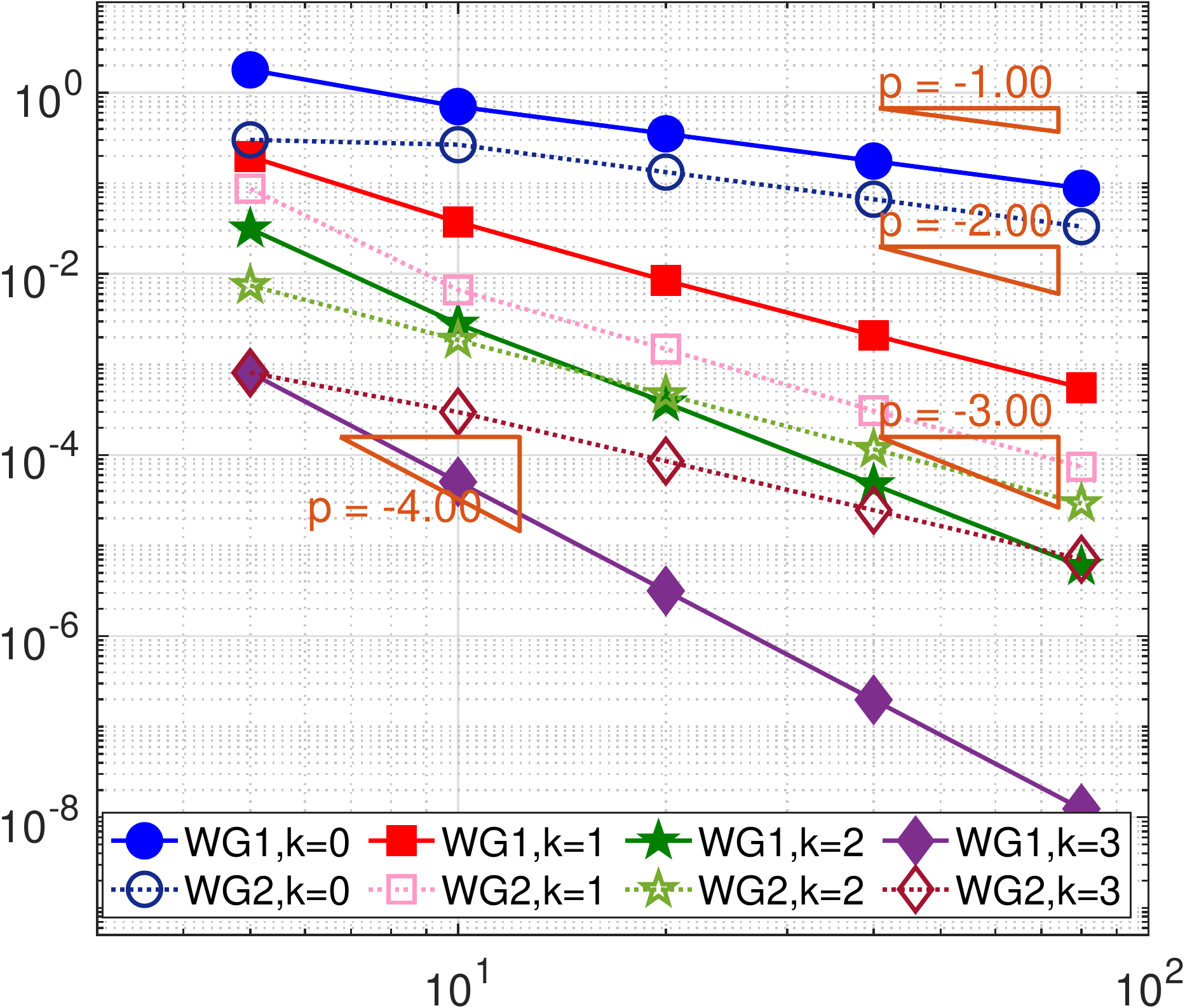}
\caption{Test~\ref{Sect:Num-4}. Convergence test for pressure error in $\|\mathcal{Q}_hp-p_h\|$-norm.}\label{Num4-pL2}
\end{figure}

The velocity error profiles and convergence results are demonstrated in Figure~\ref{Num4-uErr}. The solid lines plot the velocity error by Algorithm~\ref{Algorithm:WG1}, measured in $\3bar\be_h\3bar$ and $\|\be_h\|$, respectively. As one can see from these two figures that the convergence rates are $\mathcal{O}(h^{k+1})$ and $\mathcal{O}(h^k)$ though we have low regularity in pressure. However, the convergence rates produced by Algorithm~\ref{Algorithm:WG2} are limited by the regularity in pressure and the velocity errors converge at the order $\mathcal{O}(h^{\min(k+1,2)})$ and $\mathcal{O}(h^{\min(k+2,3)}$, respectively. Again, the results confirm the invariant error estimate with respect to pressure. This test shows that pressure-robustness plays a essential role to develop the arbitrary high order numerical scheme in the velocity simulation
\begin{figure}[H]
\centering
\begin{tabular}{cc}
\includegraphics[width=0.47\textwidth]{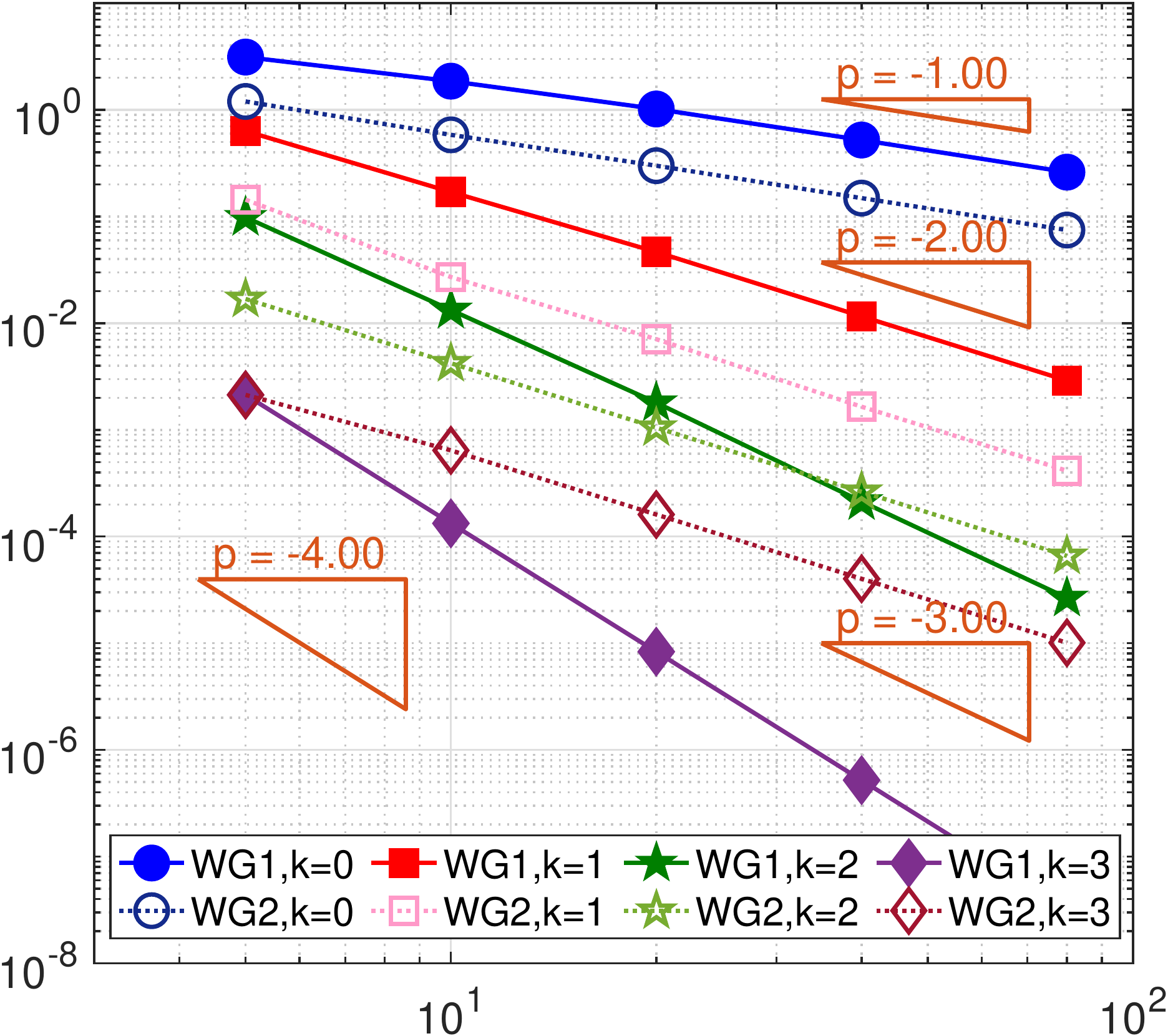}
&\includegraphics[width=0.47\textwidth]{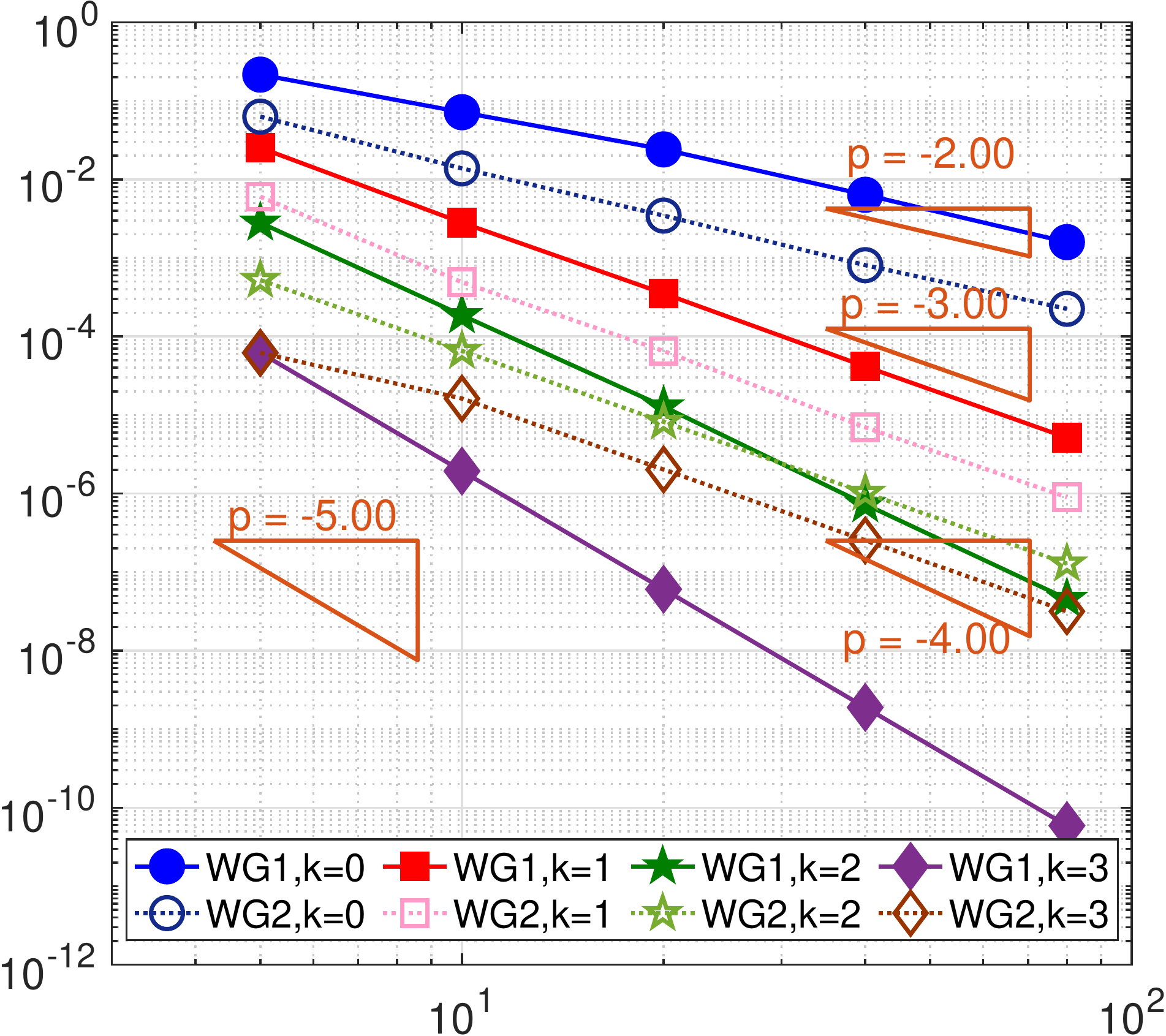}
\\
(a) & (b) 
\end{tabular}
\caption{Test~\ref{Sect:Num-4}. Convergence test for: (a) velocity error in $\3bar\cdot\3bar$-norm; (b) velocity error in $L^2$-norm.}\label{Num4-uErr}
\end{figure}

\section{Conclusion Remark}\label{Sect:Conclusion}
In this paper, we have developed a new stabilizer free and pressure-robust weak Galerkin scheme for solving Stokes equation on polygonal mesh. The new method modifies the right hand side assembling by projecting the test function into the H(div)-conforming space $\Lambda_k(T)$. This new method shows features of divergence preserving and robustness on viscosity variable $\nu$. Besides, the convergence test shows that one order superconvergence is obtained than the optimal order.
 the numerical performance can be significantly improved by using such techniques. 
Finally, the design of the reconstruction operator can be transferred into other polygonal finite element methods including HHO, HDG, and VEM methods.

The results in this paper can be applied to solve other  incompressible flow, including Navier-Stokes equations and Brinkman equations,  when the mass conservation is desired in the discretization. We shall include such applications in the future work plan.

\end{document}